\documentclass[12pt]{amsart}

\usepackage{amsfonts, amstext, amsmath, amsthm, amscd, amssymb, enumitem, url}
\usepackage{mathabx}
\usepackage{tikz-cd}
\usetikzlibrary{shapes.geometric, arrows}
\usepackage{xcolor}
\usepackage{svg}
\usepackage{pdfpages}
\usepackage{pdflscape}
\usepackage{spverbatim}

\usepackage{microtype}
\usepackage[parfill]{parskip}

\usepackage[colorlinks,citecolor=cyan,linkcolor=magenta]{hyperref}
\usepackage[nameinlink]{cleveref}

\usepackage{array}
\newcolumntype{C}[1]{>{\centering\let\newline\\\arraybackslash\hspace{0pt}}m{#1}}
\usepackage{multicol}
\usepackage{multirow}
\usepackage{makecell}
\usepackage{booktabs}

\newtheorem{thm}{Theorem}[section]
\newtheorem{cor}[thm]{Corollary}
\newtheorem{lemma}[thm]{Lemma}
\newtheorem{prop}[thm]{Proposition}

\newtheorem{conj}[thm]{Conjecture}

\newtheorem*{thm:pennergenusone}{\Cref{thm:pennergenusone}}
\newtheorem*{thm:totallyperiodicgenusone}{\Cref{thm:totallyperiodicgenusone}}
\newtheorem*{thm:horsuralmostequiv}{\Cref{thm:horsuralmostequiv}}

\theoremstyle{definition}
\newtheorem{defn}[thm]{Definition}
\newtheorem{rmk}[thm]{Remark}
\newtheorem{constr}[thm]{Construction}

\numberwithin{equation}{section}

\usepackage{todonotes}

\renewcommand{\epsilon}{\varepsilon}

\usepackage{stmaryrd}
\newcommand{\cut}{\!\bbslash\!}

\newcommand{\sing}{\mathrm{sing}}
\newcommand{\slope}{\mathrm{slope}}

\newcommand{\brloc}{\mathrm{brloc}}

\begin{document}

\title{Examples of Anosov flows with genus one Birkhoff sections}

\author{Chi Cheuk Tsang}
\address{Département de mathématiques \\
Université du Québec à Montréal \\
201 President Kennedy Avenue \\
Montréal, QC, Canada H2X 3Y7}
\email{tsang.chi\_cheuk@uqam.ca}

\maketitle

\begin{abstract}
We show that a transitive Anosov flow with orientable stable and unstable foliations that either (i) admits a Birkhoff section whose first return map is a Penner type pseudo-Anosov map, or (ii) is totally periodic admits a genus one Birkhoff section.
This provides evidence for a conjecture of Fried and Ghys.
The proof utilizes a result of the author on the horizontal Goodman surgery operation. To apply this result for showing (i), we establish correspondence between horizontal Goodman surgery on pseudo-Anosov flows and horizontal surgery on veering triangulations in the layered setting.
\end{abstract}

% \tableofcontents

\section{Introduction} \label{sec:intro}

A \textbf{Birkhoff section} to a flow $\phi^t$ on a closed $3$-manifold is a cooriented surface $S$ that is positively transverse to $\phi^t$ in its interior and tangent to $\phi^t$ along its boundary, such that every orbit of $\phi^t$ intersects $S$ in finite forward and backward time. Given a Birkhoff section $S$, one can reduce questions about the $3$-dimensional dynamics of the flow $\phi^t$ into questions about the $2$-dimensional dynamics of the first return map on $S$.

When $\phi^t$ is a transitive Anosov flow, it is a classical result of Fried \cite{Fri83} that Birkhoff sections always exist. In this case, the first return map is a pseudo-Anosov map, for which there exists a wide selection of tools for studying its dynamics, see for example \cite{FLP79}, \cite{McM00}, \cite{Yoc10}, \cite{Ago11}.
In practice, however, it is often preferable to work with Birkhoff sections that are as simple as possible. Among other reasons, this ensures that one can apply these tools with a manageable number of computational parameters. 

A convenient criterion for simplicity in this context is to ask for the Birkhoff section to have genus one. This is because if in addition the stable and unstable foliations of $\phi^t$ are orientable, then the first return map will be an Anosov map on a punctured torus. 
The action of such a map can be described by a matrix in $\mathrm{SL}_2 \mathbb{Z}$. In particular one can understand its dynamics by very simple algebraic computations.

A conjecture of Fried and Ghys states that this ideal scenario is always realized.

\begin{conj}[Fried, Ghys] \label{conj:genusonesections}
Let $\phi^t$ be a transitive Anosov flow with orientable stable and unstable foliations. Then $\phi^t$ admits a genus one Birkhoff section.
\end{conj}

It is known that \Cref{conj:genusonesections} is true for geodesic flows of negatively curved orbifolds, by work of Dehornoy and Shannon.

\begin{thm}[Dehornoy-Shannon \cite{DS19}]
Let $\phi^t$ be the geodesic flow of a negatively curved orbifold. Then $\phi^t$ admits a genus one Birkhoff section.
\end{thm}

The goal of this paper is to provide further evidence towards \Cref{conj:genusonesections} by demonstrating that two other classes of Anosov flows always admit genus one Birkhoff sections.

\begin{thm} \label{thm:pennergenusone}
Suppose $\phi^t$ is an Anosov flow with orientable stable and unstable foliations. If $\phi^t$ admits a Birkhoff section whose first return map is a Penner type pseudo-Anosov map, then $\phi^t$ admits a genus one Birkhoff section.
\end{thm}

\begin{thm} \label{thm:totallyperiodicgenusone}
Every transitive totally periodic Anosov flow with orientable stable and unstable foliations admits a genus one Birkhoff section.
\end{thm}

Here, a \textbf{Penner type pseudo-Anosov map} is a pseudo-Anosov map of the form $\sigma \tau^{n_1}_\alpha \tau^{-m_1}_\beta \dots \tau^{n_k}_\alpha \tau^{-m_k}_\beta$, where $\alpha$ and $\beta$ are a filling pair of multicurves, $\sigma$ is a homeomorphism preserving $\alpha$ and $\beta$, each $n_s$ and $m_s$ is a tuple of non-negative integers, and $\tau^{n_s}_\alpha$ and $\tau^{-m_s}_\beta$ denote suitable powers of Dehn twists along components of $\alpha$ and $\beta$, respectively. We refer to \Cref{sec:pennerpa} for more details. These maps were introduced by Penner in \cite{Pen88}, and are well-studied in the literature, see for example \cite{Lei04} and \cite{SS15}.

Meanwhile, a \textbf{totally periodic Anosov flow} is an Anosov flow on a graph manifold such that the regular fiber of each Seifert fibered piece is homotopic to a closed orbit of the flow. These were introduced by Barbot and Fenley in \cite{BF15}, and occupy one end of the spectrum in the classification of (pseudo-)Anosov flows on graph manifolds. We refer to \cite{BF13}, \cite{BF15}, and \cite{BF21} for more details on this classification.

The proofs of \Cref{thm:pennergenusone} and \Cref{thm:totallyperiodicgenusone} utilize results from \cite{Tsa24} concerning the horizontal Goodman surgery operation on Anosov flows. For the rest of this introduction, we will recall these results and provide outlines of proofs for \Cref{thm:pennergenusone} and \Cref{thm:totallyperiodicgenusone}. Finally, we will say a few words about the history and our statement of \Cref{conj:genusonesections}, then we will provide an outline of this paper.

\subsection{Almost equivalence and horizontal Goodman surgery}

We say that two flows are \textbf{almost equivalent} if they are orbit equivalent after drilling out finitely many closed orbits. It is straightforward to check that almost equivalence is an equivalence relation.

If an Anosov flow $\phi^t$ is almost equivalent to the suspension flow of an Anosov map on the torus, then we can transfer a fiber surface of the latter into a genus one Birkhoff section of $\phi^t$. Our method for showing \Cref{thm:pennergenusone} and \Cref{thm:totallyperiodicgenusone} is thus to demonstrate a chain of almost equivalences relating a given Anosov flow to such a suspension flow.

The tool we use to generate these almost equivalences is the \textbf{horizontal Goodman surgery operation}. This operation takes in an Anosov flow $\phi^t$, a positive/negative horizontal surgery curve $c$ (\Cref{defn:horsurcurve}), and a positive/negative integer $n$, respectively, and outputs an Anosov flow $\phi^t_{\frac{1}{n}}(c)$ obtained by thickening $c$ into a transverse annulus $A$, cutting along $A$ then regluing by a Dehn twist map with coefficient $n$.

This operation is a generalization of work of Goodman in \cite{Goo83}, and has been folklore knowledge among experts. A careful exposition of this operation has recently been given in \cite{Tsa24}, where the following result concerning the relation between horizontal Goodman surgery and almost equivalence is also shown.

\begin{thm}[{\cite[Theorem 1.3]{Tsa24}}] \label{thm:horsuralmostequiv}
Let $\phi^t$ be a transitive Anosov flow on a closed oriented $3$-manifold $M$. Let $c$ be a positive/negative horizontal surgery curve for $\phi^t$. Then for every positive/negative integer $n$, respectively, the flow $\phi^t_{\frac{1}{n}}(c)$ is almost equivalent to $\phi^t$.
\end{thm}

\Cref{thm:horsuralmostequiv} allows us to show \Cref{thm:pennergenusone} and \Cref{thm:totallyperiodicgenusone} by demonstrating a sequence of horizontal Goodman surgeries relating a given Anosov flow to an Anosov suspension flow. 
The strategy for finding this sequence of surgeries differs for the two theorems.

\subsection{Sketch of proof of \Cref{thm:pennergenusone}} \label{subsec:intropennergenusoneproof}

For \Cref{thm:pennergenusone}, we locate a suitable sequence of horizontal Goodman surgeries using the tool of veering triangulations. 
A \textbf{veering triangulation} is an ideal triangulation of a $3$-manifold along with certain combinatorial data (\Cref{defn:vt}). Recent work of Agol-Gueritaud, Schleimer-Segerman, and Landry-Minsky-Taylor shows that there is a correspondence between veering triangulations and pseudo-Anosov flows.

In \cite{Tsa22a}, we introduced a \textbf{horizontal surgery operation} on veering triangulations and in \cite{Tsa24} we conjectured that the horizontal Goodman surgery operation on pseudo-Anosov flows corresponds to this horizontal surgery operation on veering triangulations. In this paper, we verify this conjecture for layered veering triangulations.

\begin{thm} \label{thm:layeredvthsur}
Let $M$ be a closed oriented $3$-manifold and let $\mathcal{C}$ be a finite collection of curves in $M$. Let $\Delta$ be a layered veering triangulation on $M \backslash \mathcal{C}$. Let $\phi^t$ be the pseudo-Anosov flow on $M$ corresponding to $\Delta$. Then for every positive/negative horizontal surgery curve $c$ on the veering branched surface $B$ dual to $\Delta$, there is an isotopic positive/negative horizontal surgery curve $c'$ of the flow $\phi^t$. Moreover, for every positive/negative integer $n$, the flow $\phi^t_{\frac{1}{n}}(c')$ corresponds to the veering triangulation $\Delta_{\frac{1}{n}}(c)$ on the $3$-manifold $M_{\frac{1}{n}}(c)$. 
\end{thm}

We will deduce \Cref{thm:layeredvthsur} from the more general \Cref{thm:hsurcorr}.

The utility of \Cref{thm:layeredvthsur} towards showing \Cref{thm:pennergenusone} comes from the fact the veering triangulation that corresponds to the suspension flow of a Penner type pseudo-Anosov map $\sigma \tau^{n_1}_\alpha \tau^{-m_1}_\beta \dots \tau^{n_k}_\alpha \tau^{-m_k}_\beta$ admits a simple description. This makes it possible for us to specify a sequence of horizontal surgery operations that relates this triangulation to a triangulation corresponding to an Anosov suspension flow. 

In more detail, we first perform surgery operations to reduce to the case when $k=1$. This uses the fact that composing the monodromy of a suspension flow by a Dehn twist can be effected by performing Dehn surgery along a curve lying on a fiber surface. Here it is instructive to note that we cannot reduce to the case $k=0$ with the same reasoning, essentially because the suspension flow of $\sigma$ is not pseudo-Anosov. What is possible however, is that one can use a generalization of veering triangulations, known as \textbf{almost veering branched surfaces}, to combinatorially encode the suspension flow of $\sigma$. 

We then observe that $\sigma \tau^{n_1}_\alpha \tau^{-m_1}_\beta \dots \tau^{n_k}_\alpha \tau^{-m_k}_\beta$ being the first return map of an Anosov flow with orientable stable and unstable foliations implies that the quotient of the Birkhoff section by $\sigma$ is some genus one surface $T$. The quotient is determined by some cohomology class in $H^1(T; \mathbb{Z}/N)$, represented by some cocycle on $T$. Conversely, such a cocycle determines a Penner type pseudo-Anosov map by lifting the images of $\alpha$ and $\beta$ in $T$ to the corresponding cover. For the zero cocycle, this map is an Anosov map.

We apply further surgery operations to change the data of this cocycle, one edge at a time, until we reach the zero cocycle. Here the almost veering branched surface that encodes the suspension flow of $\sigma$ is used as a reference frame to locate the surgery curves. 

Using our correspondence result \Cref{thm:layeredvthsur}, this sequence of surgery operations on veering triangulations translates to a sequence of horizontal Goodman surgeries on flows, showing \Cref{thm:pennergenusone}.

We note that we do not actually need the full force of the correspondence theorem between veering triangulations and pseudo-Anosov flows. What is needed is the statement that every veering triangulation corresponds to a unique pseudo-Anosov flow. See \Cref{thm:uniqueflow} for the precise statement. Perhaps surprisingly, this statement does not follow from known results of the correspondence theory. We thus have to devote \Cref{sec:uniqueflow} to its proof. 

\Cref{thm:uniqueflow} is communicated to us by Michael Landry and Samuel Taylor. We thank them for allowing us to reproduce their ideas in this paper.

\subsection{Sketch of proof of \Cref{thm:totallyperiodicgenusone}} \label{subsec:intrototallyperiodicgenusoneproof}

For \Cref{thm:totallyperiodicgenusone}, we use the fact, shown in \cite{BF15}, that totally periodic Anosov flows can be combinatorially encoded by their spine graphs and choice of gluing maps between the boundary tori of the Seifert fibered pieces. 

Here the \textbf{spine graph} on the base surface of a Seifert fibered piece is the projection of a union of Birkhoff annuli which encapsulates the dynamics of the flow in that Seifert fibered piece. In this paper, it will be more convenient for us to work with the \textbf{flow graph}, which is a dual to the spine graph. We refer to \Cref{sec:totallyperiodicdefn} for more precise definitions.

First of all, we observe that we can arbitrarily modify the gluing maps between the boundary tori by performing horizontal Goodman surgery along curves lying on the tori. By \Cref{thm:horsuralmostequiv}, this means that if we only care about the almost equivalence class of the flow, we can discard the data of the gluing maps and only remember the flow graphs.

The rest of the proof of \Cref{thm:totallyperiodicgenusone} involves showing that we can modify the flow graphs in various ways by performing horizontal Goodman surgery along other curves. For example, by performing surgery along a curve lying on the torus over a curve $c$ on a base surface $S$, we can cut $S$ and the flow graph on $S$ along $c$.

Aside from this cutting move, we also describe a gluing move and an insertion move in \Cref{sec:basicmoves}.
The insertion move is based on an observation appearing in unpublished work of Sergio Fenley, Jessica Purcell, and Mario Shannon. We thank them for sharing their work with us.

Using these moves, we show that we can modify the flow graphs to be of a simple form. This translates to a chain of horizontal Goodman surgeries relating the original flow to a particularly simple totally periodic Anosov flow, which we then relate to a suspension Anosov flow by yet more horizontal Goodman surgeries.

\subsection{History of \Cref{conj:genusonesections}}

Strictly speaking, our attribution of \Cref{conj:genusonesections} to Fried and Ghys is inaccurate. In \cite{Fri83}, Fried asks whether every transitive Anosov flow admits a genus one Birkhoff section. Note that Fried does not assume that the Anosov flow has orientable stable and unstable foliations, although presumably he wishes for the induced stable and unstable foliations on the Birkhoff section to be orientable, for otherwise one loses the special description of the first return map as explained above. There is also no indication from \cite{Fri83} as to whether Fried expects the answer to his question to be `yes' or `no'. 

Meanwhile, in a few talks given around the 2000s, Ghys proposed a conjecture concerning the almost equivalence of Anosov flows. The contents of these talks are unfortunately unavailable to us, but in conversations with Pierre Dehornoy, we understand what Ghys conjectured to be that any two transitive Anosov flows are \textbf{almost commensurable}, i.e. they are commensurable after drilling out finitely many closed orbits. At least a priori, almost commensurability is much weaker than almost equivalence. In particular, note that the stable and unstable foliations of any Anosov flow can be made orientable up to lifting to a finite cover.

\Cref{conj:genusonesections} is thus an amalgation of ideas proposed by Fried and Ghys, even though to our knowledge, neither of them have explicitly made such a statement.

\subsection{Outline of paper}

This paper is divided into 3 parts. Part 1 includes \Cref{sec:prelim} and \Cref{sec:flowhsur}, and contains background information about Anosov flows, Birkhoff sections, and horizontal Goodman surgery.

Part 2 of the paper is aimed at proving \Cref{thm:pennergenusone}. In \Cref{sec:pennerpa}, we define and recall some facts about Penner type pseudo-Anosov maps. In \Cref{sec:vt}, we recall some material on veering triangulations and veering branched surfaces. In \Cref{sec:uniqueflow}, we prove \Cref{thm:uniqueflow}. In \Cref{sec:hsurcorr}, we prove \Cref{thm:layeredvthsur}. In \Cref{sec:pennergenusoneproof}, we assemble the ingredients to prove \Cref{thm:pennergenusone}.

Part 3 of the paper is aimed at proving \Cref{thm:totallyperiodicgenusone}. In \Cref{sec:totallyperiodicdefn}, we recall the definition of totally periodic Anosov flows and explain how these can be combinatorially encoded by flow graphs. In \Cref{sec:basicmoves}, we introduce moves which one can perform on flow graphs, so that the corresponding totally periodic Anosov flows remain in the same almost equivalence class. In \Cref{sec:totallyperiodicgenusoneproof}, we apply these moves to show \Cref{thm:totallyperiodicgenusone}.

We have organized the contents of the paper in this way so that Parts 2 and 3 would be completely independent. 

\subsection*{Acknowledgements}

We would like to thank Michael Landry and Samuel Taylor for allowing us to reproduce \Cref{thm:uniqueflow} in this paper. We would like to thank Sergio Fenley, Jessica Purcell, and Mario Shannon for sharing their unpublished work with us. We would also like to thank Pierre Dehornoy for helpful conversations. Part of this project was completed while the author is based at CIRGET. We would like to thank the center for its support.

\vspace{0.5cm}
\begin{large}
\begin{center}
\textbf{Part 1. Background}
\end{center}
\end{large}

\section{Preliminaries} \label{sec:prelim}

In this section, we review some basic definitions and facts regarding Anosov flows.

\subsection{Anosov flows} \label{subsec:anosovflowdefn}

We start with the definition of an Anosov flow.

\begin{defn} \label{defn:Aflow}
An \textbf{Anosov flow} on a closed oriented $3$-manifold $M$ is a smooth flow $\phi^t$ for which there is a Riemannian metric $g$ and a continuous splitting of the tangent bundle into three $\phi^t$-invariant line bundles $TM=E^s \oplus T\phi \oplus E^u$ such that 
$$||d\phi^t(v)||_g < C \lambda^{-t} ||v||_g$$
for every $v \in E^s, t>0$, and 
$$||d\phi^t(v)||_g < C \lambda^t ||v||_g$$
for every $v \in E^u, t<0$, for some $C, \lambda>1$.
\end{defn}

It is a basic fact (see for example \cite{HPS77}) that the plane fields $T\phi \oplus E^s$ and $T\phi \oplus E^u$ integrate uniquely to 2-dimensional foliations $\Lambda^s$ and $\Lambda^u$. We refer to these as the \textbf{stable} and \textbf{unstable foliations} and denote them by $\Lambda^s$ and $\Lambda^u$ respectively.

We sketch a local picture of $M$ with $\Lambda^s$ and $\Lambda^u$ in \Cref{fig:anosovlocal} left. In this paper, we adopt the convention that the flow goes vertically upwards, and we color the leaves of $\Lambda^s$ and $\Lambda^u$ by green and purple respectively.

\begin{defn} \label{defn:orbitspace}
Let $\phi^t$ be an Anosov flow on a closed oriented $3$-manifold $M$. Let $\widetilde{\phi}^t$ be the lifted flow to the universal cover $\widetilde{M}$. It is shown in \cite[Proposition 4.1]{FM01} that the space of orbits $\mathcal{O}$ of $\widetilde{\phi}^t$, with the quotient topology, is homeomorphic to $\mathbb{R}^2$. The lifted stable/unstable foliations $\widetilde{\Lambda^{s/u}}$ induce 1-dimensional foliations $\mathcal{O}^{s/u}$ on $\mathcal{O}$. We refer to the space $\mathcal{O}$ with the foliations $\mathcal{O}^{s/u}$ as the \textbf{orbit space} of $\phi^t$.
\end{defn}

In this paper, we adopt the convention of considering the leaves of $\mathcal{O}^s$ as vertical lines and the leaves of $\mathcal{O}^u$ as horizontal lines. We illustrate the local picture of the orbit space of an Anosov flow in \Cref{fig:anosovlocal} right.

\begin{figure}
    \centering
    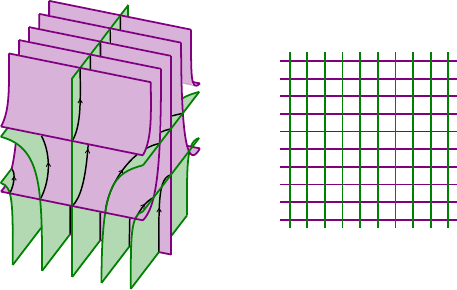
    \caption{Left: A local picture of an Anosov flow. Right: A local picture of its orbit space.}
    \label{fig:anosovlocal}
\end{figure}

In this paper, we will mainly focus on Anosov flows whose stable and unstable foliations are orientable. This is equivalent to the condition that the line bundles $E^s$ and $E^u$ are orientable. Since we assume that the $3$-manifold $M$ is oriented, this is in turn equivalent to the condition that one of $E^s$ and $E^u$ is orientable. In such a setting we will orient $E^s$ and $E^u$ in a way so that $E^s \oplus T\phi \oplus E^u$ agrees with the orientation on $M$.

Next, we recall the definitions of orbit equivalence and almost equivalence.

\begin{defn} \label{defn:orbitequiv}
Let $\phi^t_i$ be a continuous flow on a $3$-manifold $N_i$ for $i=1,2$. We say that $\phi^t_1$ and $\phi^t_2$ are \textbf{orbit equivalent}, and write $\phi^t_1 \cong \phi^t_2$, if there exists a homeomorphism $h:N_1 \to N_2$ that sends the orbits of $\phi^t_1$ to those of $\phi^t_2$ in an orientation preserving way (but not necessarily preserving their parametrizations).
\end{defn}

We will consider two Anosov flows as the same if they are orbit equivalent.

\begin{defn} \label{defn:almostequiv}
Let $\phi^t_i$ be an Anosov flow on a closed oriented $3$-manifold $M_i$ for $i=1,2$. We say that $\phi^t_1$ and $\phi^t_2$ are \textbf{almost equivalent}, and write $\phi^t_1 \sim \phi^t_2$, if there exists a finite collection of closed orbits $\mathcal{C}_i$ of $\phi^t_i$ and an orbit equivalence between $\phi^t_1$ restricted to $M_1 \backslash \mathcal{C}_1$ and $\phi^t_2$ restricted to $M_2 \backslash \mathcal{C}_2$.
\end{defn}

Note that if $\phi^t_1$ and $\phi^t_2$ are almost equivalent, then $\phi^t_1$ has orientable stable and unstable foliations if and only if $\phi^t_2$ has orientable stable and unstable foliations.

Finally, we review the construction of suspension Anosov flows.

\begin{constr} \label{constr:susflow}
Let $A \in \mathrm{SL}_2 \mathbb{Z}$ be a matrix such that $|\mathrm{tr} A| > 2$. An elementary computation shows that $A$ has two real eigenvalues $\lambda$ and $\lambda^{-1}$, where $|\lambda|>1$, with associated eigenvectors $e_{\lambda}$ and $e_{\lambda^{-1}}$ respectively. 
$A$ induces a homeomorphism on the torus $T^2=\mathbb{R}^2/\mathbb{Z}^2$, which we denote as $A$ as well.

Recall that the \textbf{mapping torus} $T_f$ of a surface homeomorphism $f:S \to S$ is constructed by taking $S \times [0,1]_t$ and gluing $S \times \{1\}$ to $S \times \{0\}$ by $f$. The \textbf{suspension flow} $\phi^t_f$ on $T_f$ is the flow induced by the vector field $\partial_t$.

In this setting, $\phi^t_A$ is an Anosov flow. Its stable and unstable line bundles lie along the directions of $e_{\lambda^{-1}}$ and $e_{\lambda}$ respectively. $\phi^t_A$ has orientable stable and unstable foliations if and only if $\lambda>1$, which occurs exactly when $\mathrm{tr} A > 2$

We say that an Anosov flow $\phi^t$ is a \textbf{suspension Anosov flow} if $\phi^t$ is orbit equivalent to $\phi^t_A$ for some $A \in \mathrm{SL}_2 \mathbb{Z}$.
\end{constr}

\subsection{Birkhoff sections} \label{subsec:birkhoffsection}

In this subsection, we recall the definition of Birkhoff sections and explain their relation with pseudo-Anosov maps. Most of the material in this subsection can be found in \cite{Fri83}.

\begin{defn} \label{defn:birkhoffsection}
Let $\phi^t$ be an Anosov flow on a closed 3-manifold $M$. A \textbf{Birkhoff section} is a cooriented compact surface with boundary $S$ such that:
\begin{itemize}
    \item $S$ is embedded along its interior $S^\circ$, where it is positively transverse to the orbits of $\phi^t$.
    \item $S$ is immersed along its boundary $\partial S$, where it lies along a union of closed orbits of $\phi^t$.
    \item Every orbit of $\phi^t$ intersects $S$ in finite forward and finite backward time, that is, for every $x \in M$, there exists $t_1, t_2 > 0$ such that $\phi^{t_1}(x) \in S$ and $\phi^{-t_2}(x) \in S$.
\end{itemize}
\end{defn}

We have the following basic existence result.

\begin{thm}[Fried \cite{Fri83}] \label{thm:friedbirkhoffsectionexist}
Every transitive Anosov flow admits a Birkhoff section.
\end{thm}

Here recall that an Anosov flow is \textbf{transitive} if the set of its closed orbits is dense.

The utility of Birkhoff sections is that it allows one to represent $3$-dimensional flows using $2$-dimensional maps. More precisely, suppose $\phi^t$ is an Anosov flow and suppose $S$ is a Birkhoff section to $\phi^t$. The last item in \Cref{defn:birkhoffsection} allows us to define the first return map $f$ on the interior $S^\circ$ of $S$: For $x \in S^\circ$, $f(x)$ is the point where $\phi^{(0,\infty)}(x)$ first intersects $S^\circ$. We abuse notation slightly and refer to $f$ as the \textbf{first return map} of the Birkhoff section $S$. Our next task is to characterize these first return maps. 

\begin{defn} \label{defn:fppAmap}
Let $S^\circ$ be a finite-type surface. A homeomorphism $f:S^\circ \to S^\circ$ is a \textbf{fully-punctured pseudo-Anosov map} if there exists a transverse pair of measured foliations $\ell^s,\ell^u$ such that
\begin{itemize}
    \item $f$ contracts the measure of $\ell^s$ by $\lambda^{-1}$ and expands the measure of $\ell^u$ by $\lambda$, for some $\lambda >1$, and
    \item around every puncture of $S^\circ$, the pair of foliations $(\ell^s, \ell^u)$ is conjugate to the pull back of the foliation of $\mathbb{C} \backslash \{0\}$ by vertical and horizontal lines under $z \mapsto z^{\frac{n}{2}}$, for some $n \geq 1$. 
\end{itemize}

We call the foliations $\ell^s$ and $\ell^u$ the \textbf{stable} and \textbf{unstable foliations} of $f$ respectively.
See \cite{FLP79} for more details about pseudo-Anosov maps.
\end{defn}

It is a standard fact that the first return map of a Birkhoff section $S$ to an Anosov flow $\phi^t$ is a fully-punctured pseudo-Anosov map $f$. 
The stable/unstable foliation of $f$ is obtained by taking the intersection of the stable/unstable foliation $\Lambda^{s/u}$ of $\phi^t$ with $S^\circ$, respectively.

Observe that $f$ acts on the set of punctures of $S^\circ$. 
The closed orbits in $\partial S$ are in one-to-one correspondence with orbits of this action.
Similarly, $f$ acts on the set of leaves of $\ell^{s/u}$ that are incident to the punctures. 
The half-leaves of $\Lambda^{s/u}$ that are incident to the closed orbits in $\partial S$ are in one-to-one correspondence with orbits of this action, respectively.

This shows that our first return maps satisfy the following additional condition.

\begin{defn} \label{defn:Anosovtype}
We say that a fully-punctured pseudo-Anosov map $f$ is of \textbf{Anosov-type} if there are at most two orbits of leaves of $\ell^s$ incident to each puncture.

Note that equivalently, we could replace $\ell^s$ by $\ell^u$ in this definition.
\end{defn}

Conversely, given an Anosov-type fully-punctured pseudo-Anosov map $f:S^\circ \to S^\circ$, consider the suspension flow $\phi^t_f$ on the mapping torus $T_f$. Note that $T_f$ is the interior of a compact $3$-manifold $\overline{T}_f$ with torus boundary components. We can extend $\phi^t_f$ to a flow $\overline{\phi}^t_f$ on $\overline{T}_f$ whose closed orbits on each boundary torus are in one-to-one correspondence with the orbits of leaves of $\ell^s$ and $\ell^u$ incident to the corresponding punctures. We can then construct an Anosov flow $\phi^t$ on a closed $3$-manifold by collapsing each boundary torus along a slope that intersects the set of closed orbits in exactly $4$ points. The image of $S^\circ$ becomes the interior of a Birkhoff section $S$ and the first return map of $S$ is given by $f$. We provide an illustration of this construction in \Cref{fig:pamaptoflow}.

\begin{figure}
    \centering
    \resizebox{!}{4.5cm}{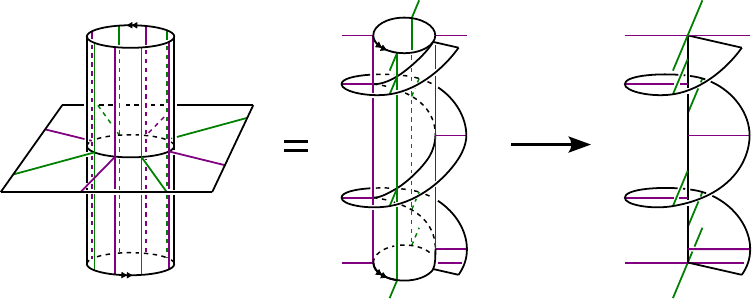}
    \caption{Given an Anosov-type fully-punctured pseudo-Anosov map $f:S^\circ \to S^\circ$, we can construct an Anosov flow by collapsing the extension $\overline{\phi}^t_f$ of the suspension flow along boundary tori. The image of $S^\circ$ becomes the interior of a Birkhoff section.}
    \label{fig:pamaptoflow}
\end{figure}

Note that the Anosov flow $\phi^t$ here is not uniquely determined, since there are many possible choices of slopes to collapse the boundary tori along. 
What is true, however, is that $\phi^t$ is determined up to almost equivalence.

We also observe that the flow $\phi^t$ has orientable stable and unstable foliations if and only if the first return map $f$ has orientable stable and unstable foliations and $f$ preserves these orientations. In this case there are exactly two orbits of leaves of $\ell^s$ incident to each puncture.

\begin{rmk} \label{rmk:pAflow}
If one does not have the Anosov-type assumption on $f$, then one might not be able to find a slope that intersects the set of closed orbits in exactly $4$ points to collapse each boundary torus along. We can generalize the construction by instead collapsing along slopes that intersect the set of closed orbits in $\geq 4$ points.  Then we would end up with what is known as a pseudo-Anosov flow $\phi^t$ on a closed $3$-manifold. 

In general, a \textbf{pseudo-Anosov flow} is essentially an Anosov flow with finitely many singular orbits, where the flow exhibits dynamics of the form shown in \Cref{fig:palocal}. Since these more general flows do not play a role in showing the main theorems of this paper, we will omit their precise definition, and instead refer the reader to \cite{Mos96}. Some of the material in this paper will be developed in the more general setting of pseudo-Anosov flows however. This is done in anticipation for future applications. The reader who is only interested in the main theorems of this paper can restrict the setting to Anosov flows in those sections.

\begin{figure}
    \centering
    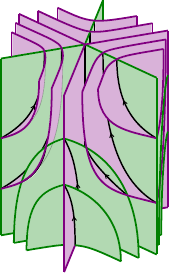
    \caption{A local picture of a pseudo-Anosov flow near a singular orbit.}
    \label{fig:palocal}
\end{figure}
\end{rmk}

Finally, we record the following proposition which illustrates the connection between Birkhoff sections and almost equivalence. 

\begin{prop} \label{prop:genusonesection=sus}
An Anosov flow with orientable stable and unstable foliations admits a genus one Birkhoff section if and only if it is almost equivalent to a suspension Anosov flow.
\end{prop}
\begin{proof}
Let $\phi^t$ be an Anosov flow with orientable stable and unstable foliations. Suppose $\phi^t$ admits a genus one Birkhoff section $S$. Then the stable and unstable foliations of the first return map are orientable. By a Poincare-Hopf argument, each puncture of $S^\circ$ must meet exactly two leaves of $\ell^s$ and exactly two leaves of $\ell^u$. Thus $f$ extends to a pseudo-Anosov map $\overline{f}$ defined on the torus $\overline{S}$ obtained by filling in the punctures. It is a classical fact that the map $\overline{f}$ must be that induced by an element of $A \in \mathrm{SL}_2 \mathbb{R}$. Thus $\phi^t$ in the complement of $\partial S$ is orbit equivalent to $\phi^t_A$ restricted to the suspension of $S^\circ \subset \overline{S}$, showing that these two flows are almost equivalent.

Conversely, suppose an Anosov flow $\phi^t$ defined on a $3$-manifold $M$ is almost equivalent to a suspension Anosov flow $\phi^t_A$, say, $\phi^t$ in the complement of $\mathcal{C}$ is orbit equivalent to $\phi^t_A$ in the complement of $\mathcal{C}_A$. Then one can transfer the surface $T^2 \backslash \mathcal{C}_A$ from $T_A$ to $M$ and take the closure to get some genus one Birkhoff section with boundary components lying along a subset of $\mathcal{C}$.
\end{proof}

\subsection{Goodman-Fried surgery} \label{subsec:gfsur}

Yet another perspective to almost equivalence is via Goodman-Fried surgery. 
Here we use the term `Goodman-Fried surgery' to refer to the surgery operation introduced by Goodman in \cite{Goo83} and that introduced by Fried in \cite{Fri83}, the two of which then showed to be equivalent by Shannon in his thesis \cite{Sha20}, at least for transitive flows. We will not go into the details of the surgery operation; we refer the reader to \cite{Sha20} for this. We simply state the essential features of the operation in the following proposition.

\begin{prop} \label{prop:gfsurgery}
Let $\phi^t$ be a transitive pseudo-Anosov flow on a closed $3$-manifold $M$. Let $\gamma$ be a closed orbit of $\phi^t$. The half leaves of the stable foliation that contain $\gamma$ determines a collection of curves on the boundary of a tubular neighborhood $N(\gamma)$ of $\gamma$. We refer to this as the \textbf{degeneracy locus} and denote it by $d$. 

For every slope $s$ on $\partial N(\gamma)$ which intersects $d$ in $n \geq 2$ points, there exists a pseudo-Anosov flow $\widehat{\phi}^t$ defined on a closed $3$-manifold $\widehat{M}$ satisfying:
\begin{itemize}
    \item $\widehat{\phi}^t$ admits a $n$-pronged closed orbit $\widehat{\gamma}$ such that the restriction of $\phi^t$ to $M \backslash \gamma$ is orbit equivalent to $\widehat{M} \backslash \widehat{\gamma}$, where
    \item under such an orbit equivalence, the meridian determined by $\widehat{M}$ on $\partial N(\gamma)$ is the slope $s$, and
    \item the orbit equivalence can be chosen to be tame near $\gamma$ and $\widehat{\gamma}$, in the sense that the image of every local section $D$ in $M$ near $\gamma$ can be extended into an immersed surface with boundary $\overline{D}$ in $\widehat{M} \backslash \widehat{\gamma}$. 
\end{itemize}
\end{prop}

\section{Horizontal Goodman surgery} \label{sec:flowhsur}

In this section, we review the material in \cite{Tsa24}. We first recall the definition of horizontal surgery curves, then we review the horizontal Goodman surgery operation, and finally we state the main theorem in \cite{Tsa24}.

\subsection{Horizontal surgery curves} \label{subsec:flowhsurcurve}

Let $\phi^t$ be an Anosov flow on a closed oriented $3$-manifold $M$. 
Fix a Riemannian metric $g$ on $M$. Let $e^{s/u}$ be unit length vectors in $E^{s/u}|_x$, respectively, such that $(e^s, \dot{\phi}, e^u)$ determines a positive basis of $TM|_x$. Suppose $\ell$ is a line in $TM|_x$ spanned by a vector $ae^s+b\dot{\phi}+ce^u$. If $a \neq 0$ and $c \neq 0$, we define the \textbf{slope} of $\ell$ (with respect to $g$) to be $\frac{a}{c}$.

We say that $\ell$ is \textbf{positive/negative} if the slope of $\ell$ is positive/negative, respectively. Notice that while the slope of $\ell$ depends on the choice of the Riemannian metric $g$, whether $\ell$ is positive/negative depends only on $\ell$.

Suppose that $L$ is a 1-dimensional sub-bundle of $TM$ defined over a set $K \subset M \backslash \sing(\phi^t)$. We say that $L$ is \textbf{positive/negative} if $L|_x$ is positive/negative at every $x \in K$, respectively.

\begin{defn} \label{defn:posnegcurve}
Let $c$ be a smooth embedded $1$-manifold (possibly with boundary) in $M$. We say that $c$ is \textbf{positive/negative} if $Tc$ is positive/negative. 
\end{defn}

In this paper, we will illustrate positive curves in red and negative curves in blue. See \Cref{fig:posnegcurve} for a local example of a positive and a negative curve projected onto the orbit space.

\begin{figure}
    \centering
    \resizebox{!}{3cm}{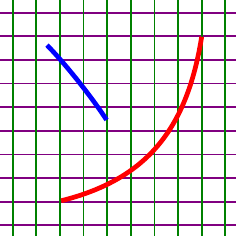}
    \caption{A local example of a positive curve (in red) and a negative curve (blue), projected along the flow direction onto a plane.}
    \label{fig:posnegcurve}
\end{figure}

\begin{defn} \label{defn:steady}
Suppose that $L$ is a positive 1-dimensional sub-bundle of $TM$ defined over a set $K \subset M \backslash \sing(\phi^t)$. We say that $L$ is \textbf{steady} if for every $x,y \in K, t>0$ such that $y=\phi^t(x)$, we have $\slope(L|_y) > \slope(d\phi^t(L|_x))$. 

Similarly, suppose that $L$ is a negative 1-dimensional sub-bundle of $TM$ defined over a set $K \subset M \backslash \sing(\phi^t)$. We say that $L$ is \textbf{steady} if for every $x,y \in K, t>0$ such that $y=\phi^t(x)$, we have $\slope(L|_y) < \slope(d\phi^t(L|_x))$. 
\end{defn}

As noted in \cite[Section 3.1]{Tsa24}, $L$ being steady is independent of the choice of the Riemannian metric.

\begin{defn} \label{defn:horsurcurve}
Let $c$ be a positive/negative smooth embedded curve in $M \backslash \sing(\phi^t)$. We say that $c$ is a \textbf{positive/negative horizontal surgery curve}, respectively, if $Tc$ is steady.
\end{defn}

We import the following terminology from \cite{Tsa24}: Suppose $c_1$ and $c_2$ are smooth embedded $1$-manifolds (possibly with boundary) in $M$. We say that $(x,y,t) \in c_1 \times c_2 \times (0,\infty)$ is a \textbf{time $t$ crossing of $c_2$ over $c_1$} if $y=\phi^t(x)$. If $c_1=c_2=c$ we abbreviate this to a \textbf{time $t$ crossing of $c$}.

If $c$ is a horizontal surgery curve, then for every crossing $(x,y,t)$ of $c$, we can take a small neighborhood of the orbit segment from $x$ to $y$ and project it along the flow direction onto a small disc. The projected segments of $c$ intersect transversely since $\slope(Tc|_y) \neq \slope(d\phi^t(Tc|_x))$. The data that $y$ lies above $x$ gives this transverse intersection point the same data as a crossing in a knot diagram.

From this perspective, the steadiness condition on $Tc$ can be reformulated as requiring that every crossing be of the form in \Cref{fig:steadycurve} left or right, depending on whether $c$ is positive or negative, respectively. 

\begin{figure}
    \centering
    \resizebox{!}{3cm}{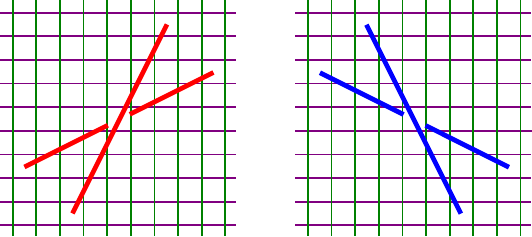}
    \caption{Left/right: Local form of a positive/negative surgery curve, respectively.}
    \label{fig:steadycurve}
\end{figure}

\subsection{Horizontal Goodman surgery and almost equivalence} \label{subsec:flowhsuralmostequiv}

In this subsection, we give a brief outline of the horizontal Goodman surgery operation. There are certain technical definitions in the construction which we will not state here, but will instead refer the reader to \cite{Tsa24}, since they will not play a role in this paper.

\begin{defn} \label{defn:surann}
A \textbf{positive/negative surgery annulus} is an embedded oriented annulus $A \subset M$ that is positively transverse to the flow, along with:
\begin{itemize}
    \item a foliation by positive/negative curves $\mathcal{H}$ where $T\mathcal{H}$ is steady, and
    \item a foliation by negative/positive non-separating arcs $\mathcal{K}$ where $T\mathcal{K}$ is steady,
\end{itemize} 
respectively.
\end{defn}

\begin{constr} \label{constr:sur}
Let $\phi^t$ be an Anosov flow on a closed oriented $3$-manifold $M$. Let $c$ be a positive/negative horizontal surgery curve. Let $n$ be a positive/negative integer, respectively.

\cite[Proposition 4.2]{Tsa24} states that there exists a surgery annulus $(A, \mathcal{H}, \mathcal{K})$ where $c$ is one of the leaves of $\mathcal{H}$ in the interior of $A$.

Let $M \cut A$ be the space obtained by cutting $M$ along $A$. There are two copies of $A$ on $M \cut A$ --- one on the positive side of $A$ and the other on the negative side of $A$. We denote these by $A_+$ and $A_-$ respectively. 

Meanwhile, note that the foliations $\mathcal{H}$ and $\mathcal{K}$ induce a parametrization $A \cong S^1 \times I$. We consider homeomorphisms $\sigma:A \to A$ of the form $\sigma(h,k)=(h+\rho(k),k)$, where $\rho(k)$ is a non-increasing/non-decreasing function, respectively, such that $\rho(0)-\rho(1)=n$.

\cite[Theorem 4.5]{Tsa24} states that for appropriate choices of $\rho$, if we glue $A_-$ to $A_+$ via $\sigma$, then the resulting flow $\phi^t_\sigma(A, \mathcal{H}, \mathcal{K}, \alpha)$ is Anosov.

Moreover, \cite[Theorem 4.12]{Tsa24} states that the orbit equivalence class of $\phi^t_\sigma(A, \mathcal{H}, \mathcal{K}, \alpha)$ only depends on $c$ and $n$. Hence we will write $\phi^t_\sigma(A, \mathcal{H}, \mathcal{K}, \alpha)$ as $\phi^t_{\frac{1}{n}}(c)$. Note that this flow is defined on the $3$-manifold $M_{\frac{1}{n}}(c)$ obtained by performing Dehn surgery on $M$ along $c$ with coefficient $\frac{1}{n}$.
\end{constr}

The key property of horizontal Goodman surgery that we use in this paper is the following theorem in \cite{Tsa24}.

\begin{thm:horsuralmostequiv}[{\cite[Theorem 1.3]{Tsa24}}]
Let $\phi^t$ be a transitive Anosov flow on a closed oriented $3$-manifold $M$. Let $c$ be a positive/negative horizontal surgery curve for $\phi^t$. Then for every positive/negative integer $n$, respectively, the flow $\phi^t_{\frac{1}{n}}(c)$ is almost equivalent to $\phi^t$.
\end{thm:horsuralmostequiv}

\vspace{0.5cm}
\begin{large}
\begin{center}
\textbf{Part 2. Penner type Anosov flows}
\end{center}
\end{large}

\section{Penner type pseudo-Anosov maps} \label{sec:pennerpa}

In this section we review the definition of a Penner type pseudo-Anosov map. We then specialize to the case where the map is of Anosov type, has orientable stable and unstable foliations and preserves those orientations, and provide a description of the mapping tori of such maps in terms of Dehn surgeries.

\subsection{Definition of Penner type maps} \label{subsec:pennerpadefn}

Let $S$ be a closed surface. Let $\alpha = \bigsqcup_{i=1}^p \alpha_i$ and $\beta = \bigsqcup_{j=1}^q \beta_j$ be two multicurves on $S$ such that $\alpha \cup \beta$ \textbf{fills} $S$, i.e. $\alpha$ intersects $\beta$ transversely and every complementary region of $\alpha \cup \beta$ is a disc with $2n \geq 4$ corners.

Let $S^\circ$ be the punctured surface obtained by removing one point from every complementary region of $\alpha \cup \beta$. We illustrate one example of $(S^\circ, \alpha, \beta)$ in \Cref{fig:pennercurves} left. Here we color $\alpha$ in dark red and $\beta$ in dark blue.

\begin{figure}
    \centering
    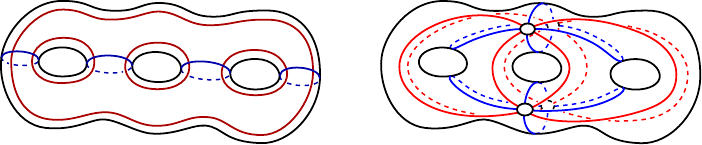
    \caption{Left: One example of $(S^\circ, \alpha, \beta)$. Right: The corresponding dual ideal quadrangulation.}
    \label{fig:pennercurves}
\end{figure}

Note that the set $\alpha \cup \beta$ has a natural structure of a $4$-valent fat graph, i.e. a $4$-valent graph with the data of a cyclic ordering of half-edges around each vertex. 

We can construct the \textbf{dual ideal quadrangulation} of $\alpha \cup \beta$ by placing an edge $f_e$ across every edge $e$ of $\alpha \cup \beta$, connecting the two punctures of $S^\circ$ lying in the two complementary regions of $\alpha \cup \beta$ that abut $e$. See \Cref{fig:pennercurves} right. 

We say that the edge $f_e$ is \textbf{positive} or \textbf{negative} if $e$ lies along $\beta$ or $\alpha$ respectively. The reason for this terminology will be apparent in \Cref{thm:pennerpA}. We color positive/negative curves in red/blue respectively, as is consistent with our convention.

For each $p$-tuple of integers $n=(n_1,...,n_p)$ we denote by $\tau^n_\alpha$ the product $\prod_{i=1}^p \tau^{n_i}_{\alpha_i}$, where each $\tau_{\alpha_i}$ is a (positive) Dehn twist around $\alpha_i$. Notice that since the $\alpha_i$ are disjoint, the order of the product does not matter here. For each $q$-tuple of integers $m=(m_1,...,m_q)$ we define $\tau^m_\beta$ similarly.

We say that a map $f:S^\circ \to S^\circ$ is of \textbf{Penner type} if it is isotopic to some map of the form 
$$\sigma \tau^{n_1}_\alpha \tau^{-m_1}_\beta \dots \tau^{n_k}_\alpha \tau^{-m_k}_\beta$$
where $\sigma$ is a homeomorphism of $S^\circ$ preserving $\alpha$ and $\beta$, and each $n_s$ and $m_s$ are tuples of non-negative integers.

Here $\sigma$ may not preserve the individual components of $\alpha$ and $\beta$, but rather it permutes them in general. We define the \textbf{$\sigma$-orbits} in $\alpha$ or $\beta$ to be the orbits of their components under this permutation. We say that a $p$-tuple of integers $n=(n_1,...,n_p)$ is \textbf{positive on a $\sigma$-orbit $\mathfrak{a}$ of $\alpha$} if $\displaystyle \sum_{\alpha_i \in \mathfrak{a}} n_i > 0$. We define a $q$-tuple of integers $m$ to be \textbf{positive on a $\sigma$-orbit of $\beta$} similarly.

In \cite{Pen88}, Penner showed the following theorem.

\begin{thm} \label{thm:pennerpA}
Suppose $n_1,...,n_k$ are $p$-tuples of non-negative integers such that $\sum_s n_s$ is positive on every $\sigma$-orbit of $\alpha$, and $m_1,...,m_k$ are $q$-tuples of non-negative integers such that $\sum_s m_s$ is positive on every $\sigma$-orbit of $\beta$.
Then 
$$\sigma \tau^{n_1}_\alpha \tau^{-m_1}_\beta \dots \tau^{n_k}_\alpha \tau^{-m_k}_\beta$$
is isotopic to a fully-punctured pseudo-Anosov map $f$, respectively.

Moreover, up to isotopy, the positive/negative edges of the dual ideal quadrangulation are positive/negative with respect to the suspension flow of $f$.
\end{thm}

\begin{rmk} \label{rmk:pennerpa}
Strictly speaking, Penner only showed \Cref{thm:pennerpA} in the case when $\sigma$ is identity. However, his proof can be easily adapted to our more general setting. Alternatively, one can apply his version of \Cref{thm:pennerpA} to a suitable power of our map.
\end{rmk}

\begin{rmk} \label{rmk:DU15}
It is known that not every (fully-punctured) pseudo-Anosov map is of Penner type. This can be deduced from \cite[Theorem 11.2]{Lei04}; see also \cite{Lie17}. In fact, from \cite{SS15}, it is even known that not every (fully-punctured) pseudo-Anosov map has a power that is of Penner type.

However, it is known that every hyperelliptic (fully-punctured) pseudo-Anosov map admits a very similar decomposition in terms of quadrangulations and diagonal changes. We refer to \cite{DU15} for details on this. Our terminology of quadrangulations is also taken from there.
\end{rmk}

The second statement in \Cref{thm:pennerpA} implies that the stable and unstable foliations in each quadrilateral must be of the form shown in \Cref{fig:pennerquad}. 
In particular, a leaf of the stable foliation $\ell^s$ that is incident to a puncture must emerge out of an ideal vertex of a quadrilateral that meets a negative side on its left and a positive side on its right. 
Symmetrically, a leaf of the unstable foliation $\ell^u$ that is incident to a puncture must emerge out of an ideal vertex of a quadrilateral that meets a positive side on its left and a negative side on its right.

\begin{figure}
    \centering
    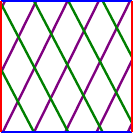
    \caption{The stable and unstable foliations in a quadrilateral.}
    \label{fig:pennerquad}
\end{figure}

\subsection{Penner type mapping tori} \label{subsec:pennermaptori}

Consider a Penner type map $\sigma \tau^{n_1}_\alpha \tau^{-m_1}_\beta \dots \tau^{n_k}_\alpha \tau^{-m_k}_\beta$ as in the previous subsection.

It will be convenient for us to think of $\sigma$ as an automorphism of $\alpha \cup \beta$, or equivalently as an automorphism of the dual ideal quadrangulation. In particular, since we only care about isotopy classes of maps, we can assume that $\sigma$ is of finite order. This allows us to state the following observation.

\begin{lemma} \label{lemma:anosovtypequotientsurface}
Suppose $\sigma \tau^{n_1}_\alpha \tau^{-m_1}_\beta \dots \tau^{n_k}_\alpha \tau^{-m_k}_\beta$ is isotopic to an Anosov type fully-punctured pseudo-Anosov map $f$ with orientable stable and unstable foliations and that $f$ preserves these orientations, then the finite cyclic group $\langle \sigma \rangle$ acts freely on $S^\circ$ and the quotient surface $S^\circ / \langle \sigma \rangle$ is of genus one.
\end{lemma}
\begin{proof}
If some $\sigma^i$ fixes a point in the interior of a quadrilateral, then it preserves the quadrilateral. Since $f$ preserves the orientations on its stable and unstable foliations, $\sigma^i$ must fix the quadrilateral (as opposed to rotating it by $\pi$). But since $S^\circ$ is connected, $\sigma^i$ must then fix every quadrilateral, hence is the identity.

If $\sigma^i$ fixes a point on an edge, then either $\sigma^i$ fixes the two incident quadrilateral or it swaps those two quadrilaterals by rotating by $\pi$ around the edge. The latter case cannot happen since $f$ preserves the orientations on its stable and unstable foliations. In the former case, $\sigma^i$ fixes every quadrilateral, hence is the identity.

This shows that $\langle \sigma \rangle$ acts freely on $S^\circ$. To show the second part of the statement, recall that every puncture of $S^\circ / \langle \sigma \rangle$ must be incident to exactly two leaves of the stable foliation. In other words, every puncture of $S^\circ / \langle \sigma \rangle$ must be incident to exactly four quadrilaterals. From a Poincare-Hopf argument, we deduce that the surface is of genus one.
\end{proof}

For the rest of this subsection, we work in the setting of \Cref{lemma:anosovtypequotientsurface}.
We write $T$ for $S^\circ / \langle \sigma \rangle$.
Observe that $\alpha$ and $\beta$ descend to multicurves $\alpha_T$ and $\beta_T$ on $T$ respectively, so that $\alpha_T \cup \beta_T$ fills $T$. We pick some choice of orientations on these curves and consider their union as the $1$-skeleton of $T$.
Similarly, the dual ideal quadrangulation of $\alpha \cup \beta$ on $S^\circ$ descends to the dual ideal quadrangulation of $\alpha_T \cup \beta_T$ on $T$.

The mapping torus $T_{\sigma}$ can be identified with $T \times S^1$. We explain one way of seeing this identification explicitly. 

The cover $S^\circ \to T$ corresponds to a homomorphism $\pi_1(T) \to \mathbb{Z}/N$, which can be interpreted as an element $[z] \in H^1(T; \mathbb{Z}/N)$. We pick a representative cocycle $z$.

Take the collection of quadrilaterals
$$\left\{q \times \left\{ \frac{w}{N} \right\} \mid \text{$q$ is a quadrilateral in $T$}, w=0,...,N-1 \right\}$$
in $T \times S^1$. For every edge $e$ in the $1$-skeleton of $T$, say oriented from quadrilateral $q_1$ to quadrilateral $q_2$, we connect each $q_1 \times \{\frac{w}{N}\}$ to $q_2 \times \{\frac{w+z(e)}{N}\}$ across $f_e \times S^1$, where $f_e$ is the edge of the quadrangulation on $T$ dual to $e$. See \Cref{fig:quadcochain} for a local example.

\begin{figure}
    \centering
    \fontsize{10pt}{10pt}\selectfont
    %% Creator: Inkscape 1.3.2 (091e20e, 2023-11-25, custom), www.inkscape.org
%% Accompanies image file 'quadcochain.pdf' (pdf, eps, ps)
\begingroup%
  \makeatletter%
  \providecommand\color[2][]{%
    \errmessage{(Inkscape) Color is used for the text in Inkscape, but the package 'color.sty' is not loaded}%
    \renewcommand\color[2][]{}%
  }%
  \providecommand\transparent[1]{%
    \errmessage{(Inkscape) Transparency is used (non-zero) for the text in Inkscape, but the package 'transparent.sty' is not loaded}%
    \renewcommand\transparent[1]{}%
  }%
  \providecommand\rotatebox[2]{#2}%
  \newcommand*\fsize{\dimexpr\f@size pt\relax}%
  \newcommand*\lineheight[1]{\fontsize{\fsize}{#1\fsize}\selectfont}%
  \ifx\svgwidth\undefined%
    \setlength{\unitlength}{188.12709298bp}%
    \ifx\svgscale\undefined%
      \relax%
    \else%
      \setlength{\unitlength}{\unitlength * \real{\svgscale}}%
    \fi%
  \else%
    \setlength{\unitlength}{\svgwidth}%
  \fi%
  \global\let\svgwidth\undefined%
  \global\let\svgscale\undefined%
  \makeatother%
  \begin{picture}(1,0.84839042)%
    \lineheight{1}%
    \setlength\tabcolsep{0pt}%
    \put(0,0){\includegraphics[width=\unitlength,page=1]{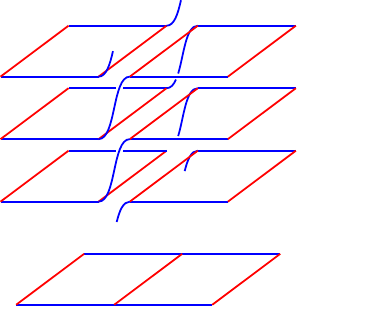}}%
    \put(0.83264328,0.52570409){\color[rgb]{0,0,0}\makebox(0,0)[lt]{\lineheight{1.25}\smash{\begin{tabular}[t]{l}$T \times S^1$\end{tabular}}}}%
    \put(0.86964547,0.11541266){\color[rgb]{0,0,0}\makebox(0,0)[lt]{\lineheight{1.25}\smash{\begin{tabular}[t]{l}$T$\end{tabular}}}}%
    \put(0.50248404,0.10058463){\color[rgb]{0,0,0}\makebox(0,0)[lt]{\lineheight{1.25}\smash{\begin{tabular}[t]{l}$q_2$\end{tabular}}}}%
    \put(0.19944857,0.1005898){\color[rgb]{0,0,0}\makebox(0,0)[lt]{\lineheight{1.25}\smash{\begin{tabular}[t]{l}$q_1$\end{tabular}}}}%
    \put(0.33482787,0.15872108){\color[rgb]{0,0,0}\makebox(0,0)[lt]{\lineheight{1.25}\smash{\begin{tabular}[t]{l}$e$\end{tabular}}}}%
    \put(0.21022377,0.00427733){\color[rgb]{0,0,0}\makebox(0,0)[lt]{\lineheight{1.25}\smash{\begin{tabular}[t]{l}$z(e)=1$\end{tabular}}}}%
    \put(0,0){\includegraphics[width=\unitlength,page=2]{quadcochain.pdf}}%
  \end{picture}%
\endgroup%

    \caption{Constructing a copy of $S^\circ$ in $T \times S^1$ by connecting up quadrilaterals across edges of the quadrangulation.}
    \label{fig:quadcochain}
\end{figure}

Then the resulting surface can be identified with $S^\circ$, and the map sending each $q \times \{\frac{w}{N}\}$ to $q \times \{\frac{w+1}{N}\}$ can be identified with $\sigma$. 

We now discuss how to add in the terms $\tau^{n_s}_\alpha$ and $\tau^{-m_s}_\beta$. Consider the mapping torus $T_{\sigma}$ with the fibers $S^\circ \times \{t\}$. For each $s=1,...,k$, we perform Dehn surgery along $\alpha_i \times \{\frac{3s-2}{3k}\}$ with coefficient $\frac{1}{(n_s)_i}$ and Dehn surgery along $\beta_j \times \{\frac{3s-1}{3k}\}$ with coefficient $-\frac{1}{(m_s)_j}$. This transforms the $3$-manifold from the mapping torus of $\sigma$ to the mapping torus of $\sigma \tau^{n_1}_\alpha \tau^{-m_1}_\beta \dots \tau^{n_k}_\alpha \tau^{-m_k}_\beta$.

The curves $\alpha_i \times \{\frac{3s-2}{3k}\}$ that we do positive Dehn surgery along can be described as curves in $T \times S^1$ in the following way:
Take the collection of arcs 
$$\left\{ a \times \left\{ \frac{3kw+3s-2}{3kN} \right\} \mid \text{$a$ connects the negative sides of quadrilateral $q$}, w=0,...,N-1 \right\}$$
in $T \times S^1$. For every edge $e$ in the $1$-skeleton of $T$ lying along the image of $\alpha$, say oriented from $q_1$ to quadrilateral $q_2$, we connect each $a_1 \times \{\frac{3kw+3s-2}{3kN}\}$ to $a_2 \times \{\frac{3k(w+z(e))+3s-2}{3kN}\}$ across $f_e \times S^1$.
The curves $\beta_j \times \{\frac{3s-1}{3k}\}$ admits a symmetric description.

This construction can be carried out in the reverse direction in the following sense: For $[z] \in H^1(T; \mathbb{Z}/N)$, we let $S^\circ_{[z]}$ be the $N$-sheeted cover of $T$ corresponding to $[z]$. The curves $\alpha_T$ and $\beta_T$ pull back to curves $\alpha_{[z]}$ and $\beta_{[z]}$ on $S^\circ_{[z]}$.

By picking a cocycle $z$ representing $[z]$, we can construct a copy of $S^\circ_{[z]}$ in $T \times S^1$ by connecting up the quadrilaterals $q \times \{\frac{w}{N}\}$ as above. Let $\sigma_{[z]}$ be the first return map on $S^\circ_{[z]}$ induced by the suspension flow on $T \times S^1$. Then $\sigma_{[z]}$ is a finite order homeomorphism on $S^\circ_{[z]}$ such that $S^\circ_{[z]}/ \langle \sigma_{[z]} \rangle = T$. Here we caution that $S^\circ_{[z]}$ may be disconnected.

For choices of $n_s$ and $m_s$ satisfying the assumption of \Cref{thm:pennerpA}, $\sigma_{[z]} \tau^{n_1}_{\alpha_{[z]}} \tau^{-m_1}_{\beta_{[z]}} \dots \tau^{n_k}_{\alpha_{[z]}} \tau^{-m_k}_{\beta_{[z]}}$ is isotopic to an Anosov type fully-punctured pseudo-Anosov map $f_{[z]}$ with orientable stable and unstable foliations, and $f_{[z]}$ preserves these orientations.

We record the discussion in this subsection as the following proposition.

\begin{prop} \label{prop:torusbundletopennermaptori}
Suppose $\sigma \tau^{n_1}_\alpha \tau^{-m_1}_\beta \dots \tau^{n_k}_\alpha \tau^{-m_k}_\beta$ is isotopic to an Anosov type fully-punctured pseudo-Anosov map $f$ with orientable stable and unstable foliations and that $f$ preserves these orientations.
Then the cover $S^\circ \to T = S^\circ/\langle \sigma \rangle$ is determined by some $[z] \in H^1(T; \mathbb{Z}/N)$.
There exists disjoint curves $\alpha_i \times \{\frac{3s-2}{3k}\}$ and $\beta_j \times \{\frac{3s-1}{3k}\}$ on $T \times S^1$, such that performing $\frac{1}{(n_s)_i}$ horizontal surgery on each $\alpha_i \times \{\frac{3s-2}{3k}\}$ and performing $-\frac{1}{(m_s)_j}$ horizontal surgery on each $\beta_j \times \{\frac{3s-1}{3k}\}$ gives $T_f$.

Conversely, given $[z] \in H^1(T; \mathbb{Z}/N)$, there exists a map $\sigma_{[z]} \tau^{n_1}_{\alpha_{[z]}} \tau^{-m_1}_{\beta_{[z]}} \dots \tau^{n_k}_{\alpha_{[z]}} \tau^{-m_k}_{\beta_{[z]}}$, defined on the covering surface $S^\circ_{[z]}$ determined by $[z]$, which is isotopic to an Anosov type fully-punctured pseudo-Anosov map $f_{[z]}$ with orientable stable and unstable foliations, where $f_{[z]}$ preserves these orientations.
\end{prop}

\section{Veering triangulations} \label{sec:vt}

In this section, we review some definitions and constructions in the subject of veering triangulations.

\subsection{Definition of veering triangulations} \label{subsec:vtdefn}

An \textbf{ideal tetrahedron} is a tetrahedon with its 4 vertices removed. The removed vertices are called the \textbf{ideal vertices}. 
An \textbf{ideal triangulation} of a $3$-manifold $M$ is a decomposition of $M$ into finitely many ideal tetrahedra glued along pairs of faces.

A \textbf{taut structure} on an ideal triangulation is a labelling of the dihedral angles by $0$ or $\pi$, such that 
\begin{itemize}
    \item Each tetrahedron has exactly two dihedral angles labelled $\pi$, and they are opposite to each other.
    \item The angle sum around each edge in the triangulation is $2\pi$.
\end{itemize}

A \textbf{transverse taut structure} is a taut structure along with a coorientation on each face, such that for any edge labelled $0$ in a tetrahedron, exactly one of the faces adjacent to it is cooriented inwards.

A \textbf{transverse taut ideal triangulation} is an ideal triangulation with a transverse taut structure.

\begin{defn} \label{defn:vt}
A \textbf{veering triangulation} is a transverse taut ideal triangulation with a coloring of the edges by red or blue, so that going counterclockwise around the four $0$-labelled edges, starting from an endpoint of a $\pi$-labelled edge, the edges are colored red, blue, red, blue, in that order.
See \Cref{fig:veertet}.
\end{defn}

\begin{figure} 
    \centering
    \fontsize{10pt}{10pt}\selectfont
    %% Creator: Inkscape 1.3.2 (091e20e, 2023-11-25, custom), www.inkscape.org
%% Accompanies image file 'veertet.pdf' (pdf, eps, ps)
\begingroup%
  \makeatletter%
  \providecommand\color[2][]{%
    \errmessage{(Inkscape) Color is used for the text in Inkscape, but the package 'color.sty' is not loaded}%
    \renewcommand\color[2][]{}%
  }%
  \providecommand\transparent[1]{%
    \errmessage{(Inkscape) Transparency is used (non-zero) for the text in Inkscape, but the package 'transparent.sty' is not loaded}%
    \renewcommand\transparent[1]{}%
  }%
  \providecommand\rotatebox[2]{#2}%
  \newcommand*\fsize{\dimexpr\f@size pt\relax}%
  \newcommand*\lineheight[1]{\fontsize{\fsize}{#1\fsize}\selectfont}%
  \ifx\svgwidth\undefined%
    \setlength{\unitlength}{119.33702123bp}%
    \ifx\svgscale\undefined%
      \relax%
    \else%
      \setlength{\unitlength}{\unitlength * \real{\svgscale}}%
    \fi%
  \else%
    \setlength{\unitlength}{\svgwidth}%
  \fi%
  \global\let\svgwidth\undefined%
  \global\let\svgscale\undefined%
  \makeatother%
  \begin{picture}(1,0.71259849)%
    \lineheight{1}%
    \setlength\tabcolsep{0pt}%
    \put(0,0){\includegraphics[width=\unitlength,page=1]{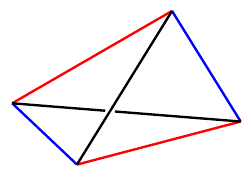}}%
    \put(0.2817407,0.49923989){\color[rgb]{1,0,0}\makebox(0,0)[lt]{\lineheight{1.25}\smash{\begin{tabular}[t]{l}0\end{tabular}}}}%
    \put(0.86849956,0.46395339){\color[rgb]{0,0,1}\makebox(0,0)[lt]{\lineheight{1.25}\smash{\begin{tabular}[t]{l}0\end{tabular}}}}%
    \put(0.62996822,0.03820694){\color[rgb]{1,0,0}\makebox(0,0)[lt]{\lineheight{1.25}\smash{\begin{tabular}[t]{l}0\end{tabular}}}}%
    \put(0.11462107,0.08220049){\color[rgb]{0,0,1}\makebox(0,0)[lt]{\lineheight{1.25}\smash{\begin{tabular}[t]{l}0\end{tabular}}}}%
    \put(0,0){\includegraphics[width=\unitlength,page=2]{veertet.pdf}}%
    \put(0.43330668,0.3564665){\color[rgb]{0,0,0}\makebox(0,0)[lt]{\lineheight{1.25}\smash{\begin{tabular}[t]{l}$\pi$\end{tabular}}}}%
    \put(0.50144876,0.19259692){\color[rgb]{0,0,0}\makebox(0,0)[lt]{\lineheight{1.25}\smash{\begin{tabular}[t]{l}$\pi$\end{tabular}}}}%
  \end{picture}%
\endgroup%

    \caption{A tetrahedron in a transverse veering triangulation. There are no restrictions on the colors of the top and bottom edges.} 
    \label{fig:veertet}
\end{figure}

The edge colorings impose strong combinatorial restrictions on a veering triangulation. We state two instances of this which will come into play in \Cref{sec:uniqueflow}.

\begin{prop} \label{prop:vtsideedge}
Let $e$ be an edge of a veering triangulation $\Delta$. Then $e$ is the top edge of exactly one tetrahedron, the bottom edge of exactly one tetrahedron, and the side edge of at least one tetrahedron on either side of $e$.
\end{prop}

To state the second proposition, we have to set up some notation:
Let $v$ be an ideal vertex of $\Delta$. We write $\partial_v \Delta$ for the link of $v$. This is a triangulated surface, where the vertices/edges/faces of $\partial_v \Delta$ correspond to vertices of edges/faces/tetrahedra of $\Delta$ at $v$ respectively. In particular, each vertex of $\partial \Delta$ inherits the color of the corresponding edge of $\Delta$, and each edge of $\partial \Delta$ inherits the coorientation of the corresponding face of $\Delta$.

\begin{prop} \label{prop:vtboundarytriang}
Let $\Delta$ be a veering triangulation and let $v$ be an ideal vertex of $\Delta$. There exist edge paths $\{l_i\}$ in $\partial_v \Delta$ such that:
\begin{itemize}
    \item The vertices along $l_{2i}$ are all colored blue, while the vertices along $l_{2i+1}$ are all colored red.
    \item The faces between $l_{2i}$ and $l_{2i+1}$ form a stack of upward pointing triangles, while the faces between $l_{2i}$ and $l_{2i+1}$ form a stack of downward pointing triangles.
\end{itemize}
See \Cref{fig:vtboundarytriang}.
\end{prop}

\begin{figure}
    \centering
    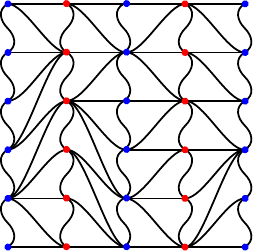
    \caption{The boundary triangulation at a vertex of a veering triangulation.}
    \label{fig:vtboundarytriang}
\end{figure}

We refer to the paths $\{l_{2i}\}$ as the \textbf{blue ladderpole curves} and the paths $\{l_{2i+1}\}$ as the \textbf{red ladderpoole curves} at $v$. We refer to the homotopy class of the union of curves $\bigcup l_{2i}$ as the \textbf{ladderpole class} of $\Delta$ at $v$.

We refer the reader to \cite[Section 2]{FG13} for proofs of \Cref{prop:vtsideedge} and \Cref{prop:vtboundarytriang}.

\subsection{(Almost) veering branched surfaces} \label{subsec:vbsdefn}

In this subsection, we recall the definition of veering branched surfaces. This is, in some sense, the dual notion to veering triangulations. One utility of veering branched surfaces over veering triangulations, however, is that the former generalizes to almost veering branched surfaces. This generalization will play a role in the proof of \Cref{thm:pennergenusone}.

\begin{defn} \label{defn:branchsurf}
Let $M$ be a 3-manifold. A \textbf{branched surface} $B$ in $M$ is an embedded finite $2$-complex smoothly modelled after one of the pictures in \Cref{fig:bslocal}.

\begin{figure}
    \centering
    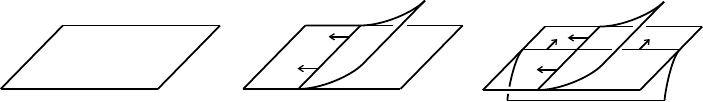
    \caption{The local models for branched surfaces. The arrows indicate the maw coorientation of the branch locus.}
    \label{fig:bslocal}
\end{figure}

The set of points where $B$ is locally of the form of \Cref{fig:bslocal} middle or right is called the \textbf{branch locus} of $B$ and is denoted by $\brloc(B)$. The points where $B$ is locally of the form of \Cref{fig:bslocal} right are called the \textbf{triple points} of $B$. The complementary regions of $\brloc(B)$ in $B$ are called the \textbf{sectors} of $B$. 

The branch locus $\brloc(B)$ is a union of smoothly embedded circles, which we refer to as the \textbf{branch loops}. Each branch loop has a canonical coorientation on $B$, which we call the \textbf{maw coorientation}, given locally by the direction from the side with more sectors to the side with less sectors. See the arrows in \Cref{fig:bslocal}.

The sectors of $B$ are surfaces with boundary, with corners at where the boundary locally switches from lying along one branch loop to another. 
\end{defn}

Let $C$ be a complementary region of a branched surface $B$. The boundary of $C$ is smooth away from a finite collection of \textbf{cusp circles}, where $\partial C$ is attached onto the branch loops of $B$. We say that $C$ is a \textbf{cusped torus shell} if $C$ is homeomorphic to a $T^2 \times [0,\infty)$ and the collection of cusp circles on its boundary is nonempty.

\begin{defn} \label{defn:vbs}
A branched surface $B$ in a $3$-manifold $M$ along with a choice of orientations on its branch loops is \textbf{veering} if:
\begin{enumerate}[label=(\roman*)]
    \item Each sector of $B$ is a disc with four corners.
    \item Each component of $M \cut B$ is a cusped torus shell.
    \item At each triple point, the orientation of each branch loop induces the opposite of the maw coorientation on the other branch loop.
\end{enumerate}

A branched surface $B$ in a $3$-manifold $M$ along with a choice of orientations on its branch loops is \textbf{almost veering} if:
\begin{enumerate}[label=(\roman*)]
    \item Each sector of $B$ is a disc with four corners, or an annulus or a Mobius band without corners.
    \item Each component of $M \cut B$ is a cusped torus shell.
    \item At each triple point, the orientation of each branch loop induces the opposite of the maw coorientation on the other branch loop.
\end{enumerate}
\end{defn}

\begin{rmk}
\Cref{defn:vbs} is modified from \cite[Definitions 3.1 and 3.5]{Tsa22a} in two ways. Firstly, we have removed the possibility for components of $M \cut B$ to be cusped solid tori. This is because such components will not come up in this paper, and in doing so we can simplify certain statements in this paper, such as \Cref{prop:vtdualvbs} below.

Secondly, we have reversed the orientation condition at each triple point. This accounts for the fact that \cite{Tsa22a} uses the `flow goes downwards' convention while this paper uses the `flow goes upwards' convention.
\end{rmk}

\begin{prop} \label{prop:vtdualvbs}
Let $M$ be an oriented $3$-manifold. For every veering triangulation $\Delta$ on $M$, the dual $2$-complex $B$ to $\Delta$ can be given the structure of a veering branched surface.

Conversely, for every veering branched surface $B$ on $M$, the dual ideal triangulation $\Delta$ to $M$ can be given the structure of a veering triangulation.
\end{prop}
\begin{proof}[Sketch of proof]
This is shown in \cite[Propositions 2.5 and 3.2]{Tsa22a}. For completeness, we provide a sketch of proof.

The heart of the proof is \Cref{fig:vtdualvbs}. Given a veering triangulation $\Delta$, we can define a branched surface structure on the dual $2$-complex to $\Delta$ according to the figure. Conversely, given a veering branched surface $B$, we can define a veering structure on the dual ideal triangulation to $B$ according to the figure.

\begin{figure}
    \centering
    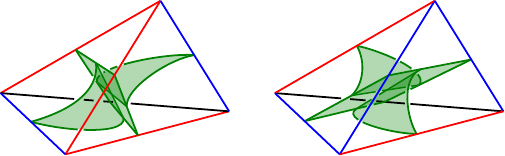
    \caption{Key for duality between veering triangulations and veering branched surfaces.}
    \label{fig:vtdualvbs}
\end{figure}

A useful mnemonic for \Cref{fig:vtdualvbs} is that an edge is b\textbf{L}ue or \textbf{R}ed if the `fins' at the bottom of the dual sector spiral in a \textbf{L}eft- or \textbf{R}ight-handed manner, respectively.
\end{proof}

In the context of \Cref{prop:vtdualvbs}, we say that $\Delta$ is the \textbf{dual veering triangulation} to $B$, and conversely, $B$ is the \textbf{dual veering branched surface} to $\Delta$.

\subsection{Horizontal surgery} \label{subsec:vbshsurdefn}

In this subsection, we recall the horizontal surgery operation on almost veering branched surfaces.

We first define the types of curves we perform the surgery operation along. To do so we have to set up some terminology.

Let $B$ be a branched surface in an oriented 3-manifold $M$. Let $c \subset B$ be a smoothly embedded curve which avoids the triple points of $B$.
$c$ is said to be \textbf{orientation preserving} if the tangent planes of $B$ along $c$ can be oriented in a coherent way.

Let $N$ be a small tubular neighborhood of $c$ in $M$. When $c$ is orientation preserving, for $N$ small enough, $N \cap B$ is an annulus $A$ with sectors attached along disjoint arcs, with each arc corresponding to a point of intersection between $c$ and $\brloc(B)$. We call $A$ a \textbf{smooth neighborhood} of $c$ in $B$. We refer to the two components of $N \cut A$ as the \textbf{regular half-neighborhoods} of $\alpha$ in $M$.

\begin{defn} \label{defn:vbshsurcurve}
Let $B$ be an almost veering branched surface in an oriented 3-manifold $M$. Let $c \subset B$ be a smoothly embedded curve which avoids the triple points of $B$. Suppose $c$ is orientation preserving. Let $N$ be a tubular neighborhood of $c$ in $M$, let $A$ be a smooth neighborhood of $\alpha$ in $B$, and let $N_1$ and $N_2$ be the regular half-neighborhoods of $c$ in $M$.
We say that $c$ is a \textbf{horizontal surgery curve} if:
\begin{enumerate} 
    \item The arcs in $\brloc(N \cap B)$ are all oriented from the same boundary component of $A$ to the other.
    \item The arcs in each $\brloc(N_i  \cap B)$ are cooriented in the same direction, but the two directions for the two regular half-neighborhoods are opposite to each other. 
\end{enumerate}

Recall that $M$ is oriented. Observe that for each $i$, 
$$(\text{coorientation of arcs in $\brloc(N_i \cap B)$}, \text{orientation of arcs in $\brloc(N_i \cap B)$})$$
determines an orientation of $A$. If the induced coorientation on $A$ points into $N_i$, we say that $c$ is \textbf{positive}. If the induced coorientation on $A$ points out of $N_i$, we say that $c$ is \textbf{negative}. 

See \Cref{fig:vbshsur} top for an illustration of a positive and a negative horizontal surgery curve.
\end{defn}

\begin{figure}
    \centering
    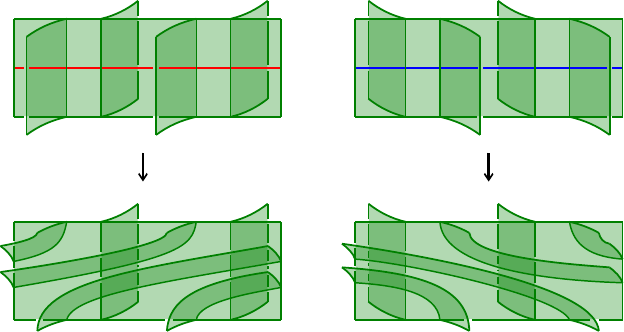
    \caption{Left/Right: Performing horizontal surgery on a veering branched surface along a positive/negative horizontal surgery curve, respectively.}
    \label{fig:vbshsur}
\end{figure}

\begin{constr}
Suppose $c$ is a positive/negative horizontal surgery curve as in \Cref{defn:vbshsurcurve}. 
Fix one of the regular half-neighborhoods $N_i$.
Let $\mu \subset \partial N_i$ be a meridian curve. We orient $c$ so that $(\mu,c)$ is a positive basis of $H_1(\partial N_i; \mathbb{R})$.

For each positive/negative integer $n$, respectively, we cut $N_i$ out of $M$, and glue it back with a map that is identity on $\partial N_i \backslash A$ and sends the meridian to a curve of isotopy class $\mu+nc$, such that the arcs in the branch locus of $N_i \cap B$ intersect those in $N_{i+1} \cap B$ minimally. See \Cref{fig:vbshsur}. This gives us a branched surface $B_{\frac{1}{n}}(c)$ in the 3-manifold $M_{\frac{1}{n}}(c)$ obtained by doing $\frac{1}{n}$ surgery along $c$ on $M$ (with respect to the basis we chose above). We call this operation a \textbf{$\frac{1}{n}$ horizontal surgery on $B$ along $c$}.

We draw attention to the fact that there are two possible ways to perform this surgery: One can cut and reglue either $N_1$ or $N_2$. If one only performs a single surgery operation, or more generally, if one performs surgery along disjoint horizontal surgery curves, then the isotopy class of the resulting branched surface is independent of these choices. However, this point will come up in the proof of \Cref{thm:pennergenusone}, when we have to perform surgery along intersecting horizontal surgery curves.
\end{constr}

\begin{rmk}
Our notation here differs slightly from \cite{Tsa22a}. In \cite{Tsa22a} we describe the surgery coefficient using an orientation on $N$ derived from the data of $c$. However, since we deal with oriented $3$-manifolds in this paper, it is more natural to compute the surgery coefficient with respect to the ambient orientation instead, which accounts for the difference in sign in the case when $c$ is positive.
\end{rmk}

\begin{prop}[{\cite[Proposition 4.4]{Tsa22a}}] \label{prop:vbshsur}
Let $c$ be a positive/negative horizontal surgery curve on an almost veering branched surface $B$. Let $n$ be a positive/negative integer, respectively. Then the branched surface $B_{\frac{1}{n}}(c)$ is almost veering.

Furthermore, if $B$ is veering then $B_{\frac{1}{n}}(c)$ is veering as well.
\end{prop}

In the setting of the second statement of \Cref{prop:vbshsur}, if $\Delta$ is the veering triangulation dual to $B$, we will say that the veering triangulation dual to $B_{\frac{1}{n}}(c)$ is obtained by \textbf{$\frac{1}{n}$ horizontal surgery on $\Delta$ along $c$} and denote it as $\Delta_{\frac{1}{n}}(c)$.

We will need some rough understanding of how the triangulation is modified from $\Delta$ to $\Delta_{\frac{1}{n}}(c)$. Notice that the sequence of edges of $\brloc(B)$ that $c$ passes through dualizes to a sequence of faces that connect up to an annulus $A$ (with punctures on its boundary) carried by the $2$-skeleton of $\Delta$. We say that $A$ is the \textbf{dual annulus} to $c$. We illustrate one possibility for the dual annulus to the positive surgery curve in \Cref{fig:vbshsur} top left in \Cref{fig:dualannulus}.

\begin{figure}
    \centering
    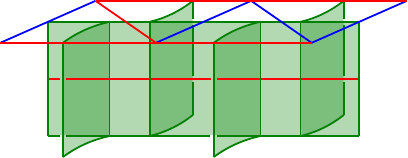
    \caption{One possibility for the dual annulus to the positive surgery curve in \Cref{fig:vbshsur} top left.}
    \label{fig:dualannulus}
\end{figure}

Without loss of generality suppose that $c$ is positive. Using \Cref{fig:vtdualvbs}, one can check that the boundary of $A$ always consists of red edges of $\Delta$. On the other hand, the edges in the interior of $A$ can be of either color. If any red edges come up however, the local structure of $\Delta$ is uniquely determined in the following sense.

\begin{lemma} \label{lemma:vbshsurdualannuluswrongcolor}
Let $c$ be a positive horizontal surgery curve of a veering branched surface $B$. Let $A$ be the dual annulus of $c$. If there is a red edge $e$ in the interior of $A$, then the two faces in $A$ containing $e$ is the pair of bottom faces of a tetrahedron that has a blue top edge.
\end{lemma}
\begin{proof}
Each edge in the interior of $A$ corresponds to a segment of $c$ between two adjacent intersection points between $c$ and $\brloc(B)$. For the segment $s$ of $c$ corresponding to $e$, the maw coorientation must point into $s$ at both endpoints. Meanwhile, the segments adjacent to $s$ correspond to blue edges. 

Consider the top corner $v$ of the sector which $s$ lies on.
The two faces in $A$ containing $e$ must be the pair of bottom faces of the tetrahedron dual to $v$. Moreover, $v$ must be of the form of \Cref{fig:vtdualvbs} right, hence the top edge of the tetrahedron must be blue.
\end{proof}

We can isotope $A$ by moving each pair of faces as in \Cref{lemma:vbshsurdualannuluswrongcolor} across the tetrahedron so that they now lie on the pair of top faces. This ensures that every edge in the interior of $A$ is blue. 

This corresponds to isotoping the horizontal surgery curve $c$ across a triple point of $B$. It is straightforward to check that the branched surfaces obtained by performing horizontal surgery on the original and isotoped curve are isotopic. See \Cref{fig:vbshsurcurvemovecorrectcolor} for an example. Hence we can assume that all the interior edges in $A$ are blue.

\begin{figure}
    \centering
    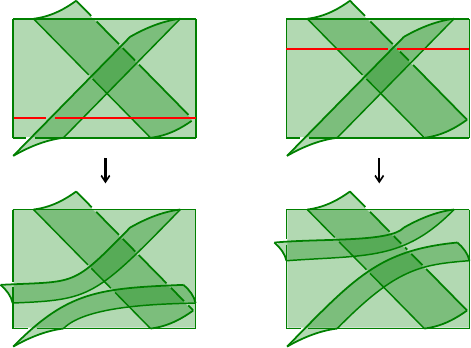
    \caption{Isotoping the horizontal surgery curve so that all interior edges of the dual annulus are blue.}
    \label{fig:vbshsurcurvemovecorrectcolor}
\end{figure}

With this arranged, we make the following technical observation, which will play a role in the proof of \Cref{thm:hsurcorr}.

\begin{lemma} \label{lemma:vthsurtopbottomfaces}
Let $c$ be a positive horizontal surgery curve of a veering branched surface $B$. Let $A$ be the dual annulus of $c$. Suppose that $A$ only has blue interior edges. Let $v$ be a vertex of $A$. Suppose there is more than one interior edges of $A$ with an endpoint at $v$. We order these edges $e_1,...,e_n$ from left to right, looking into $A$. 

Then for every $i=1,...,n-1$, the face between $e_i$ and $e_{i+1}$ is a bottom face of the tetrahedron that has $e_i$ as its bottom edge. Symmetrically, for every $i=2,...,n$, the face between $e_{i-1}$ and $e_i$ is a top face of the tetrahedron that has $e_i$ as its top edge.
\end{lemma}

We can now proceed with the surgery. We can interpret the surgery on $B$ as cutting out $N \cap B$, where $N$ is a tubular neighborhood of $c$, and gluing in a different branched surface $N \cap B_{\frac{1}{n}}(c)$. From this perspective, \Cref{fig:vbshsur} can be taken to illustrate an example of the cut branched surface $N \cap B$ and corresponding reglued branched surface $N \cap B_{\frac{1}{n}}(c)$.

The corresponding operation on $\Delta$ is that we cut along $A$, obtaining a complex with two embedded copies of $A$ --- $A_+$ on the positive side and $A_-$ on the negative side, then glue in an (ideally) triangulated solid torus $T$. We illustrate this in \Cref{fig:vthsur}.

\begin{figure}
    \centering
    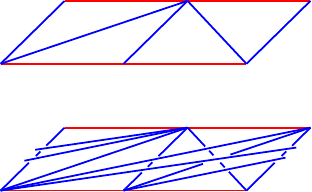
    \caption{Horizontal surgery on the level of veering triangulations involves cutting along the dual annulus $A$ and gluing in an (ideally) triangulated solid torus.}
    \label{fig:vthsur}
\end{figure}

We denote the subsurface of $\partial T$ that is identified with $A_\pm$ by $\partial_\pm T$, respectively, and call the edges in $\partial_+ T \cap \partial_- T$ the \textbf{side} edges of $T$.

\begin{defn} \label{defn:layered}
We say that an ideal triangulation of a $3$-manifold $M$ with finitely many tetrahedra is \textbf{layered} if it can built by placing tetrahedra one at a time, with certain bottom faces of each new tetrahedron placed on top of certain top faces of existing tetrahedra, and finally, if $M$ has no boundary, gluing the remaining top faces to the remaining bottom faces once all the tetrahedra are placed. 
\end{defn}

We claim that the triangulation on $T$ is layered.
This can be seen by observing that $\brloc(N \cap B_{\frac{1}{n}}(c))$ is a union of arcs lying on a smooth neighborhood of $c$, and each tetrahedra in $T$ is dual to an intersection of these arcs. Thus we can build $T$ by inductively placing a tetrahedron corresponding to a lowermost intersection point.

We record this discussion as the following proposition.

\begin{prop} \label{prop:vthsur}
Let $c$ be a positive/negative horizontal surgery curve on a veering branched surface $B$. Let $n$ be a positive/negative integer, respectively. Let $A$ be the dual annulus to $c$. Up to isotoping $c$ we can assume that all the edges in the interior of $A$ are blue/red, respectively. 
Then $\Delta_{\frac{1}{n}}(c)$ can be obtained from $\Delta$ by cutting along the dual annulus $A$ and inserting a layered (ideally) triangulated solid torus $T$.
\end{prop}

\subsection{Penner type mapping tori revisited} \label{subsec:vbspennermaptori}

Equipped with the language of (almost) veering branched surfaces, we revisit the material in \Cref{subsec:pennermaptori}. Recall the setting: We have a pair of multicurves $(\alpha,\beta)$ on a surface $S^\circ$ which induces a dual ideal quadrangulation. We suppose that $\sigma \tau^{n_1}_\alpha \tau^{-m_1}_\beta \dots \tau^{n_k}_\alpha \tau^{-m_k}_\beta$ is isotopic to an Anosov type fully-punctured pseudo-Anosov map $f$ with orientable stable and unstable foliations and that $f$ preserves these orientations. Then \Cref{lemma:anosovtypequotientsurface} states that $S^\circ$ covers a genus one surface $T$.

The quadrangulation on $S^\circ$ descends down to a quadrangulation on $T$. We define a train track $\tau$ on $T$ by applying \Cref{fig:quadtt} to each quadrilateral, i.e. we place a train track with two cusps facing the two vertices that meets a negative side on its left and a positive side on its right.

\begin{figure}
    \centering
    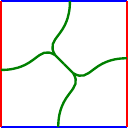
    \caption{Inside each quadrilateral, we place a train track as above.}
    \label{fig:quadtt}
\end{figure}

Then $\tau \times S^1$ becomes a branched surface $B_\sigma$ in $T \times S^1$. All the sectors of $B_\sigma$ are annuli, hence if we orient its branch loops upwards, i.e. in the direction of the suspension flow, then it is easy to check that $B_\sigma$ becomes an almost veering branched surface.

Meanwhile, $\tau$ pulls back to a train track $\tau_S$ on $S^\circ$. We can also obtain $\tau_S$ by intersecting the copy of $S^\circ$ constructed in $T \times S^1$ in \Cref{subsec:pennermaptori} with $B_\sigma$. In other words, under our identification of $T \times S^1$ with $T_\sigma$, $B_\sigma$ can be identified with the suspension of $\tau_S$.

Notice that we can consider $\alpha$ and $\beta$ as subsets of $\tau_S$. This way, we can consider $\alpha \times \{t\}, \beta \times \{t\} \subset T_\sigma$ to be curves lying on $B_\sigma$, for each $t$. Observe that these curves are positive and negative horizontal surgery curves respectively. See \Cref{fig:quadvbs} for a local illustration.

\begin{figure}
    \centering
    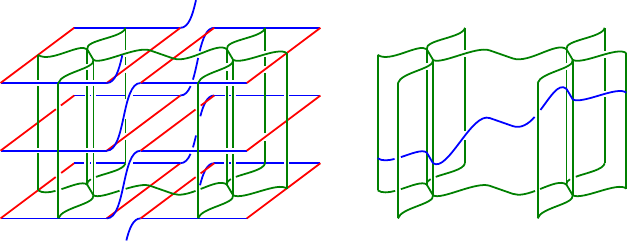
    \caption{Left: The branched surface $B_\sigma$. Right: For each $t$, $\beta \times \{t\}$ is a negative horizontal surgery curve lying on $B_\sigma$.}
    \label{fig:quadvbs}
\end{figure}

In particular, when we perform positive/negative Dehn surgeries along the curves $\alpha_i \times \{\frac{3s-2}{3k}\}$/$\beta_j \times \{\frac{3s-1}{3k}\}$, respectively, to obtain the mapping torus of $f$ from that of $\sigma$, we can perform positive/negative horizontal surgery of $B_\sigma$ along these curves. The end result is an almost veering branched surface $B_f$ on the mapping torus $T_f$.

We claim that $B_f$ is in fact veering. For otherwise one can construct a properly embedded annulus or Mobius band that is dual to an annulus or Mobius band sector, contradicting the fact that the mapping torus of a pseudo-Anosov map is a hyperbolic $3$-manifold. 

For future purposes, we also note that the dual veering triangulation $\Delta_f$ to $B_f$ is layered (recall \Cref{defn:layered}). This can be shown by observing that $B_\sigma$ is transverse to the suspension foliation of $T_\sigma$ by the fibers $S^\circ \times \{t\}$, with the branch loops intersecting the leaves positively. We can perform horizontal surgery on $B_\sigma$ by cutting out a neighborhood $N$ foliated as $A \times I$ and regluing. This transforms the suspension foliation on $T_\sigma$ into the suspension foliation on $T_f$. As a result, we conclude that the surgered branched surface $B_f$ is transverse to the latter, with branch loops intersecting the leaves positively.

The dual veering triangulation to $B_f$ can be constructed by sweeping across leaves of the suspension foliation. Whenever we sweep across a triple point, we add a tetrahedron. As we return to the starting leaf, we glue the remaining top faces of the tetrahedra to the remaining bottom faces.

We record the discussion in this subsection as the following proposition.

\begin{prop} \label{prop:torusbundletovbs}
Suppose $\sigma \tau^{n_1}_\alpha \tau^{-m_1}_\beta \dots \tau^{n_k}_\alpha \tau^{-m_k}_\beta$ is isotopic to an Anosov type fully-punctured pseudo-Anosov map $f$ with orientable stable and unstable foliations and that $f$ preserves these orientations. Then there is an almost veering branched surface $B_\sigma$ on $T_\sigma$. There exists positive horizontal surgery curves $\alpha_i \times \{\frac{3s-2}{3k}\}$ and negative horizontal surgery curves $\beta_j \times \{\frac{3s-1}{3k}\}$ on $B_\sigma$, such that performing $\frac{1}{(n_s)_i}$ horizontal surgery on each $\alpha_i \times \{\frac{3s-2}{3k}\}$ and performing $-\frac{1}{(m_s)_j}$ horizontal surgery on each $\beta_j \times \{\frac{3s-1}{3k}\}$ gives a veering branched surface $B_f$ on $T_f$.
Moreover the veering triangulation dual to $B_f$ is layered.
\end{prop}

\section{Correspondence between flows and triangulations} \label{sec:uniqueflow}

The main goal of this section is to prove the following theorem.

\begin{thm}[Landry-Taylor] \label{thm:uniqueflow}
Let $M$ be a closed oriented $3$-manifold and let $\mathcal{C}$ be a finite collection of curves in $M$. Let $\Delta$ be a veering triangulation on $M \backslash \mathcal{C}$. Then there is at most one pseudo-Anosov flow $\phi^t$ on $M$ such that $\Delta$ can be put in transverse position with respect to $\phi^t$, up to orbit equivalence by a map isotopic to identity.

Moreover, such a flow $\phi^t$ is transitive and the degeneracy locus of $\phi^t$ at each element of $\mathcal{C}$ is given by the ladderpole class of $\Delta$ at the corresponding ideal vertex.
\end{thm}

\Cref{thm:uniqueflow} is communicated to us by Michael Landry and Samuel Taylor. We reproduce their proof in \Cref{subsec:uniqueflowproof}. Before that, in \Cref{subsec:transrect} we define transverse position and discuss how it allows one to translate the combinatorics of veering triangulations into the dynamics of pseudo-Anosov flows.

In the setting of \Cref{thm:uniqueflow}, we will say that $\phi^t$ is the flow \textbf{corresponding to $\Delta$ on $M$}. See also \Cref{rmk:generalcorr} below.
Note that here it is important to specify the closed $3$-manifold $M$, since $\Delta$ only determines $M$ up to Dehn surgery along $\mathcal{C}$. 

If we only care about flows up to almost equivalence, however, then the following corollary states that we can get by without paying attention to the choice of $M$.

\begin{cor} \label{cor:uniqueflow}
Let $\Delta$ be a veering triangulation. For $i=1,2$, suppose $\phi^t_i$ is the pseudo-Anosov flow corresponding to $\Delta$ on a closed oriented $3$-manifold $M_i$. Then $\phi^t_1$ and $\phi^t_2$ are almost equivalent.
\end{cor}

We demonstrate the proof of \Cref{cor:uniqueflow} at the end of \Cref{subsec:uniqueflowproof}.

We remark that we will only use \Cref{thm:uniqueflow} and \Cref{cor:uniqueflow} in the case when $\phi^t$ is Anosov. We choose to state and prove \Cref{thm:uniqueflow} in the more general context of pseudo-Anosov flows in anticipation of future applications.

\begin{rmk} \label{rmk:generalcorr}
Combining work of Schleimer-Segerman and Landry-Minsky-Taylor, it is already known that under the assumptions of \Cref{thm:uniqueflow}, as long as the meridians of $M$ intersect the ladderpole class of $\Delta$ at each vertex in $\geq 2$ points, there is at least one pseudo-Anosov flow $\phi^t$ on $M$ such that $\Delta$ is in transverse position with respect to $\phi^t$. \Cref{thm:uniqueflow} thus completes this picture by implying that there is exactly one such pseudo-Anosov flow, up to orbit equivalence by a map isotopic to identity. We refer to \cite[Chapter 2]{Tsa23a} for an exposition, and refer to \cite{SS20}, \cite{SS19}, \cite{SS21}, \cite{SS23}, \cite{SSpart5}, and \cite{LMT23} for proofs.
\end{rmk}

\subsection{Transverse position and rectangles} \label{subsec:transrect}

\begin{defn} \label{defn:transpos}
Let $M$ be a closed oriented $3$-manifold and let $\mathcal{C}$ be a finite collection of curves in $M$. Let $\phi^t$ be a pseudo-Anosov flow on $M$ such that $\mathcal{C}$ is a collection of orbits of $\phi^t$. Let $\Delta$ be a veering triangulation on $M \backslash \mathcal{C}$.

We say that $\Delta$ is \textbf{in transverse position with respect to $\phi^t$} if:
\begin{itemize}
    \item each face $f$ of $\Delta$ is positively transverse to the orbits of $\phi^t$, and
    \item each edge $e$ of $\Delta$ extends to a smoothly embedded interval $\overline{e}$ that is transverse to the stable and unstable foliations of $\phi^t$.
\end{itemize}

The second condition requires some elaboration when $\phi^t$ is a pseudo-Anosov flow, since in this case the stable and unstable foliations of $\phi^t$ contain singularities. Here we mean that $\overline{e}$ lies away from the singular orbits in its interior $e$, and at its endpoints its tangent line does not lie along a leaf of the stable or unstable foliation.
\end{defn}

For the rest of this subsection, we work in the setting of \Cref{defn:transpos}.

First observe that $\mathcal{C}$ lifts to a collection of orbits of $\widetilde{\phi}^t$ in the universal cover $\widetilde{M}$, thus determines a collection of points $\widetilde{\mathcal{C}}$ in the orbit space $\mathcal{O}$. We make the following definition.

\begin{defn} \label{defn:rectangles}
A \textbf{rectangle} $R$ in $\mathcal{O}$ is a rectangle $[0,1]^2$ embedded in $\mathcal{O}$ such that the restrictions of $\mathcal{O}^s$ and $\mathcal{O}^u$ foliate the rectangle as a product, i.e. conjugate to the foliations of $[0,1]^2$ by vertical and horizontal lines, and such that no element of $\widetilde{\mathcal{C}}$ lies in the interior of $R$.

Let $R_1$ and $R_2$ be rectangles in $\mathcal{O}$. $R_2$ is said to be \textbf{taller} than $R_1$ if every leaf of $\mathcal{O}^u$ that intersects $R_1$ intersects $R_2$. $R_1$ is said to be \textbf{wider} than $R_2$ if every leaf of $\mathcal{O}^s$ that intersects $R_2$ intersects $R_1$. 
We write $R_1 < R_2$ if $R_2$ is taller than $R_1$ and $R_1$ is wider than $R_2$.

An \textbf{edge rectangle} in $\mathcal{O}$ is a rectangle with two opposite corners on $\widetilde{\mathcal{C}}$. A \textbf{face rectangle} in $\mathcal{O}$ is a rectangle with one corner on $\widetilde{\mathcal{C}}$ and the two opposite sides to the corner containing elements of $\widetilde{\mathcal{C}}$ in their interior. A \textbf{tetrahedron rectangle} in $\mathcal{O}$ is a rectangle all of whose sides contain elements of $\widetilde{\mathcal{C}}$ in their interior. See \Cref{fig:rectangles}.

\begin{figure}
    \centering
    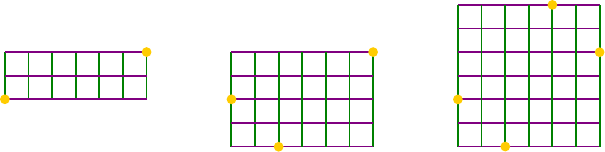
    \caption{From left to right: an edge rectangle, a face rectangle, and a tetrahedron rectangle. Yellow dots denote elements of $\widetilde{\mathcal{C}}$.}
    \label{fig:rectangles}
\end{figure}
\end{defn}

Meanwhile, $\Delta$ lifts to a triangulation on $\widetilde{M} \backslash \widetilde{\mathcal{C}}$.

\begin{lemma} \label{lemma:transrect}
For every tetrahedron $t$ of $\widetilde{\Delta}$, the projection of $t$ in the orbit space $\mathcal{O}$ lies within a tetrahedron rectangle $R_t$. See \Cref{fig:transtetrect}. In particular, for every face $f$ of $\widetilde{\Delta}$, the projection of $f$ in the orbit space $\mathcal{O}$ lies within a face rectangle $R_f$, and for every edge $e$ of $\widetilde{\Delta}$, the projection of $e$ in the orbit space $\mathcal{O}$ lies within an edge rectangle $R_e$.
\end{lemma}

\begin{figure}
    \centering
    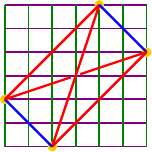
    \caption{If $\Delta$ is in transverse position with respect to $\phi^t$, then for every tetrahedron $t$ of $\widetilde{\Delta}$, the projection of $t$ in the orbit space $\mathcal{O}$ lies within a tetrahedron rectangle $R_t$.}
    \label{fig:transtetrect}
\end{figure}

\begin{proof}
Let $v$ be a vertex of $t$ that meets a blue side edge $e_L$ on its left and a red side edge $e_R$ on its right. Notice that $v$ must lie on a point $\widetilde{\gamma}$ on $\widetilde{\mathcal{C}}$. We claim that the projection of $t$ meets a half leaf of $\mathcal{O}^s$ emerging from $\widetilde{\gamma}$.

To see this, recall \Cref{prop:vtboundarytriang}. Suppose $\widetilde{\gamma}$ covers an orbit $\gamma$ of $\phi^t$. Then \Cref{prop:vtboundarytriang} implies that for some $k > 0$, there exists a sequence of faces $f_1,...,f_N$ such that
\begin{itemize}
    \item $e_L$ lies on the left of the vertex of $f_1$ at $v$,
    \item the edge that lies on the right of the vertex of $f_i$ at $v$ equals the edge that lies on the left of the vertex of $f_{i+1}$ at $v$, for each $i=1,...,N-1$, and
    \item $[\gamma]^k \cdot e_L$ lies on the right of the vertex of $f_N$ at $v$.
\end{itemize}
Namely, this sequence of faces can be obtained by following the lift of a blue ladderpole curve.

Symmetrically, there exists a sequence of faces $f'_1,...,f'_N$ such that
\begin{itemize}
    \item $e_R$ lies on the right of the vertex of $f'_1$ at $v$,
    \item the edge that lies on the left of the vertex of $f_i$ at $v$ equals the edge that lies on the right of the vertex of $f_{i+1}$ at $v$, for each $i=1,...,N-1$, and
    \item $[\gamma]^k \cdot e_R$ lies on the left of the vertex of $f'_N$ at $v$.
\end{itemize}

This implies that $[\gamma]^k$ sends the cone at $\widetilde{\gamma}$ spanned by the tangents of the projections of $\overline{e_L}$ and $\overline{e_R}$ into itself, which implies the claim. See \Cref{fig:transstablesector}.

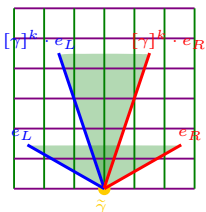
\begin{figure}
    \centering
    \fontsize{6pt}{6pt}\selectfont
    %% Creator: Inkscape 1.3.2 (091e20e, 2023-11-25, custom), www.inkscape.org
%% Accompanies image file 'transstablesector.pdf' (pdf, eps, ps)
\begingroup%
  \makeatletter%
  \providecommand\color[2][]{%
    \errmessage{(Inkscape) Color is used for the text in Inkscape, but the package 'color.sty' is not loaded}%
    \renewcommand\color[2][]{}%
  }%
  \providecommand\transparent[1]{%
    \errmessage{(Inkscape) Transparency is used (non-zero) for the text in Inkscape, but the package 'transparent.sty' is not loaded}%
    \renewcommand\transparent[1]{}%
  }%
  \providecommand\rotatebox[2]{#2}%
  \newcommand*\fsize{\dimexpr\f@size pt\relax}%
  \newcommand*\lineheight[1]{\fontsize{\fsize}{#1\fsize}\selectfont}%
  \ifx\svgwidth\undefined%
    \setlength{\unitlength}{99.80609804bp}%
    \ifx\svgscale\undefined%
      \relax%
    \else%
      \setlength{\unitlength}{\unitlength * \real{\svgscale}}%
    \fi%
  \else%
    \setlength{\unitlength}{\svgwidth}%
  \fi%
  \global\let\svgwidth\undefined%
  \global\let\svgscale\undefined%
  \makeatother%
  \begin{picture}(1,0.77832867)%
    \lineheight{1}%
    \setlength\tabcolsep{0pt}%
    \put(0,0){\includegraphics[width=\unitlength,page=1]{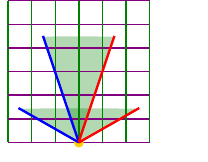}}%
    \put(0.02685237,0.29211139){\color[rgb]{0,0,1}\makebox(0,0)[lt]{\lineheight{1.25}\smash{\begin{tabular}[t]{l}$e_L$\end{tabular}}}}%
    \put(0.65824899,0.29211898){\color[rgb]{1,0,0}\makebox(0,0)[lt]{\lineheight{1.25}\smash{\begin{tabular}[t]{l}$e_R$\end{tabular}}}}%
    \put(0.48006641,0.63535347){\color[rgb]{1,0,0}\makebox(0,0)[lt]{\lineheight{1.25}\smash{\begin{tabular}[t]{l}$[\gamma]^k \cdot e_R$\end{tabular}}}}%
    \put(-0.00303324,0.63535152){\color[rgb]{0,0,1}\makebox(0,0)[lt]{\lineheight{1.25}\smash{\begin{tabular}[t]{l}$[\gamma]^k \cdot e_L$\end{tabular}}}}%
    \put(0.34554702,0.0082777){\color[rgb]{1,0.8,0}\makebox(0,0)[lt]{\lineheight{1.25}\smash{\begin{tabular}[t]{l}$\widetilde{\gamma}$\end{tabular}}}}%
  \end{picture}%
\endgroup%

    \caption{Arguing that the projection of $t$ meets a half leaf of $\mathcal{O}^s$ emerging from every vertex that meets a blue side edge $e_L$ on its left and a red side edge $e_R$ on its right.}
    \label{fig:transstablesector}
\end{figure}

Symmetrically, the projection of $t$ meets a half leaf of $\mathcal{O}^u$ emerging from every vertex that meets a red side edge on its left and a blue side edge on its right. By a Poincare-Hopf argument, the projection of $t$ meets exactly one half leaf of $\mathcal{O}^s$ or $\mathcal{O}^u$ emerging from each of its vertex. In particular, the red side edges of $t$ are positive while the blue side edges of $t$ are negative. By \Cref{prop:vtsideedge}, every edge is the side edge of some tetrahedron, so every red edge is positive while every blue edge is negative.

Observe that it remains to show that the projection of each edge $e$ spans an edge rectangle. Without loss of generality we show this when $e$ is red. Let $t$ be the tetrahedron that has $e$ as a bottom edge. Let $e'$ be a blue side edge of $t$. The union of projections of $t'$ where $t'$ ranges over tetrahedra that have $e'$ as a side edge contains a triangle with one side on the projection of $e$ and the other two sides on leaves of $\mathcal{O}^s$ and $\mathcal{O}^u$ respectively. See \Cref{fig:transcoveredgerect}. By performing this argument on the other blue side edge of $t$, one can extract an edge rectangle $R_e$ containing $e$. (See \cite[Lemma 4.14]{Tsa22b} for a quantified version of this argument.)
\end{proof}

\begin{figure}
    \centering
    \fontsize{10pt}{10pt}\selectfont
    %% Creator: Inkscape 1.3.2 (091e20e, 2023-11-25, custom), www.inkscape.org
%% Accompanies image file 'transcoveredgerect.pdf' (pdf, eps, ps)
\begingroup%
  \makeatletter%
  \providecommand\color[2][]{%
    \errmessage{(Inkscape) Color is used for the text in Inkscape, but the package 'color.sty' is not loaded}%
    \renewcommand\color[2][]{}%
  }%
  \providecommand\transparent[1]{%
    \errmessage{(Inkscape) Transparency is used (non-zero) for the text in Inkscape, but the package 'transparent.sty' is not loaded}%
    \renewcommand\transparent[1]{}%
  }%
  \providecommand\rotatebox[2]{#2}%
  \newcommand*\fsize{\dimexpr\f@size pt\relax}%
  \newcommand*\lineheight[1]{\fontsize{\fsize}{#1\fsize}\selectfont}%
  \ifx\svgwidth\undefined%
    \setlength{\unitlength}{110.48709058bp}%
    \ifx\svgscale\undefined%
      \relax%
    \else%
      \setlength{\unitlength}{\unitlength * \real{\svgscale}}%
    \fi%
  \else%
    \setlength{\unitlength}{\svgwidth}%
  \fi%
  \global\let\svgwidth\undefined%
  \global\let\svgscale\undefined%
  \makeatother%
  \begin{picture}(1,0.87454355)%
    \lineheight{1}%
    \setlength\tabcolsep{0pt}%
    \put(0,0){\includegraphics[width=\unitlength,page=1]{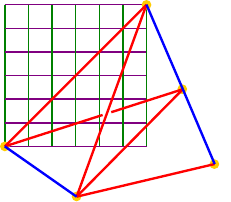}}%
    \put(0.24423111,0.56645749){\color[rgb]{1,0,0}\makebox(0,0)[lt]{\lineheight{1.25}\smash{\begin{tabular}[t]{l}$e$\end{tabular}}}}%
    \put(0.05419499,0.06691654){\color[rgb]{0,0,1}\makebox(0,0)[lt]{\lineheight{1.25}\smash{\begin{tabular}[t]{l}$e'$\end{tabular}}}}%
    \put(0.75293264,0.6835321){\color[rgb]{0,0,0}\makebox(0,0)[lt]{\lineheight{1.25}\smash{\begin{tabular}[t]{l}$t'$\end{tabular}}}}%
    \put(0.89873489,0.33437096){\color[rgb]{0,0,0}\makebox(0,0)[lt]{\lineheight{1.25}\smash{\begin{tabular}[t]{l}$t'$\end{tabular}}}}%
    \put(0,0){\includegraphics[width=\unitlength,page=2]{transcoveredgerect.pdf}}%
  \end{picture}%
\endgroup%

    \caption{Arguing that the projection of each edge $e$ spans an edge rectangle.}
    \label{fig:transcoveredgerect}
\end{figure}

Note that \Cref{lemma:transrect} in particular implies the statement in \Cref{thm:uniqueflow} about degeneracy loci and ladderpole classes.

\begin{lemma} \label{lemma:facerectmarkov}
Suppose there is an orbit segment $\widetilde{\gamma}$ in $\widetilde{\phi}^t$ going from a point on a face $f$ to a point on a face $f'$.
Then $R_f < R_{f'}$.
\end{lemma}
\begin{proof}
Since the faces of $\widetilde{\Delta}$ are positively transverse to the orbits of $\widetilde{\phi}^t$, there exists a sequence of faces $f=f_0,...,f_N=f'$ passed through by $\widetilde{\gamma}$ such that $f_{i-1}$ is a bottom face and $f_i$ is a top face of some tetrahedron $t_i$, for each $i$. As reasoned in the proof of \Cref{lemma:transrect}, the red edges of $\widetilde{\Delta}$ are positive while the blue edges of $\widetilde{\Delta}$ are negative.
From \Cref{defn:vt}, it follows that for each $i$, $R_{f_{i-1}} < R_{f_i}$.
This implies the statement of the lemma.
\end{proof}

\begin{lemma} \label{lemma:transclosedorbits}
The collection of homotopy classes of closed orbits of $\phi^t$ in $M \backslash \mathcal{C}$ is exactly the collection of homotopy classes of directed cycles in $\brloc(B)$, where $B$ is the dual veering branched surface to $\Delta$.
\end{lemma}
\begin{proof}
Given a closed orbit $\gamma$ of $\phi^t$ in $M \backslash \mathcal{C}$, there exists a cyclic sequence of faces $(f_i)$ passed through by $\gamma$ such that $f_{i-1}$ is a bottom face and $f_i$ is a top face of some tetrahedron $t_i$, for each $i$. This sequence dualizes to a cyclic sequence of edges in $\brloc(B)$ that connect up to give a directed cycle $c$ passing through the same sequence of faces, thus is homotopic to $\gamma$.

Conversely, given a directed cycle $c$ of $\brloc(B)$, we have a cyclic sequence of dual faces $(f_i)$. We lift this to a sequence of faces $(\widetilde{f}_i)$ of $\widetilde{\Delta}$. Reasoning as in \Cref{lemma:facerectmarkov}, we see that $R_{\widetilde{f}_{i-1}} < R_{\widetilde{f}_i}$ for each $i$, hence $\bigcap R_{\widetilde{f}_i}$ is a nonempty $[c]$-invariant subset of $\mathcal{O}$. This subset must consist of exactly one orbit of $\widetilde{\phi}^t$ which covers a closed orbit of $\phi^t$ homotopic to $c$, otherwise each corner of the subset determines such an orbit of $\widetilde{\phi}^t$, and there would be some leaf of the stable or unstable foliation of $\phi^t$ that contains more than one closed orbit. 
\end{proof}

We derive two more consequences on the dynamics of $\phi^t$ using similar ideas.

\begin{lemma} \label{lemma:transtransitive}
The flow $\phi^t$ is transitive.
\end{lemma}
\begin{proof}
Let $U$ be a small neighborhood in $M$. Lift $U$ to a small neighborhood $\widetilde{U}$ in $\widetilde{M}$. Let $\gamma$ be an orbit of $\widetilde{\phi}^t$ that passes through $\widetilde{U}$. Since $\Delta$ only has finitely many faces, there exists a sequence of faces $(f_i)$ passed through by $\gamma$, all covering the same face of $\Delta$, such that there is an orbit segment going from a point on $f_{i-1}$ to a point on $f_i$.

Consider the projection of $\widetilde{U}$ on $\mathcal{O}$. The intersection of the face rectangles $R_{f_i}$ is the point corresponding to $\gamma$, which lies in the projection of $\widetilde{U}$. This implies that for large $k$, $R_{f_{-k}}$ and $R_{f_k}$ intersect in a rectangle $R$ contained in the projection of $\widetilde{U}$. 

Meanwhile, one can obtain a cyclic sequence of faces of $\Delta$ by taking the images of faces passed through by the segment of $\gamma$ from $f_{-k}$ to $f_k$. This cyclic sequence lifts to a sequence of faces $(f'_i)$ of $\widetilde{\Delta}$ where $f_i=f'_i$ for $-k \leq i \leq k$. The intersection of the face rectangles $R_{f'_i}$ determines a point $\gamma'$ corresponding to an orbit of $\widetilde{\phi}^t$ that covers a closed orbit of $\phi^t$. But $\gamma'$ lies in the rectangle $R$ constructed in the previous paragraph, hence lies in the projection of $\widetilde{U}$. Hence we have produced a closed orbit of $\phi^t$ that passes through $U$.
\end{proof}

\Cref{lemma:transtransitive} in particular implies the transitivity statement in \Cref{thm:uniqueflow}.

To state the last proposition of this subsection, we need a definition.

\begin{defn} \label{defn:lozenge}
Let $\phi^t$ be a pseudo-Anosov flow on a closed 3-manifold $M$, and let $\mathcal{C}$ be a finite collection of closed orbits of $\phi^t$.

A \textbf{lozenge} is a rectangle-with-two-opposite-ideal-vertices $[0,1]^2 \backslash \{(0,0), (1,1)\}$ properly embedded in $\mathcal{O}$ such that the restrictions of $\mathcal{O}^s$ and $\mathcal{O}^u$ to the rectangle foliate it as a product, i.e. conjugate to the foliations of $[0,1]^2 \backslash \{(0,0), (1,1)\}$ by vertical and horizontal lines. See \Cref{fig:lozenge} left.

\begin{figure}
    \centering
    \fontsize{10pt}{10pt}\selectfont
    \resizebox{!}{2.5cm}{%% Creator: Inkscape 1.3 (0e150ed6c4, 2023-07-21), www.inkscape.org
%% Accompanies image file 'lozenge.pdf' (pdf, eps, ps)
\begingroup%
  \makeatletter%
  \providecommand\color[2][]{%
    \errmessage{(Inkscape) Color is used for the text in Inkscape, but the package 'color.sty' is not loaded}%
    \renewcommand\color[2][]{}%
  }%
  \providecommand\transparent[1]{%
    \errmessage{(Inkscape) Transparency is used (non-zero) for the text in Inkscape, but the package 'transparent.sty' is not loaded}%
    \renewcommand\transparent[1]{}%
  }%
  \providecommand\rotatebox[2]{#2}%
  \newcommand*\fsize{\dimexpr\f@size pt\relax}%
  \newcommand*\lineheight[1]{\fontsize{\fsize}{#1\fsize}\selectfont}%
  \ifx\svgwidth\undefined%
    \setlength{\unitlength}{248.94721084bp}%
    \ifx\svgscale\undefined%
      \relax%
    \else%
      \setlength{\unitlength}{\unitlength * \real{\svgscale}}%
    \fi%
  \else%
    \setlength{\unitlength}{\svgwidth}%
  \fi%
  \global\let\svgwidth\undefined%
  \global\let\svgscale\undefined%
  \makeatother%
  \begin{picture}(1,0.30913379)%
    \lineheight{1}%
    \setlength\tabcolsep{0pt}%
    \put(0,0){\includegraphics[width=\unitlength,page=1]{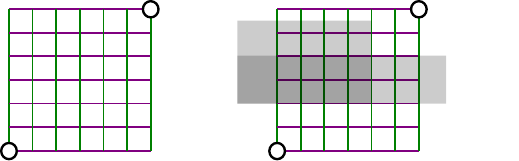}}%
    \put(0.87640147,0.14556634){\color[rgb]{0,0,0}\makebox(0,0)[lt]{\lineheight{1.25}\smash{\begin{tabular}[t]{l}$R_{f_{k-1}}$\end{tabular}}}}%
    \put(0.42157153,0.29260618){\color[rgb]{0,0,0}\makebox(0,0)[lt]{\lineheight{1.25}\smash{\begin{tabular}[t]{l}$R_{f_k}$\end{tabular}}}}%
    \put(0,0){\includegraphics[width=\unitlength,page=2]{lozenge.pdf}}%
  \end{picture}%
\endgroup%
}
    \caption{Left: A lozenge. Right: Via comparison with a sequence of face rectangles, we see that any lozenge must contain a point of $\widetilde{\mathcal{C}}$ in its interior.}
    \label{fig:lozenge}
\end{figure}

We say that $\phi^t$ \textbf{has no lozenges relative to $\mathcal{C}$} if every lozenge in $\mathcal{O}$ contains a point of $\widetilde{\mathcal{C}}$ in its interior.
\end{defn}

\begin{lemma} \label{lemma:translozenges}
The flow $\phi^t$ has no lozenges relative to $\mathcal{C}$.
\end{lemma}
\begin{proof}
Let $L$ be a lozenge.
Pick a point $\gamma \in L$. There exists a sequence of faces $(f_i)$ passed through by $\gamma$ such that $f_{i-1}$ is a bottom face and $f_i$ is a top face of some tetrahedron $t_i$, for each $i$.

We can extend the definition of taller and wider in \Cref{defn:rectangles} to include comparisons with lozenges. 
With this understood, note that $R_{f_i} < L$ for small $i$ and $L < R_{f_i}$ for large $i$. Hence there is a number $k$ such that $R_{f_{k-1}} < L$ and $R_{f_k} \not< L$. 
But recall that $R_{f_{k-1}}$ and $R_{f_k}$ are sub-rectangles of the tetrahedron rectangle $R_{t_k}$. This forces $L$ to contain a side of $R_{f_k}$ thus an element of $\widetilde{\mathcal{C}}$ in its interior. See \Cref{fig:lozenge}.
\end{proof}

\subsection{Proof of \Cref{thm:uniqueflow}} \label{subsec:uniqueflowproof}

Suppose $\phi^t_1$ and $\phi^t_2$ are pseudo-Anosov flows that satisfy the hypothesis of \Cref{thm:uniqueflow}. \Cref{lemma:transclosedorbits} implies that the collection of homotopy classes of orbits of $\phi^t_1$ in $M \backslash \mathcal{C}$ equals that of $\phi^t_2$, which implies that the collection of homotopy classes of orbits of $\phi^t_1$ in $M$ equals that of $\phi^t_2$.

If $\phi^t_1$ or $\phi^t_2$ has no tree of scalloped regions, then we are done by \cite[Theorem 1.1]{BFM23}. We note that this case suffices for our applications in this paper. In the general case, we need to work harder.
As mentioned below \Cref{lemma:transrect}, the degeneracy loci of both $\phi^t_1$ and $\phi^t_2$ agrees with the ladderpole class of $\Delta$ at each $\gamma \in \mathcal{C}$, hence they agree with each other.

We can thus perform Goodman-Fried surgery on $\phi^t_1$ and $\phi^t_2$ along $\mathcal{C}$ to obtain pseudo-Anosov flows $\widehat{\phi}^t_1$ and $\widehat{\phi}^t_2$ on a closed $3$-manifold $\widehat{M}$ that have singular orbits along $\mathcal{C}$, each of which having a different number of prongs. 
Since each $\widehat{\phi}^t_i$ restricted to $\widehat{M} \backslash \mathcal{C}$ is orbit equivalent to $\phi^t_i$ restricted to $M \backslash \mathcal{C}$, the collection of homotopy classes of orbits of $\widehat{\phi}^t_1$ in $\widehat{M}$ equals that of $\widehat{\phi}^t_2$.
This also implies that the universal cover of the punctured orbit space $\mathcal{O}_i \backslash \widetilde{\mathcal{C}}$ for $\phi^t_i$ can be identified with that of $\widehat{\mathcal{O}}_i \backslash \widetilde{\mathcal{C}}$ for $\widehat{\phi}^t_i$. Since $\phi^t_1$ and $\phi^t_2$ have no lozenges relative to $\mathcal{C}$, $\widehat{\phi}^t_1$ and $\widehat{\phi}^t_2$ have no lozenges, hence in particular have no tree of scalloped regions. See \cite[Proposition 2.7]{Tsa22b} for more details on this argument.

The case tackled above thus shows that $\widehat{\phi}^t_1$ and $\widehat{\phi}^t_2$ are orbit equivalent via a map that is isotopic to identity. 
Furthermore, this orbit equivalence must send elements of $\mathcal{C}$ to themselves because of the condition on the number of prongs. Now perform Goodman-Fried surgery to recover $\phi^t_1$ and $\phi^t_2$ respectively. The orbit equivalence between $\widehat{\phi}^t_1$ and $\widehat{\phi}^t_2$ induces an orbit equivalence between $\phi^t_1$ and $\phi^t_2$.

\begin{proof}[Proof of \Cref{cor:uniqueflow}]
For each $i=1,2$, there exists a finite collection of closed orbits $\mathcal{C}_i$ such that $\Delta$ is a triangulation on $M_i \backslash \mathcal{C}_i$. In particular $M_1 \backslash \mathcal{C}_1 \cong M_2 \backslash \mathcal{C}_2$. We let $s_2$ denote the collection of meridians determined by $M_2$ on the boundary of tubular neighborhoods of vertices of $\Delta$. 

Since $\phi^t_2$ is a pseudo-Anosov flow, $s_2$ intersects the degeneracy locus of $\phi^t_2$ in $\geq 2$ points at each element of $\mathcal{C}_2$. By the second statement in \Cref{thm:uniqueflow}, this is equivalent to the statement that $s_2$ intersects the ladderpole class of $\Delta$ in $\geq 2$ points at each vertex. Applying the second statement in \Cref{thm:uniqueflow} again, this is in turn equivalent to $s_2$ intersecting the degeneracy locus of $\phi^t_1$ in $\geq 2$ points at each element of $\mathcal{C}_1$.

Thus we can apply Goodman-Fried surgery on $\phi^t_1$ along $\mathcal{C}_1$ to get a flow $\widehat{\phi}^t_1$ on $M_2$, such that $\phi^t_1$ restricted to $M_1 \backslash \mathcal{C}_1$ is orbit equivalent to $\widehat{\phi}^t_1$ restricted to $M_2 \backslash \mathcal{C}_2$. In particular, since $\Delta$ is in transverse position with respect to $\phi^t_1$, $\Delta$ is in transverse position with respect to $\widehat{\phi}^t_1$. Note that here we used the third point in \Cref{prop:gfsurgery} to ensure that the edges of $\Delta$ can be extended into intervals. 

By the uniqueness in \Cref{thm:uniqueflow}, $\widehat{\phi}^t_1$ is orbit equivalent to $\phi^t_2$. Thus $\phi^t_1$ is almost equivalent to $\phi^t_2$.
\end{proof}

\section{Correspondence of horizontal surgeries} \label{sec:hsurcorr}

The main goal of this section is to prove the following theorem.

\begin{thm} \label{thm:hsurcorr}
Let $M$ be a closed oriented $3$-manifold and let $\mathcal{C}$ be a finite collection of curves in $M$. Let $\phi^t$ be a pseudo-Anosov flow on $M$ such that $\mathcal{C}$ is a collection of orbits of $\phi^t$. Let $\Delta$ be a veering triangulation on $M \backslash \mathcal{C}$. Suppose $\Delta$ is in steady position with respect to $\phi^t$.

Then for every positive/negative horizontal surgery curve $c$ on the veering branched surface $B$ dual to $\Delta$, there is an isotopic positive/negative horizontal surgery curve $c'$ of the flow $\phi^t$. Moreover, for every positive/negative integer $n$, the veering triangulation $\Delta_{\frac{1}{n}}(c)$ can be placed in transverse position with respect to the flow $\phi^t_{\frac{1}{n}}(c')$.
\end{thm}

Like \Cref{thm:uniqueflow}, we will only use \Cref{thm:hsurcorr} in the case when $\phi^t$ in Anosov. We state and prove \Cref{thm:hsurcorr} in the more general context of pseudo-Anosov flows in anticipation for future applications.

Since the proof of \Cref{thm:hsurcorr} involves technical details that do not play a role in the rest of the paper, the reader is advised to skim through it on the first reading.

Here is how this section is organized: In \Cref{subsec:steadyposdefn}, we define the notion of steady position. We then show that when $\Delta$ is layered, we can always arrange for steady position. This implies that \Cref{thm:layeredvthsur} follows from \Cref{thm:hsurcorr}. 

We explain the proof of \Cref{thm:hsurcorr} over \Cref{subsec:hsurcorrproofcurve} and \Cref{subsec:hsurcorrproofsur}. We show the existence of the horizontal surgery curve $c'$ in \Cref{subsec:hsurcorrproofcurve} and show that $\Delta_{\frac{1}{n}}(c)$ can be placed in transverse position with respect to $\phi^t_{\frac{1}{n}}(c')$ in \Cref{subsec:hsurcorrproofsur}.

\subsection{Steady position} \label{subsec:steadyposdefn}

\begin{defn} \label{defn:steadypos}
Let $M$ be a closed oriented $3$-manifold and let $\mathcal{C}$ be a finite collection of curves in $M$. Let $\phi^t$ be a pseudo-Anosov flow on $M$ such that $\mathcal{C}$ is a collection of orbits of $\phi^t$. Let $\Delta$ be a veering triangulation on $M \backslash \mathcal{C}$.

We say that $\Delta$ is \textbf{in steady position with respect to $\phi^t$} if:
\begin{itemize}
    \item $\Delta$ is in transverse position with respect to $\phi^t$,
    \item $\bigcup_{\text{red edges $e$}} Te$ is positive and steady, and
    \item $\bigcup_{\text{blue edges $e$}} Te$ is negative and steady.
\end{itemize}
\end{defn}

We conjecture that if $\phi^t$ is the flow corresponding to $\Delta$ on $M$, then it is always possible to isotope $\Delta$ so that it is in steady position with respect to $\phi^t$. Currently, we only know how to show this when $\Delta$ is layered. We explain the proof of this (\Cref{prop:layeredvtsteady}) in the rest of the subsection. 

We first show that the flows that correspond to layered veering triangulations are suspension flows, in the following sense.

\begin{prop} \label{prop:layeredvtfacts}
Let $M$ be a closed oriented $3$-manifold and let $\mathcal{C}$ be a finite collection of curves in $M$. Let $\phi^t$ be a pseudo-Anosov flow on $M$ such that $\mathcal{C}$ is a collection of orbits of $\phi^t$. 
Suppose $\phi^t$ restricted to $M \backslash \mathcal{C}$ is the suspension flow of some fully-punctured pseudo-Anosov map $f:S^\circ \to S^\circ$. Then $\phi^t$ corresponds to a unique layered veering triangulation of $M \backslash \mathcal{C}$, up to isotopy.

Conversely, if $\phi^t$ corresponds to a layered veering triangulation of $M \backslash \mathcal{C}$, then $\phi^t$ restricted to $M \backslash \mathcal{C}$ is the suspension flow of some fully-punctured pseudo-Anosov map. 
\end{prop}
\begin{proof}
For the first statement, the existence of such a layered veering triangulation is the main result of \cite{Ago11}.
For uniqueness, if a layered veering triangulation $\Delta$ can be placed in transverse position with respect to $\phi^t$, then the dual veering branched surface $B$ to $\Delta$ can be obtained by suspending a $f$-periodic splitting sequence of train tracks. As explained in the proof of \cite[Proposition 4.2]{Ago11}, one can use the results of Hamenst\"adt in \cite{Ham09} to show that such a branched surface $B$ is unique up to isotopy. This implies that $\Delta$ is unique up to isotopy.

For the second statement, the $2$-skeleton of a layered veering triangulation $\Delta$ of $M \backslash \mathcal{C}$ carries the fibers of a fibration over $S^1$. If $\phi^t$ corresponds to $\Delta$, then the orbits of $\phi^t$ are positively transverse to each fiber surface, hence is the suspension flow of some fully-punctured pseudo-Anosov map.
\end{proof}

\begin{prop} \label{prop:layeredvtsteady}
Let $M$ be a closed oriented $3$-manifold and let $\mathcal{C}$ be a finite collection of curves in $M$. Let $\Delta$ be a veering triangulation on $M \backslash \mathcal{C}$. Suppose $\Delta$ is layered and suppose $\phi^t$ is the pseudo-Anosov flow corresponding to $\Delta$, then $\Delta$ can be placed in steady position with respect to $\phi^t$.
\end{prop}
\begin{proof}
Consider the punctured orbit space $\mathcal{O} \backslash \widetilde{\mathcal{C}}$. The second statement of \Cref{prop:layeredvtfacts} states that the restriction of $\phi^t$ to $M \backslash \mathcal{C}$ is the suspension flow of some fully-punctured pseudo-Anosov map $f:S^\circ \to S^\circ$. This implies that the punctured orbit space can be identified with a cover of $S^\circ$. In particular the transverse measures of the stable and unstable foliations of $f$ can be lifted to this cover then transferred to $\mathcal{O} \backslash \widetilde{\mathcal{C}}$, giving the latter a Euclidean structure. 

In \cite{Gue16}, Gu\'eritaud constructs a layered veering triangulation $\Delta'$ on $M \backslash \mathcal{C}$ such that
\begin{itemize}
    \item $\Delta'$ is in transverse position with respect to $\phi^t$, and
    \item for every edge $e$ of $\widetilde{\Delta'}$, the projection of $e$ to $\mathcal{O} \backslash \widetilde{\mathcal{C}}$ is a straight line with respect to the Euclidean structure. 
\end{itemize}
From the first statement of \Cref{prop:layeredvtfacts}, the first property implies that $\Delta=\Delta'$ up to isotopy. It remains to show that $\Delta'$ is in steady position with respect to $\phi^t$.

We have already seen from the proof of \Cref{lemma:transrect} that $\Delta$ being in transverse position with respect to $\phi^t$ implies that every red edge is positive while every blue edge is negative.

Now suppose there is a crossing $(x,y,t)$ of $e_2$ over $e_1$, where $e_1$ and $e_2$ are red edges. We lift the orbit segment from $x$ to $y$ to $\widetilde{M}$, so that it begins at a point $\widetilde{x}$ on a red edge $\widetilde{e_1}$ and ends at a point $\widetilde{y}$ on a red edge $\widetilde{e_2}$. There exists a sequence of faces $f_0,...,f_N$ passed through by this lifted orbit segment, such that 
\begin{itemize}
    \item $f_0$ is a bottom face of a tetrahedron that has $\widetilde{e_1}$ as its bottom edge,
    \item $f_{i-1}$ is a bottom face and $f_i$ is a top face of some tetrahedron $t_i$, for each $i$, and
    \item $f_N$ is a top face of a tetrahedron that has $\widetilde{e_2}$ as its top edge.
\end{itemize}
Then we have the chain of inequalities $R_{\widetilde{e_1}} < R_{f_0} < \dots < R_{f_N} < R_{\widetilde{e_2}}$.
Together with the fact that the projection of $\widetilde{e_i}$ is a straight line spanning $R_{\widetilde{e_i}}$ for each $i$, we see that the projection of $\widetilde{e_2}$ must intersect that of $\widetilde{e_1}$ at exactly one point, at which the tangent line of $\widetilde{e_2}$ has greater slope than that of $\widetilde{e_1}$. This implies the steady condition for the crossing $(x,y,t)$ and shows that $\bigcup_{\text{red edges $e$}} Te$ is steady.

The symmetric argument shows that $\bigcup_{\text{blue edges $e$}} Te$ is steady.
\end{proof}

Once we prove \Cref{thm:hsurcorr}, it combines with \Cref{prop:layeredvtsteady} to give \Cref{thm:layeredvthsur}.

\subsection{Constructing the horizontal surgery curve} \label{subsec:hsurcorrproofcurve}

We initiate the proof of \Cref{thm:hsurcorr}. We will prove the theorem in the case when $c$ is positive. The case when $c$ is negative is symmetric. 

Our task in this subsection is to construct the horizontal surgery curve $c'$. 

Let $A$ be the dual annulus to $c$. By \Cref{prop:vthsur}, we can assume that the edges in the boundary of $A$ are all red while the edges in the interior of $A$ are all blue. 
Notice that the same face of $\Delta$ might appear multiple times in $A$. In particular, $A$ is in general not embedded at this point.

We can make $A$ embedded by isotoping its faces slightly along orbits of $\phi^t$. See \Cref{fig:hsurcorrannemb} middle. 
After this modification, the boundary of $A$ consists of translates of red edges, i.e. images of red edges isotoped slightly along orbits of $\phi^t$. We choose one boundary component of $A$ and denote the cyclic sequence of (ideal) vertices on that component by $(v_i)$. We let $e_i$ be the edge that goes from $v_i$ to $v_{i+1}$. Up to relabelling $v_i$ we can assume that this orientation on $e_i$ is the one induced from the orientation on $A$ (which is in turn induced from the flow direction and the orientation on $M$).

\begin{figure}
    \centering
    \fontsize{8pt}{8pt}\selectfont
    %% Creator: Inkscape 1.3.2 (091e20e, 2023-11-25, custom), www.inkscape.org
%% Accompanies image file 'hsurcorrannemb.pdf' (pdf, eps, ps)
\begingroup%
  \makeatletter%
  \providecommand\color[2][]{%
    \errmessage{(Inkscape) Color is used for the text in Inkscape, but the package 'color.sty' is not loaded}%
    \renewcommand\color[2][]{}%
  }%
  \providecommand\transparent[1]{%
    \errmessage{(Inkscape) Transparency is used (non-zero) for the text in Inkscape, but the package 'transparent.sty' is not loaded}%
    \renewcommand\transparent[1]{}%
  }%
  \providecommand\rotatebox[2]{#2}%
  \newcommand*\fsize{\dimexpr\f@size pt\relax}%
  \newcommand*\lineheight[1]{\fontsize{\fsize}{#1\fsize}\selectfont}%
  \ifx\svgwidth\undefined%
    \setlength{\unitlength}{422.11284968bp}%
    \ifx\svgscale\undefined%
      \relax%
    \else%
      \setlength{\unitlength}{\unitlength * \real{\svgscale}}%
    \fi%
  \else%
    \setlength{\unitlength}{\svgwidth}%
  \fi%
  \global\let\svgwidth\undefined%
  \global\let\svgscale\undefined%
  \makeatother%
  \begin{picture}(1,0.11423831)%
    \lineheight{1}%
    \setlength\tabcolsep{0pt}%
    \put(0,0){\includegraphics[width=\unitlength,page=1]{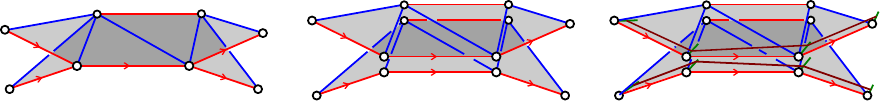}}%
    \put(0.89259287,0.06889377){\color[rgb]{0.50196078,0,0}\makebox(0,0)[lt]{\lineheight{1.25}\smash{\begin{tabular}[t]{l}$c''$\end{tabular}}}}%
  \end{picture}%
\endgroup%

    \caption{An outline for constructing the horizontal surgery curve $c'$: First we make the dual annulus $A$ embedded, then we push the edges in one boundary component of $A$ inwards.}
    \label{fig:hsurcorrannemb}
\end{figure}

The idea is to push each $e_i$ into $A$ slightly and take its closure to get an interval $\overline{e}_i$, then connect these up into a piecewise smooth curve $c''$. See \Cref{fig:hsurcorrannemb} right. We use the steadiness of $\bigcup_{\text{red edges $e$}} Te$ to arrange for $c''$ to be a piecewise smooth horizontal surgery curve. Here, one difficulty is that $\bigcup_i Te_i$ is not steady, since there might be elements that are translates of the same edge of $\Delta$. To overcome this, we have to choose $\overline{e}_i$ carefully. Once we have $c''$, we can approximate it using a smooth horizontal surgery curve $c'$.

Here we have to recall the notion of piecewise smooth horizontal surgery curves from \cite{Tsa24}.

\begin{defn} \label{defn:piecewisehorsurcurve}
A \textbf{piecewise smooth curve} is an embedded curve $c$ that is a concatenation of finitely many smooth embedded paths. The points at which the paths are joint together are called the \textbf{turns}. At each turn, the two tangent lines determined by the two paths may or may not agree.

A piecewise smooth curve is \textbf{positive/negative} if it is a concatenation of positive/negative paths respectively.

Fix a Riemannian metric. A positive/negative piecewise smooth curve $c$ is a \textbf{piecewise smooth horizontal surgery curve} if at every crossing $(x,y,t)$ of $c$, the slope of $Tc|_y$ is greater than that of $d\phi^t(Tc|_x)$ on either side. More precisely, fix an orientation on $E^s|_x$ and transfer it to an orientation on $E^s|_y$ using $d\phi^t$. If $p^L_x$ and $p^R_x$, and $p^L_y$ and $p^R_y$ are the paths that constitute $c$ on the left and right of $x$, and $y$ respectively, then $\slope(T{p^L_y}|_y) > \slope(d\phi^t(Tp^L_x|_x))$ and $\slope(T{p^R_y}|_y) > \slope(d\phi^t(Tp^R_x|_x))$.
\end{defn}

Here are the details for constructing $c''$. For each $i$, let $s_i$ be a short segment on $A$ emerging from $v_i$ and lying along a leaf of the stable foliation. For each $i$, we wish to construct an interval $\overline{e}_i$ connecting a point $z_i$ on $s_i$ to a point $z_{i+1}$ on $s_{i+1}$. If each $z_i$ is close enough to $v_i$, we can take each $\overline{e}_i$ to be $C^1$-close to $e_i$. See \Cref{fig:hsurcorrtranslatebraid}.

There is one property that we wish for the intervals $\overline{e}_i$ to satisfy: 
Let $e$ be a red edge of $\Delta$. Let $N \cong [0,1] \times [0,1] \times [0,1]$ be a neighborhood where $e \cong (0,1) \times \{0\} \times \{\frac{1}{2}\}$ and the intervals $\{x\} \times \{y\} \times [0,1]$ are orbit segments. Then the translates of $e$ among $(e_i)$ lie along the plane $[0,1] \times \{0\} \times [0,1]$. We call such an $N$ a \textbf{translate neighborhood}. Let $\epsilon_N$ be the collection $\{e_i \mid \text{$A$ meets the interior of $N$ near $e_i$}\}$. That is, $\epsilon_N$ is the collection of elements among $(e_i)$ for which we wish to push into $N$.
 
The condition we wish to impose is that upon projecting the intervals $\overline{e}_i$ that lie within each such $N$ to $[0,1] \times [0,1]$ in the first two coordinates, we only have negative crossings.
See \Cref{fig:hsurcorrtranslatebraid}.

\begin{figure}
    \centering
    \fontsize{10pt}{10pt}\selectfont
    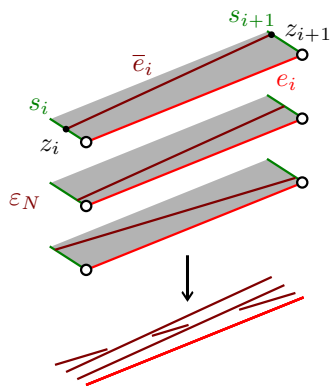
    \caption{We choose $\overline{e}_i$ so that for each $N$, the projection of the $\overline{e}_i$ that lie within $N$ only contains negative crossings.}
    \label{fig:hsurcorrtranslatebraid}
\end{figure}

One way to arrange for this is to fix some $e_i$ that is the bottom element of some $\epsilon_N$, say this is $e_1$. We choose $\overline{e}_i$ for $i=2,3,...$ in order by arranging for $z_{i+1}$ to be furthest away from $v_{i+1}$ compared to the $z_2,...,z_i$ that lie on translates of $s_{i+1}$. This ensures that no crossings are produced before we choose $\overline{e}_1$, and by that point we can choose $\overline{e}_1$ to only produce negative crossings with $\overline{e}_2,\overline{e}_3,...$.

Note that it is still true that by choosing each $z_i$ close enough to $v_i$, we can take each $\overline{e}_i$ to be $C^1$-close to $e_i$. 
Furthermore, up to a small perturbation of the points $z_i$ along $s_i$, we can assume that the orbits of $\phi^t$ passing through $z_i$ do not coincide. 

Let $c''$ be the union of the $\overline{e}_i$.

\begin{lemma} \label{lemma:pushpiecewisehsurcurve}
By choosing each $z_i$ close enough to $v_i$, $\bigcup_i T\overline{e}_i \cup \bigcup_{\text{red edges $e$}} Te$ is steady.
In particular $c''$ is a piecewise smooth horizontal surgery curve. 
\end{lemma}

We recall the notion of an instantaneous metric from \cite{Tsa24} for the proof of \Cref{lemma:pushpiecewisehsurcurve}. This is a useful tool for working with the steadiness condition.

\begin{defn} \label{defn:instantmetric}
Let $\phi^t$ be an Anosov flow. An instantaneous metric is a Riemannian metric $g$ such that
$$||d\phi^t(v)||_g < \lambda^{-t} ||v||_g$$
for every $v \in E^s, t>0$, and 
$$||d\phi^t(v)||_g < \lambda^t ||v||_g$$
for every $v \in E^u, t<0$, for some $\lambda>1$.

When $\phi^t$ is a pseudo-Anosov flow, we only require $g$ to be defined away from $\sing(\phi^t)$. We refer to \cite[Lemma 2.4]{Tsa24} for details.
\end{defn}

\begin{proof}[Proof of \Cref{lemma:pushpiecewisehsurcurve}]
Fix an instantaneous metric. Pick a number $T_1$ so that $\lambda^{-2T_1}< \frac{\min \slope \bigcup_{\text{red edges $e$}} Te}{\max \slope \bigcup_{\text{red edges $e$}} Te}$. We also pick a number $T_0>0$ smaller than the first return time of $\bigcup_{\text{red edges $e$}} e$. By making the isotopy we perform to make $A$ embedded small enough, and by making each $z_i$ close enough to $v_i$, we can ensure that each $\overline{e}_i$ is close to a red edge of $\Delta$ in the $C^1$-topology. In particular, we have $\lambda^{-2T_1}< \frac{\min \slope \bigcup_i T\overline{e}_i}{\max \slope \bigcup_i T\overline{e}_i}$ and that every crossing of $\bigcup_i \overline{e}_i$ not lying within some translate neighborhood has time $>T_0$. 

The inequalities imply that it suffices to check the steadiness condition for time $<T_1$ crossings.
Meanwhile, by \cite[Claim 3.5]{Tsa24}, away from fixed neighborhoods of $\mathcal{C}$, every time $[T_0,T_1]$ crossing of $\bigcup_i \overline{e}_i \cup \bigcup_{\text{red edges $e$}} e$ is close to a crossing of $\bigcup_{\text{red edges $e$}} e$, so the steadiness condition is satisfied for the time $[T_0,T_1]$ crossings lying away from the fixed neighborhoods of $\mathcal{C}$.

The crossings that lie within some translate neighborhood satisfy the steadiness condition by construction. These include all the time $<T_0$ crossings that lie away from the fixed neighborhoods of $\mathcal{C}$.

Finally, we consider crossings within a fixed neighborhood $\eta$ of $\gamma \in \mathcal{C}$. For simplicity let us suppose that $\gamma$ is not a singular orbit. Suppose we have vertices $v_i$ and $v_j$ lying on $\gamma$. We let $A_i$ and $A_j$ be the components of $A \cap \eta$ containing $v_i$ and $v_j$ respectively. If there is a crossing of $\overline{e}_{i-1} \cup \overline{e}_i$ over $\overline{e}_{j-1} \cup \overline{e}_j$ contained within $\eta$, we lift the crossing to $\widetilde{M}$. Then the lift of $A_i$ must lie above the lift of $A_j$. Suppose that $A_i$ and $A_j$ lie on the same side of $\gamma$. Then as we push $e_{i-1} \cup e_i$ and $e_{j-1} \cup e_j$ into $A$, the crossing must be of $\overline{e}_{i-1}$ over $\overline{e}_{j-1}$ or of $\overline{e}_i$ over $\overline{e}_j$. See \Cref{fig:hsurcorrpushsteady} top row.

Suppose the crossing is of $\overline{e}_{i-1}$ over $\overline{e}_{j-1}$. If $\overline{e}_{i-1}$ and $\overline{e}_{j-1}$ are translates of the same edge $e$ of $\widetilde{\Delta}$, then the crossing is within translate neighborhood and satisfies the steadiness condition as addressed above. Otherwise $\overline{e}_{i-1}$ and $\overline{e}_{j-1}$ are translates of different edges of $\widetilde{\Delta}$. In this case the slope of $e_{i-1}$ at $v_i$ is larger than the slope of $e_{j-1}$ at $v_j$, since the former lies above the latter in the same ladderpole curve (recall \Cref{prop:vtboundarytriang}) and $\Delta$ is in transverse position with respect to $\phi^t$. As we push $e_{i-1}$ and $e_{j-1}$ into $A$, the crossing will satisfy the steadiness condition. See \Cref{fig:hsurcorrpushsteady} first row, first case. 

\begin{figure}
    \centering
    \fontsize{8pt}{8pt}\selectfont
    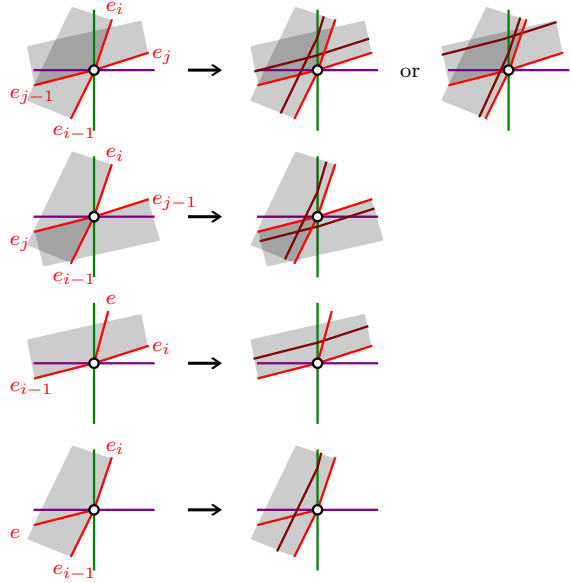
    \caption{Checking for the steadiness condition for crossings within fixed neighborhoods of $\mathcal{C}$.}
    \label{fig:hsurcorrpushsteady}
\end{figure}

The cases when the crossing is of $\overline{e}_i$ over $\overline{e}_j$, and when $A_i$ and $A_j$ lie on different sides of $\gamma$, can be checked similarly and are illustrated in \Cref{fig:hsurcorrpushsteady} first row, second case and \Cref{fig:hsurcorrpushsteady} second row, respectively. 

Now suppose there is a crossing of a red edge $e$ of $\Delta$ over $\overline{e}_{i-1} \cup \overline{e}_i$. We lift the crossing to $\widetilde{M}$. Then the lift of $e$ must lie above the lift of $A_i$. As we push $e_{i-1} \cup e_i$ into $A$, the crossing must be of $e$ over $\overline{e}_i$. Suppose the latter is the case. If $e_i$ is a translate of the lift of $e$, then the crossing arises within a translate neighborhood and satisfies the steadiness condition as addressed above. Otherwise the slope of $e$ is larger than the slope of $e_i$, since $\Delta$ is in transverse position. As we push $e_i$ into $A$, the crossing will satisfy the steadiness condition. See \Cref{fig:hsurcorrpushsteady} third row.

Finally, the case when there is a crossing of $\overline{e}_{i-1} \cup \overline{e}_i$ over a red edge $e$ of $\Delta$ can be checked similarly. See \Cref{fig:hsurcorrpushsteady} fourth row.

When $\gamma$ is singular, similar arguments hold with slightly more bookkeeping.
\end{proof}

\begin{defn}
We say that a sequence of smooth curves $(c_n)$ \textbf{converges towards a piecewise smooth curve $c$} if for every $\epsilon>0$, there exists $N$ so that for every $n \geq N$,
\begin{itemize}
    \item $c_n$ is $\epsilon$-close to $c$ outside of the $\epsilon$-neighborhoods of the turns of $c$ in the $C^1$-topology, and
    \item $\slope(Tc_n|_x)$ stays bounded within the $\epsilon$-neighborhood of $\slope(Tc|_v)$ for $x$ in the $\epsilon$-neighborhood of each turn $v$ of $c$.
\end{itemize}
\end{defn}

Observe that the turns of $c''$ lie along the points $z_i$, which lie on distinct orbits, hence $c''$ is generic in the sense of \cite[Section 5.2]{Tsa24}. Hence \cite[Proposition 5.3]{Tsa24} allows us to approximate $c''$ by a smooth horizontal surgery curve $c'$ lying on $A$. For $c'$ close enough to $c''$, \Cref{lemma:pushpiecewisehsurcurve} holds with $c''$ replaced by $c'$, that is:

\begin{lemma} \label{lemma:hsurcorrcurve}
$Tc' \cup \bigcup_{\text{red edges $e$}} Te$ is steady.
\end{lemma}

\subsection{Carrying out the surgery} \label{subsec:hsurcorrproofsur}

The second part of the proof of \Cref{thm:hsurcorr} involves showing how to perform horizontal Goodman surgery on $c'$ so that $\Delta_{\frac{1}{n}}(c)$ is in transverse position with respect to $\phi^t_{\frac{1}{n}}(c')$.

We first have to construct a surgery annulus containing $c'$. Recall that such a surgery annulus has to contain a foliation $\mathcal{H}$ by closed curves and a foliation $\mathcal{K}$ by nonseparating arcs. The foliation $\mathcal{H}$ will be chosen by just taking curves lying close to $c'$. To construct the foliation $\mathcal{K}$, we will require a slight perturbation on $c'$.

Fix an instantaneous metric as in the previous subsection.
Pick a number $T_2$ so that $\lambda^{-2T_2}< \frac{\min \slope \bigcup_{\text{blue edges $e$}} Te}{\max \slope \bigcup_{\text{blue edges $e$}} Te}$.

Recall that we made $A$ embedded by isotoping along orbits. We can bring the blue interior edges along in the isotopy, so that $A$ admits a collection of translates of blue edges in its interior.
We let $E_0$ be the collection of these translated blue edges, and let $E_1$ be the collection of blue edges of $\Delta$ that do not lie in the interior of the initial dual annulus, i.e. those blue edges that are not translated.

Since $c'$ is positive while $E_0$ is negative, $c'$ intersects each edge in $E_0$ exactly one point.
Up to an arbitrarily small perturbation of $c'$, we can assume that the intersection points of $c'$ with the elements of $E_0$ lie along distinct orbits.

With this modification, we pick an sub-annulus $A'$ of $A$ that is transverse to the flow and admits a foliation $\mathcal{H}$ by closed curves including $c'$ as a leaf. Up to shrinking $A'$, $T\mathcal{H}$ can be made arbitrarily close to $Tc'$, hence we can assume that $T\mathcal{H} \cup \bigcup_{\text{red edges $e$}} Te$ is steady. 

Consider the set of points $\mathfrak{c}$ on $c'$ that lies on its intersection with $E_0$ or are involved in a time $< T_2$ crossing with $c'$ or $E_0 \cup E_1$. Since $c'$ is positive and $\bigcup_{\text{blue edge $e$}} Te$ is steady, this is a finite set. 

For each point $z \in \mathfrak{c}$, we pick a negative value $m_z$ between $\min \slope \bigcup_{\text{blue edges $e$}} Te$ and $\max \slope \bigcup_{\text{blue edges $e$}} Te$, such that:
\begin{itemize}
    \item If $z$ is the intersection of $c'$ with $e \in E_0$, then $m_z = \slope(Te|_z)$.
    \item For every time $< T_2$ crossing $(x,y,t)$ of $c'$, $m_x > m_y$.
    \item For every time $< T_2$ crossing $(x,y,t)$ of $c'$ over $e \in E_0 \cup E_1$, $\slope(d\phi^t(Te|_x)) > m_y$.
    \item For every time $< T_2$ crossing $(x,y,t)$ of $e \in E_0 \cup E_1$ over $c'$, $m_x > \slope(d\phi^{-t}(Te|_y))$.
\end{itemize}
That this is possible uses the fact that $\bigcup_{\text{blue edge $e$}} Te$ is steady.

We then define the foliation $\mathcal{K}$ by requiring that it contains the intersection of $A'$ with $E_0$ as leaves, and that the leaf passing through $z \in \mathfrak{c}$ has slope $m_z$ at $z$. Then up to shrinking $A'$, the time $< T_2$ crossings of $\mathcal{K}$ lie close to the crossings of $c'$, hence by construction $T\mathcal{K}$ is steady.

Every crossing $(x,y,t)$ of $\mathcal{K}$ over $E_0 \cup E_1$ is either of time $\geq T_2$, or lies close to a crossing of $c'$ over $E_0 \cup E_1$, hence $\slope(d\phi^t(\bigcup_{e \in E_0 \cup E_1} Te|_x)) < \slope(T\mathcal{K}|_y)$. Similarly, for every crossing $(x,y,t)$ of $\bigcup_{e \in E_0 \cup E_1} e$ over $\mathcal{K}$, we have $\slope(d\phi^t(T\mathcal{K}|_x)) < \slope(\bigcup_{e \in E_0 \cup E_1} Te|_y)$.

We record the crucial points of our progress so far as a lemma.

\begin{lemma} \label{lemma:hsurcorrsurann}
There exists a surgery annulus $(A',\mathcal{H},\mathcal{K})$ containing $c'$ such that
\begin{itemize}
    \item \begin{itemize}
        \item $T\mathcal{H}$ is steady.
        \item For every crossing $(x,y,t)$ of $\mathcal{H}$ over $\bigcup_{\text{red edges $e$}} e$, 
        $$\slope(d\phi^t(\bigcup_{\text{red edges $e$}} Te|_x)) < \slope(T\mathcal{H}|_y).$$
        \item For every crossing $(x,y,t)$ of $\bigcup_{\text{red edges $e$}} e$ over $\mathcal{H}$, 
        $$\slope(d\phi^t(T\mathcal{H}|_x)) < \slope(\bigcup_{\text{red edges $e$}} Te|_y).$$
    \end{itemize}
    \item \begin{itemize}
        \item $T\mathcal{K}$ is steady.
        \item For every crossing $(x,y,t)$ of $\mathcal{K}$ over $E_0 \cup E_1$, 
        $$\slope(d\phi^t(\bigcup_{e \in E_0 \cup E_1} Te|_x)) < \slope(T\mathcal{K}|_y).$$
        \item For every crossing $(x,y,t)$ of $\bigcup_{e \in E_0 \cup E_1} e$ over $\mathcal{K}$, 
        $$\slope(d\phi^t(T\mathcal{K}|_x)) < \slope(\bigcup_{e \in E_0 \cup E_1} Te|_y).$$
    \end{itemize}
\end{itemize}
\end{lemma}

We carry out the surgery operation using this surgery annulus: We cut $M$ along $A'$ to get $M \cut A'$, which has two copies of $A'$ on its boundary --- $A'_+$ on the positive side and $A'_-$ on the negative side. Notice that this also separates each edge $e$ in $E_0$ into two edges, $e_-$ and $e_+$, that split apart on $A'$. See \Cref{fig:hsurcorrinsertsolidtorus} second row. We denote the set of $e_\pm$ by $E^\pm_0$ respectively.

\begin{figure}
    \centering
    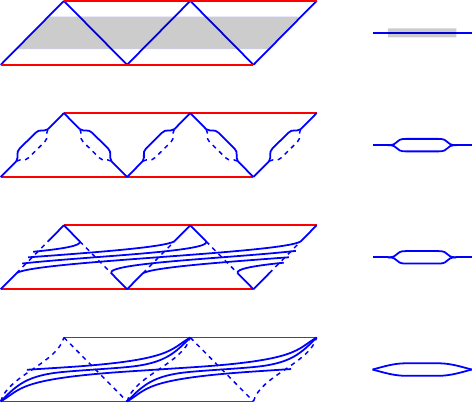
    \caption{Steps of the surgery operation. Left: A top view. Right: A cross-sectional view. We first cut along a subannulus $A' \subset A$. This separates $E_0$ into $E^-_0$ and $E^+_0$, which do not match up under the gluing map. We resolve bigons formed by edges in $E^-_0$ and $E^+_0$ and further split along the entire $A$. At this point we can insert the triangulated solid torus.}
    \label{fig:hsurcorrinsertsolidtorus}
\end{figure}

To complete the surgery, we glue $A'_-$ to $A'_+$ using a map $\sigma:A'_- \to A'_+$.
In \cite{Tsa24}, it is shown that for appropriate $\sigma$, the reglued flow $\overline{\phi}^t$ defined on the surgered flow $\overline{M}$ is pseudo-Anosov. The only part of this proof that matters to us is the way $T\overline{\phi} \oplus \overline{E}^s$ and $T\overline{\phi} \oplus \overline{E}^u$ are located.

For $T\overline{\phi} \oplus \overline{E}^u$, we defined a cone field $C^{cu}$ on $M \cut A'$ as follows:
\begin{itemize}
    \item If $y = x_+ \in A'_+$, then $C^{cu}|_y$ is the union of the two opposite quadrants bounded by $T\mathcal{H}|_x$ and $T\mathcal{K}|_x$ that meet $E^u$, direct sum with $T\phi|_x$.
    \item If $y = \phi^t(x_+)$ for $x_+ \in A'_+$ and $t>0$, then $C^{cu}|_y = d\phi^t(C^{cu}_{x_+})$.
    \item If $y \notin \phi^{[0,\infty)}(A'_+)$, then $C^{cu}|_y = E^u|_y \oplus T\phi|_y$.
\end{itemize}

By identifying $\overline{M}$ with the complement of the interior of $A'_-$ in $M \cut A'$, $C^{cu}$ descends to a discontinuous cone field $\overline{C}^{cu}$ on $\overline{M}$. We then showed that $d\overline{\phi}^t(\overline{C}^{cu}|_{\overline{\phi}^{-t}(x)}) \subset \overline{C}^{cu}|_x$ for every $t>0$, and $T\overline{\phi} \oplus \overline{E}^u = \bigcap_{t \in [0,\infty)} d\overline{\phi}^t(\overline{C}^{cu}|_{\overline{\phi}^{-t}(x)})$.

Symmetrically, for $T\overline{\phi} \oplus \overline{E}^s$, we defined a cone field $C^{cs}$ on $M \cut A'$ as follows:
\begin{itemize}
    \item If $y = x_- \in A'_-$, then $C^{cs}|_y$ is the union of the two opposite quadrants bounded by $T\mathcal{H}|_x$ and $T\mathcal{K}|_x$ that meet $E^s$, direct sum with $T\phi|_x$.
    \item If $y = \phi^t(x_-)$ for $x_- \in A'_-$ and $t<0$, then $C^{cs}|_y = d\phi^t(C^{cs}_{x_-})$.
    \item If $y \notin \phi^{(-\infty,0]}(A'_-)$, then $C^{cs}|_y = E^s|_y \oplus T\phi|_y$.
\end{itemize}

By identifying $\overline{M}$ with the complement of the interior of $A'_+$ in $M \cut A'$, $C^{cs}$ descends to a discontinuous cone field $\overline{C}^{cs}$ on $\overline{M}$. We have $d\overline{\phi}^t(\overline{C}^{cs}|_{\overline{\phi}^{-t}(x)}) \subset \overline{C}^{cs}|_x$ for every $t<0$, and $T\overline{\phi} \oplus \overline{E}^s = \bigcap_{t \in (-\infty,0]} d\overline{\phi}^t(\overline{C}^{cs}|_{\overline{\phi}^{-t}(x)})$.

In particular, note that $C^{cu}$ and $C^{cs}$ intersect in $T\phi$ at every point.

\begin{figure}
    \centering
    \fontsize{10pt}{10pt}\selectfont
    %% Creator: Inkscape 1.3.2 (091e20e, 2023-11-25, custom), www.inkscape.org
%% Accompanies image file 'hsurcorrcones.pdf' (pdf, eps, ps)
\begingroup%
  \makeatletter%
  \providecommand\color[2][]{%
    \errmessage{(Inkscape) Color is used for the text in Inkscape, but the package 'color.sty' is not loaded}%
    \renewcommand\color[2][]{}%
  }%
  \providecommand\transparent[1]{%
    \errmessage{(Inkscape) Transparency is used (non-zero) for the text in Inkscape, but the package 'transparent.sty' is not loaded}%
    \renewcommand\transparent[1]{}%
  }%
  \providecommand\rotatebox[2]{#2}%
  \newcommand*\fsize{\dimexpr\f@size pt\relax}%
  \newcommand*\lineheight[1]{\fontsize{\fsize}{#1\fsize}\selectfont}%
  \ifx\svgwidth\undefined%
    \setlength{\unitlength}{130.56902788bp}%
    \ifx\svgscale\undefined%
      \relax%
    \else%
      \setlength{\unitlength}{\unitlength * \real{\svgscale}}%
    \fi%
  \else%
    \setlength{\unitlength}{\svgwidth}%
  \fi%
  \global\let\svgwidth\undefined%
  \global\let\svgscale\undefined%
  \makeatother%
  \begin{picture}(1,0.72277392)%
    \lineheight{1}%
    \setlength\tabcolsep{0pt}%
    \put(0,0){\includegraphics[width=\unitlength,page=1]{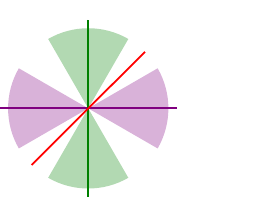}}%
    \put(0.30332063,0.67562458){\color[rgb]{0,0.50196078,0}\makebox(0,0)[lt]{\lineheight{1.25}\smash{\begin{tabular}[t]{l}$C^{cs}$\end{tabular}}}}%
    \put(0.68494297,0.31003881){\color[rgb]{0.50196078,0,0.50196078}\makebox(0,0)[lt]{\lineheight{1.25}\smash{\begin{tabular}[t]{l}$C^{cu}$\end{tabular}}}}%
    \put(0.55612783,0.53318963){\color[rgb]{1,0,0}\makebox(0,0)[lt]{\lineheight{1.25}\smash{\begin{tabular}[t]{l}$Te$\end{tabular}}}}%
  \end{picture}%
\endgroup%

    \caption{The form of $C^{cu}$ and $C^{cs}$ at a point on a red edge $e$.}
    \label{fig:hsurcorrcones}
\end{figure}
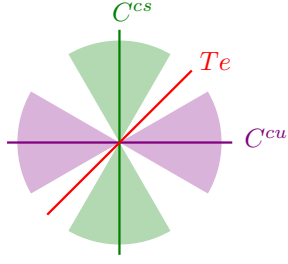

Now let $e$ be a red edge of $\Delta$ and let $z$ be a point on $e$.
The last two items in the first point of \Cref{lemma:hsurcorrsurann} imply that $Te|_z$ sits inbetween $C^{cu}$ and $C^{cs}$ as in \Cref{fig:hsurcorrcones}.
This implies that $Te|_z$ is positive with respect to the reglued flow.
Similarly, each edge in $E^-_0 \cup E^+_0 \cup E_1$ is negative with respect to the reglued flow.

Notice that the edges in $E^-_0$ and $E^+_0$ do \emph{not} match up under the gluing map $\sigma$. This corresponds to the fact that we have to insert a triangulated solid torus $T$ to obtain the triangulation $\Delta_{\frac{1}{n}}(c)$. The remaining task in the proof is to show that we can peel $E^-_0$ and $E^+_0$ apart in $\overline{M}$ then insert $T$ in a way so that its faces are positively transverse to the flow and its interior edges are negative.
See \Cref{fig:hsurcorrinsertsolidtorus} last two rows.

For each blue edge $e$ in $T$, we claim that there is a negative path on $A$ isotopic to $e$ rel endpoints.
For $e$ on $\partial_\pm T$, we can take the corresponding edge in $E^\pm_0$. Otherwise, suppose the endpoints of $e$ of $v_1$ and $v_2$. Note that there is an immersed triangular surface $T$ in $A$ that is bounded by an edge $e_-$ of $E^-_0$ with an endpoint on $v_1$, an edge $e_+$ of $E^+_0$ with an endpoint on $v_1$, and a segment of $c_2$ containing $v_2$. See \Cref{fig:hsurcorrsolidtorusedges}.
Since $e_\pm$ are negative while $c_2$ is positive (since it consists of red edges), there is a negative path from $v_1$ to $v_2$.

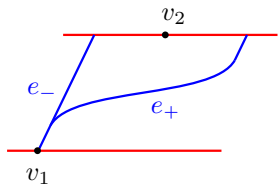
\begin{figure}
    \centering
    \fontsize{10pt}{10pt}\selectfont
    %% Creator: Inkscape 1.3.2 (091e20e, 2023-11-25, custom), www.inkscape.org
%% Accompanies image file 'hsurcorrsolidtorusedges.pdf' (pdf, eps, ps)
\begingroup%
  \makeatletter%
  \providecommand\color[2][]{%
    \errmessage{(Inkscape) Color is used for the text in Inkscape, but the package 'color.sty' is not loaded}%
    \renewcommand\color[2][]{}%
  }%
  \providecommand\transparent[1]{%
    \errmessage{(Inkscape) Transparency is used (non-zero) for the text in Inkscape, but the package 'transparent.sty' is not loaded}%
    \renewcommand\transparent[1]{}%
  }%
  \providecommand\rotatebox[2]{#2}%
  \newcommand*\fsize{\dimexpr\f@size pt\relax}%
  \newcommand*\lineheight[1]{\fontsize{\fsize}{#1\fsize}\selectfont}%
  \ifx\svgwidth\undefined%
    \setlength{\unitlength}{101.9101637bp}%
    \ifx\svgscale\undefined%
      \relax%
    \else%
      \setlength{\unitlength}{\unitlength * \real{\svgscale}}%
    \fi%
  \else%
    \setlength{\unitlength}{\svgwidth}%
  \fi%
  \global\let\svgwidth\undefined%
  \global\let\svgscale\undefined%
  \makeatother%
  \begin{picture}(1,0.68346063)%
    \lineheight{1}%
    \setlength\tabcolsep{0pt}%
    \put(0,0){\includegraphics[width=\unitlength,page=1]{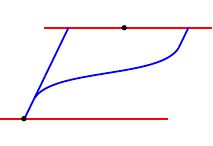}}%
    \put(0.57155176,0.60794979){\color[rgb]{0,0,0}\makebox(0,0)[lt]{\lineheight{1.25}\smash{\begin{tabular}[t]{l}$v_2$\end{tabular}}}}%
    \put(0.06974514,0.01729662){\color[rgb]{0,0,0}\makebox(0,0)[lt]{\lineheight{1.25}\smash{\begin{tabular}[t]{l}$v_1$\end{tabular}}}}%
    \put(0.07476652,0.33435012){\color[rgb]{0,0,1}\makebox(0,0)[lt]{\lineheight{1.25}\smash{\begin{tabular}[t]{l}$e_-$\end{tabular}}}}%
    \put(0.53137884,0.2628788){\color[rgb]{0,0,1}\makebox(0,0)[lt]{\lineheight{1.25}\smash{\begin{tabular}[t]{l}$e_+$\end{tabular}}}}%
  \end{picture}%
\endgroup%

    \caption{Arguing that for each edge $e$ in $T$, there is a negative path on $A$ isotopic to $e$ rel endpoints.}
    \label{fig:hsurcorrsolidtorusedges}
\end{figure}

We then want to arrange it so that each pair of such paths intersect efficiently in $A$, that is, they do not bound a bigon. We first arrange for this for the paths in $E^\pm_0$. At each vertex $v$ of $A$, we can label the edges in $E^\pm_0$ with an endpoint at $v$ as $e^\pm_1,...,e^\pm_n$ as in \Cref{lemma:vthsurtopbottomfaces}. 
After regluing by the surgery map, there are bigons formed between $e^+_i$ and $e^-_j$ for every $i \neq n$ and every $j \neq 1$. We push $e^+_{n-1}$ and $e^-_2$ across all such bigons. The resulting paths are still negative, and since, by \Cref{lemma:vthsurtopbottomfaces}, the face between $e^+_{n-1}$ and $e^+_n$ is a bottom face of the tetrahedron that has $e^+_{n-1}$ as the bottom edge, and the face between $e^-_1$ and $e^-_2$ is a top face of the tetrahedron that has $e^-_2$ as the top edge, all the faces remain positively transverse to the flow. See \Cref{fig:hsurcorrinsertsolidtorus} third to fourth row for an example in the simplest setting where $n=2$. We then push $e^+_{n-2}$ and $e^-_3$ across all remaining bigons, and so on.

To arrange for efficient intersection for the paths correpsonding to edges in the interior of $T$, we can simply push them across bigons. Since these are not edges of existing faces, we do not have to modify the existing triangulation accordingly.

We can now construct $T$. First we peel $E^-_0$ and $E^+_0$ apart on $A$ by isotoping them along orbits. Then using the fact that the triangulation on $T$ is layered (as recorded in \Cref{prop:vthsur}), we inductively define a map of each tetrahedron in $T$ into $\overline{M}$ by mapping the edges onto translates of their correpsonding paths constructed above. The fact that the paths intersect efficiently on $\overline{A}$ implies that one can arrange for the image of each face to be positively transverse to the flow. This builds $\Delta_{\frac{1}{n}}(c)$ in $\overline{M}$ in a way that places it in transverse position relative to $\overline{\phi}^t$. This concludes the proof of \Cref{thm:hsurcorr}.

\section{Proof of \Cref{thm:pennergenusone}} \label{sec:pennergenusoneproof}

\subsection{Simplifying the exponents of Dehn twists}

Suppose $\phi^t$ is an Anosov flow defined on an oriented closed $3$-manifold $M$ with orientable stable and unstable foliations. Suppose that $\phi^t$ admits a Birkhoff section whose first return map is a Penner type pseudo-Anosov map. Then there exists a collection $\mathcal{C}$ of closed orbits of $\phi^t$ such that $\phi^t$ restricted on $M \backslash \mathcal{C}$ is the suspension flow of a Penner type fully-punctured pseudo-Anosov map $\sigma \tau_{\alpha}^{n_1} \tau_{\beta}^{-m_1} \dots \tau_{\alpha}^{n_k} \tau _{\beta}^{-m_k}$ defined on some surface $S^\circ$.

Recall \Cref{prop:torusbundletovbs}: There exists an almost veering branched surface $B_\sigma$ on $T_\sigma$ and disjoint horizontal surgery curves $\alpha_i \times \{\frac{3s-2}{3k}\}$ and $\beta_j \times \{\frac{3s-1}{3k}\}$, so that given $p$-tuples $n_1,...,n_k$ of non-negative integers such that $\sum_s n_s$ is positive on every $\sigma$-orbit of $\alpha$, and $q$-tuples $m_1,...,m_k$ of non-negative integers such that $\sum_s m_s$ is positive on every $\sigma$-orbit of $\beta$, performing horizontal surgery along these curves gives a layered veering branched surface on the mapping torus of $\sigma \tau_{\alpha}^{n_1} \tau_{\beta}^{-m_1} \dots \tau_{\alpha}^{n_k} \tau _{\beta}^{-m_k}$. For notational convenience, we temporarily denote the map by $f_{n_s,m_s}$ and the veering branched surface as $B_{n_s,m_s}$.

For two choices of such tuples $n^{(1)}_1,m^{(1)}_1,...,n^{(1)}_{k^{(1)}},m^{(1)}_{k^{(1)}}$ and $n^{(2)}_1,m^{(2)}_1,...,n^{(2)}_{k^{(2)}},m^{(2)}_{k^{(2)}}$, we have a square of horizontal surgery operations
\begin{center}
\begin{tikzcd}
 & B_{\max\{n^{(1)}_s,n^{(2)}_s\},\max\{m^{(1)}_s,m^{(2)}_s\}} & \\
B_{n^{(1)}_s,m^{(1)}_s} \arrow[ru] & & B_{n^{(2)}_s,m^{(2)}_s} \arrow[lu] \\
 & B_\sigma \arrow[lu] \arrow[ru] &
\end{tikzcd}
\end{center}
Here, and in the rest of the proof, an arrow $B \to B'$ means that one performs horizontal surgery on $B$ along disjoint curves to obtain $B'$. Also, we take $n^{(i)}_s = 0$ and $m^{(i)}_s = 0$ for $s > k^{(i)}$ in order to make sense of the maximums.

Choosing $n^{(1)}_s$ and $m^{(1)}_s$ to be the tuples $n_s$ and $m_s$ coming from the starting Anosov flow $\phi^t$, and $n^{(2)}_s$ and $m^{(2)}_s$ to be the single $p$-tuple $(1,...,1)$ and the single $q$-tuple $(1,...,1)$ respectively, we can apply \Cref{thm:layeredvthsur} and \Cref{thm:horsuralmostequiv} to conclude that the flow $\phi^t_1$ corresponding to the veering triangulation dual to $B_{\max\{n^{(1)}_s,n^{(2)}_s\},\max\{m^{(1)}_s,m^{(2)}_s\}}$ is almost equivalent to $\phi^t$. The same reasoning shows that the flow $\phi^t_2$ corresponding to the veering triangulation dual to $B_{\max\{n^{(1)}_s,n^{(2)}_s\},\max\{m^{(1)}_s,m^{(2)}_s\}}$ is almost equivalent to the suspension flow $\phi^t_{\sigma \tau_{\alpha} \tau_{\beta}^{-1}}$. Meanwhile, \Cref{cor:uniqueflow} implies that $\phi^t_1$ and $\phi^t_2$ are almost equivalent. We conclude that $\phi^t$ is almost equivalent to $\phi^t_{\sigma \tau_{\alpha} \tau_{\beta}^{-1}}$.

Recall \Cref{lemma:anosovtypequotientsurface} and \Cref{prop:torusbundletopennermaptori}: $S^\circ \to S^\circ/\langle \sigma \rangle$ is a covering of $S^\circ$ over some genus one surface $T$ determined by some class $[z_0] \in H^1(T; \mathbb{Z}/N)$. There exists a sequence of cocycles $z_0,...,z_{N-1},z_N=0$ such that for each $i$, either
\begin{itemize}
    \item $z_{i-1}$ and $z_i$ differ only on a positive edge $e$ and $z_i(e) = z_{i-1}(e) + 1$, or
    \item $z_{i-1}$ and $z_i$ differ only on a negative edge $e$ and $z_i(e) = z_{i-1}(e) - 1$.
\end{itemize}
For each $i$, we let $f_i$ be the map $\sigma_{[z_i]} \tau_{\alpha_{[z_i]}} \tau^{-1}_{\beta_{[z_i]}}$ defined on the covering surface $S^\circ_{[z_i]}$ determined by $[z_i]$. In particular, we have already seen that the starting Anosov flow is almost equivalent to the suspension flow $\phi^t_{f_0}$. 

\subsection{Varying the cocycles}

The rest of the proof is devoted to showing that $\phi^t_{f_{i-1}}$ and $\phi^t_{f_i}$ are almost equivalent for each $i$. Once we show this, we can conclude that $\phi^t$ and $\phi^t_{f_N}$ are almost equivalent. Since $z_N=0$, $S_{z_N}$ is a disjoint union of genus one surfaces, thus $\phi^t_{f_N}$ is a suspension Anosov flow. Hence by \Cref{prop:genusonesection=sus}, $\phi^t$ admits a genus one Birkhoff section.

Without loss of generality, suppose that $z_{i-1}$ and $z_i$ differ only on a positive edge $e$ and $z_i(e) = z_{i-1}(e) + 1$. We identify $e$ with the branch of $\tau$ passing through the edge of the quadrangulation on $T$ that is dual to $e$. Then $e \times S^1$ is a sector of $B_\sigma$. 

Let $C_{i-1}$ and $C_i$ be the collections of surgery curves on $B_\sigma$ for which surgery along gives the veering branched surfaces $B_{f_{i-1}}$ and $B_{f_i}$ respectively. Note that the sets $\bigcup C_{i-1}$ and $\bigcup C_i$ differ only on $e \times S^1$. We let $D_{i-1}$ and $D_i$ be the subsets of $C_{i-1}$ and $C_i$ consisting of elements that pass through $e \times S^1$. Note that these must consist of positive horizontal surgery curves. We isotope the elements of $D_i$ so that they lie slightly above the elements of $D_{i-1}$ outside of $e \times S^1$ and so that they intersect the elements of $D_{i-1}$ transversely in a minimal number of points in $e \times S^1$. See \Cref{fig:concurrenthsur} bottom.

For each element $c$ in $C_{i-1} \cup C_i$, we choose a tubular neighborhood $N_c$ of $c$ in $T \times S^1$. 
Let $A_c$ be the smooth neighborhood of $c$ that is contained in $N_c$.
The orientation on $\alpha_T$ induces an orientation on $c$, allowing us to refer to the components of $N_c \cut A_c$ as the left or right regular half-neighborhood of $c$, which we denote by $N^L_c$ and $N^R_c$ respectively.

The key observation now is that we can obtain a veering branched surface $\widehat{B}_i$ by
\begin{itemize}
    \item for each positive or negative $c \in C_{i-1} \cap C_i$, cutting out $N^L_c$ and regluing it with a map that is identity on $\partial N^L_c \backslash A_c$ and sends the meridian to a curve of isotopy class $\mu+c$ or $\mu-c$, respectively, such that the arcs in $\brloc(N^L_c \cap (\tau \times S^1))$ intersect those in $\brloc(N^R_c \cap (\tau \times S^1))$ minimally;
    \item for each $c \in D_{i-1}$, cutting out $N^L_c$ and regluing it with a map that is identity on $\partial N^L_c \backslash A_c$ and sends the meridian to a curve of isotopy class $\mu+c$, such that the arcs in $\brloc(N^L_c \cap (\tau \times S^1))$ intersect those in $\brloc(N^R_c \cap (\tau \times S^1))$ minimally;
    \item for each $c \in D_i$, cutting out $N^R_c$ and regluing it with a map that is identity on $\partial N^R_c \backslash A_c$ and sends the meridian to a curve of isotopy class $\mu+c$, such that the arcs in $\brloc(N^R_c \cap (\tau \times S^1))$ intersect those in $\brloc(N^L_c \cap (\tau \times S^1))$ and those in $\brloc(N^L_{c'} \cap (\tau \times S^1))$ for $c' \in D_{i-1}$ minimally. 
\end{itemize}
See \Cref{fig:concurrenthsur}. Indeed, by only performing the first two types of cutting and regluing, we end up with the veering branched surface $B_{f_{i-1}}$ having $D_i$ as a collection of positive horizontal surgery curves. Hence by performing the final type of cutting and regluing, we are performing horizontal surgery on $B_{f_{i-1}}$ along $D_i$, which by \Cref{prop:vbshsur}, results in a veering branched surface $\widehat{B}_i$. By symmetry, we can also consider $\widehat{B}_i$ as obtained by horizontal surgery on $B_{f_i}$. In other words, we have the following square
\begin{center}
\begin{tikzcd}
 & \widehat{B}_i & \\
B_{f_{i-1}} \arrow[ru] & & B_{f_i} \arrow[lu] \\
 & B_\sigma \arrow[lu] \arrow[ru] & 
\end{tikzcd}
\end{center}

\begin{figure}
    \centering
    \fontsize{8pt}{8pt}\selectfont
    %% Creator: Inkscape 1.3.2 (091e20e, 2023-11-25, custom), www.inkscape.org
%% Accompanies image file 'concurrenthsur.pdf' (pdf, eps, ps)
\begingroup%
  \makeatletter%
  \providecommand\color[2][]{%
    \errmessage{(Inkscape) Color is used for the text in Inkscape, but the package 'color.sty' is not loaded}%
    \renewcommand\color[2][]{}%
  }%
  \providecommand\transparent[1]{%
    \errmessage{(Inkscape) Transparency is used (non-zero) for the text in Inkscape, but the package 'transparent.sty' is not loaded}%
    \renewcommand\transparent[1]{}%
  }%
  \providecommand\rotatebox[2]{#2}%
  \newcommand*\fsize{\dimexpr\f@size pt\relax}%
  \newcommand*\lineheight[1]{\fontsize{\fsize}{#1\fsize}\selectfont}%
  \ifx\svgwidth\undefined%
    \setlength{\unitlength}{253.44412255bp}%
    \ifx\svgscale\undefined%
      \relax%
    \else%
      \setlength{\unitlength}{\unitlength * \real{\svgscale}}%
    \fi%
  \else%
    \setlength{\unitlength}{\svgwidth}%
  \fi%
  \global\let\svgwidth\undefined%
  \global\let\svgscale\undefined%
  \makeatother%
  \begin{picture}(1,1.05607063)%
    \lineheight{1}%
    \setlength\tabcolsep{0pt}%
    \put(0,0){\includegraphics[width=\unitlength,page=1]{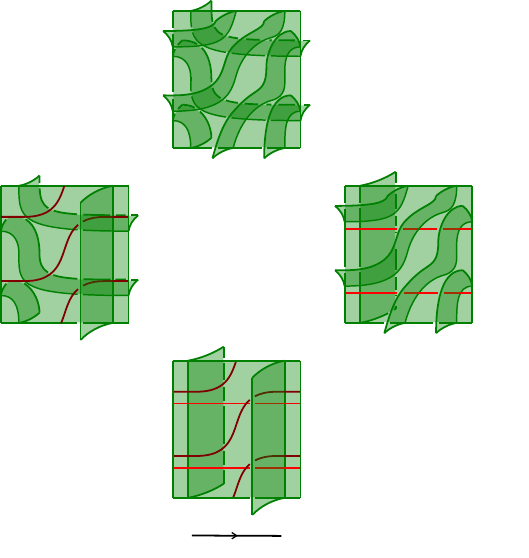}}%
    \put(0.27043019,0.57879233){\color[rgb]{0,0,0}\makebox(0,0)[lt]{\lineheight{1.25}\smash{\begin{tabular}[t]{l}$B_{f_{i-1}}$\end{tabular}}}}%
    \put(0.243447,0.2783329){\color[rgb]{1,0,0}\makebox(0,0)[lt]{\lineheight{1.25}\smash{\begin{tabular}[t]{l}$D_{i-1}$\end{tabular}}}}%
    \put(0.27304276,0.3163603){\color[rgb]{0.50196078,0,0}\makebox(0,0)[lt]{\lineheight{1.25}\smash{\begin{tabular}[t]{l}$D_i$\end{tabular}}}}%
    \put(0.43371345,0.00278199){\color[rgb]{0,0,0}\makebox(0,0)[lt]{\lineheight{1.25}\smash{\begin{tabular}[t]{l}$e$\end{tabular}}}}%
    \put(0.92022399,0.57879336){\color[rgb]{0,0,0}\makebox(0,0)[lt]{\lineheight{1.25}\smash{\begin{tabular}[t]{l}$B_{f_i}$\end{tabular}}}}%
    \put(0.59899943,0.91265146){\color[rgb]{0,0,0}\makebox(0,0)[lt]{\lineheight{1.25}\smash{\begin{tabular}[t]{l}$\widehat{B}_i$\end{tabular}}}}%
    \put(0,0){\includegraphics[width=\unitlength,page=2]{concurrenthsur.pdf}}%
  \end{picture}%
\endgroup%

    \caption{The veering branched surface $B_{f_{i-1}}$ and $B_{f_i}$ are related by horizontal surgeries via $\widehat{B}_i$.}
    \label{fig:concurrenthsur}
\end{figure}

\begin{rmk} \label{rmk:concurrenthsur}
We remark that the construction of $\widehat{B}_i$ actually works whenever $z_{i-1}$ and $z_i$ differ only on a positive edge $e$ and $z_i(e) = z_{i-1}(e) + k$ for some $k>0$. We choose to simplify to the case when $k=1$ for easier visualization. Note that the construction fails to give a veering branched surface if $k<0$, however. In that case one can modify the construction by cutting out the right regular half-neighborhoods of elements in $D_{i-1}$ and left regular half-neighborhoods of elements in $D_i$ instead.

Another remark is that this construction fits into the framework of what might be called concurrent horizontal surgery in \cite{Tsa22a}. To elaborate more: In \cite{Tsa22a} a `halved' variant of horizontal surgery is discussed, which is then generalized to a `concurrent halved' variant, where one performs halved horizontal surgery along multiple intersecting curves. It is straightforward to remove the `halved' aspect of concurrent halved horizontal surgery and develop simply a `concurrent' variant of horizontal surgery. We let the interested reader fill out the details.
\end{rmk}

Arguing as above, \Cref{thm:layeredvthsur} and \Cref{thm:horsuralmostequiv} imply that $\phi^t_{f_{i-1}}$ and $\phi^t_{f_i}$ differ by horizontal Goodman surgeries, hence are almost equivalent. One difference between this setting and the one in the previous subsection is that $B_{\max\{n^{(1)}_s,n^{(2)}_s\},\max\{m^{(1)}_s,m^{(2)}_s\}}$ is layered while $\widehat{B}_i$ is not. This is not an issue however, since \Cref{thm:layeredvthsur} only requires the veering branched surface \emph{before} the horizontal surgery to be layered. 

The following diagram summarizes the whole argument. 
\begin{center}
\begin{scriptsize}
\begin{tikzpicture}[x=0.75pt,y=0.75pt,yscale=-1,xscale=1]
\tikzstyle{arrow} = [thick,->,>=stealth]

\draw (0,80) node [anchor=west] {$\phi^t_{n^{(1)}_s,m^{(1)}_s}$};
\draw [arrow] (40,70) -- (60,50);
\draw (30,40) node [anchor=west] {$\phi^t_{\max\{n^{(1)}_s,n^{(2)}_s\},\max\{m^{(1)}_s,m^{(2)}_s\}}$};
\draw [arrow] (180,70) -- (160,50);
\draw (175,80) node [anchor=west] {$\phi^t_{n^{(2)}_s,m^{(2)}_s} = \phi^t_{f_0}$};

\draw [arrow] (260,70) -- (280,50) ;
\draw (280,40) node [anchor=west] {$\widetilde{\phi}^t_1$};
\draw [arrow] (320,70) -- (300,50) ;
\draw (315,80) node [anchor=west] {$\phi^t_{f_1}$};

\draw [arrow] (335,70) -- (355,50) ;
\draw (360,50) node [anchor=west] {$\dots$};
\draw [arrow] (405,70) -- (385,50) ;
\draw (400,80) node [anchor=west] {$\phi^t_{f_{N-1}}$};

\draw [arrow] (420,70) -- (440,50) ;
\draw (440,40) node [anchor=west] {$\widetilde{\phi}^t_N$};
\draw [arrow] (480,70) -- (460,50) ;
\draw (475,80) node [anchor=west] {$\phi^t_{f_N}$};

\draw (0,200) node [anchor=west] {$B_{n^{(1)}_s,m^{(1)}_s}$};
\draw [arrow] (40,190) -- (60,170);
\draw (30,160) node [anchor=west] {$B_{\max\{n^{(1)}_s,n^{(2)}_s\},\max\{m^{(1)}_s,m^{(2)}_s\}}$};
\draw [arrow] (180,190) -- (160,170);
\draw (175,200) node [anchor=west] {$B_{n^{(2)}_s,m^{(2)}_s} = B_{f_0}$};
\draw [arrow] (60,230) -- (40,210);
\draw [arrow] (160,230) -- (180,210);
\draw (85,240) node [anchor=west] {$B_\sigma$};

\draw [arrow] (260,190) -- (280,170);
\draw (280,160) node [anchor=west] {$\widetilde{B}_1$};
\draw [arrow] (320,190) -- (300,170);
\draw (315,200) node [anchor=west] {$B_{f_1}$};
\draw [arrow] (280,230) -- (260,210);
\draw [arrow] (305,230) -- (325,210);
\draw (275,240) node [anchor=west] {$B_\sigma$};

\draw [arrow] (335,190) -- (355,170) ;
\draw (360,170) node [anchor=west] {$\dots$};
\draw [arrow] (405,190) -- (385,170) ;
\draw (400,200) node [anchor=west] {$B_{f_{N-1}}$};
\draw [arrow] (360,230) -- (340,210);
\draw [arrow] (385,230) -- (405,210);
\draw (360,230) node [anchor=west] {$\dots$};

\draw [arrow] (420,190) -- (440,170) ;
\draw (440,160) node [anchor=west] {$\widetilde{B}_N$};
\draw [arrow] (480,190) -- (460,170) ;
\draw (475,200) node [anchor=west] {$B_{f_N}$};
\draw [arrow] (440,230) -- (420,210);
\draw [arrow] (460,230) -- (480,210);
\draw (430,240) node [anchor=west] {$B_\sigma$};

\draw (515,200) node [anchor=west] {Veering, layered};
\draw (515,240) node [anchor=west] {Almost veering};

\draw (240,180) -- (240,140);
\draw (240,180) -- (600,180);
\draw (10,220) -- (600,220);

\end{tikzpicture}
\end{scriptsize}
\end{center}

\vspace{0.5cm}
\begin{large}
\begin{center}
\textbf{Part 3. Totally periodic Anosov flows}
\end{center}
\end{large}

\section{Base surfaces and graphs} \label{sec:totallyperiodicdefn}

Totally periodic (pseudo-)Anosov flows are a class of (pseudo-)Anosov flows introduced by Barbot and Fenley in \cite{BF15}. We recall their definition below.

\begin{defn} \label{defn:totallyperiodic}
Let $M$ be a closed oriented graph manifold. Let $\{P_i\}$ be the Seifert fibered pieces of $M$. A (pseudo-)Anosov flow $\phi^t$ on $M$ is said to be \textbf{totally periodic} if for each $i$, a regular fiber of $P_i$ is homotopic to a closed orbit of $\phi^t$.
\end{defn}

A feature of totally periodic (pseudo-)Anosov flows is that their dynamical data can be recorded by certain combinatorial data, namely the base surfaces of the Seifert fibered pieces $P_i$, graphs sitting inside these base surfaces, and the gluing maps between the boundary components of the $P_i$. In \Cref{subsec:flowgraph} we recall this theory in the case when the flow is Anosov and has orientable stable and unstable foliations. 
We refer the reader to \cite{BF15} and \cite{BFM23} for proofs.

In \Cref{subsec:indgluingmap} we recall a corollary of \Cref{thm:horsuralmostequiv} that shows that in the context of showing \Cref{thm:totallyperiodicgenusone}, the data of the gluing maps is unimportant.

\subsection{Spines and graphs} \label{subsec:flowgraph}

Let $M$ be a closed oriented graph manifold. Let $\{P_i\}$ be the Seifert fibered pieces of $M$. Let $S_i$ be the base orbifold of $P_i$ and $\pi_i:P_i \to S_i$ be the projection map. We caution that even though $M$ is oriented, each $S_i$ may be nonorientable. We refer to the (isotopy classes of) identifications between $\bigsqcup_i \partial P_i$ that recover $M$ as the \textbf{gluing maps}.

Let $\phi^t$ be a totally periodic Anosov flow on $M$. Suppose that $\phi^t$ has orientable stable and unstable foliations. Recall that this is equivalent to the line bundle $E^s$ and $E^u$ being orientable. We fix a choice of orientation on these line bundles so that $TM \cong E^s \oplus T\phi \oplus E^u$ is orientation preserving.

\begin{defn}
A \textbf{Birkhoff annulus} is a cooriented annulus $A$ such that:
\begin{itemize}
    \item $A$ is embedded along its interior, where it is positively transverse to the orbits of $\phi^t$.
    \item $A$ is immersed along its boundary, where it lies along closed orbits of $\phi^t$.
\end{itemize}

We say that a Birkhoff annulus $A$ is \textbf{positive/negative} if there exists a positive/negative nonseparating arc on $A$, respectively. See \Cref{fig:totallyperiodicspine} top for an example of a negative Birkhoff annulus.
\end{defn}

\begin{figure}
    \centering
    \fontsize{6pt}{6pt}\selectfont
    %% Creator: Inkscape 1.3.2 (091e20e, 2023-11-25, custom), www.inkscape.org
%% Accompanies image file 'totallyperiodicspine.pdf' (pdf, eps, ps)
\begingroup%
  \makeatletter%
  \providecommand\color[2][]{%
    \errmessage{(Inkscape) Color is used for the text in Inkscape, but the package 'color.sty' is not loaded}%
    \renewcommand\color[2][]{}%
  }%
  \providecommand\transparent[1]{%
    \errmessage{(Inkscape) Transparency is used (non-zero) for the text in Inkscape, but the package 'transparent.sty' is not loaded}%
    \renewcommand\transparent[1]{}%
  }%
  \providecommand\rotatebox[2]{#2}%
  \newcommand*\fsize{\dimexpr\f@size pt\relax}%
  \newcommand*\lineheight[1]{\fontsize{\fsize}{#1\fsize}\selectfont}%
  \ifx\svgwidth\undefined%
    \setlength{\unitlength}{139.75286289bp}%
    \ifx\svgscale\undefined%
      \relax%
    \else%
      \setlength{\unitlength}{\unitlength * \real{\svgscale}}%
    \fi%
  \else%
    \setlength{\unitlength}{\svgwidth}%
  \fi%
  \global\let\svgwidth\undefined%
  \global\let\svgscale\undefined%
  \makeatother%
  \begin{picture}(1,1.18249802)%
    \lineheight{1}%
    \setlength\tabcolsep{0pt}%
    \put(0,0){\includegraphics[width=\unitlength,page=1]{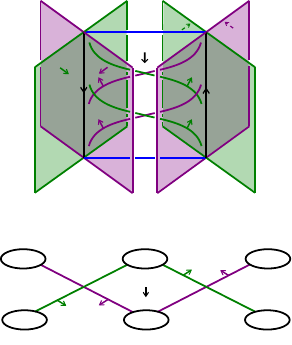}}%
    \put(0.05594547,0.34572955){\color[rgb]{0,0,0}\makebox(0,0)[lt]{\lineheight{1.25}\smash{\begin{tabular}[t]{l}$+$\end{tabular}}}}%
    \put(0.47677668,0.34571965){\color[rgb]{0,0,0}\makebox(0,0)[lt]{\lineheight{1.25}\smash{\begin{tabular}[t]{l}$-$\end{tabular}}}}%
    \put(0.89333434,0.34571965){\color[rgb]{0,0,0}\makebox(0,0)[lt]{\lineheight{1.25}\smash{\begin{tabular}[t]{l}$+$\end{tabular}}}}%
    \put(0.05593944,0.00757769){\color[rgb]{0,0,0}\makebox(0,0)[lt]{\lineheight{1.25}\smash{\begin{tabular}[t]{l}$-$\end{tabular}}}}%
    \put(0.4767708,0.00756779){\color[rgb]{0,0,0}\makebox(0,0)[lt]{\lineheight{1.25}\smash{\begin{tabular}[t]{l}$+$\end{tabular}}}}%
    \put(0.89332831,0.00756779){\color[rgb]{0,0,0}\makebox(0,0)[lt]{\lineheight{1.25}\smash{\begin{tabular}[t]{l}$-$\end{tabular}}}}%
    \put(0,0){\includegraphics[width=\unitlength,page=2]{totallyperiodicspine.pdf}}%
  \end{picture}%
\endgroup%

    \caption{For each edge $e$ of the spine graph $X_i$, $\pi_i^{-1}(e)$ is a Birkhoff annulus of $\phi^t$.}
    \label{fig:totallyperiodicspine}
\end{figure}

In \cite{BF15}, it is shown that up to isotoping $\phi^t$, there is a graph $X_i$ embedded in $S_i$ for each $i$, so that for each vertex $v$ of $X_i$, $\pi_i^{-1}(v)$ is a closed orbit of $\phi^t$, and for each edge $e$ of $X_i$, $\pi_i^{-1}(e)$ is a Birkhoff annulus of $\phi^t$. Moreover, $P_i$ deformation retracts onto the set $\pi_i^{-1}(X_i)$ via homotoping along orbits of $\phi^t$. See \Cref{fig:totallyperiodicspine}. Note that since we have assumed that $\phi^t$ is Anosov and has orientable stable and unstable foliations, each $X_i$ is $4$-valent. This in particular implies that each $S_i$ is a surface, i.e. $P_i$ has no singular fibers, thus $P_i \cong S_i \times S^1$.

We refer to the orbits of $\phi^t$ that lie over the vertices of $X_i$ as the \textbf{vertical orbits}.
Since $\phi^t$ has orientable stable and unstable leaves, the local stable and unstable leaves of each vertical orbit within $P_i$ are properly embedded annuli. Up to isotoping $\phi^t$, we can assume that each of these is of the form $\pi_i^{-1}(\alpha)$ where $\alpha$ is a properly embedded interval in $S_i$. 
The chosen coorientations on the stable and unstable leaves induce coorientations on these intervals. We let $F_i$ be the union of these cooriented intervals. See \Cref{fig:totallyperiodicgraphs} middle.

The edges of $X_i$ are naturally cooriented by the direction of the flow. 
We orient the edges of $X_i$ using the coorientation on the projections of the stable leaves in $F_i$.
We also say that an edge of $X_i$ is \textbf{positive/negative} if it is the projection of a positive/negative Birkhoff annulus, respectively. 
We refer to the $X_i$ with the data of these coorientations, orientations, and signs as the \textbf{spine graphs} of $\phi^t$. See \Cref{fig:totallyperiodicgraphs} top.

\begin{figure}
    \centering
    \fontsize{6pt}{6pt}\selectfont
    \resizebox{!}{7.5cm}{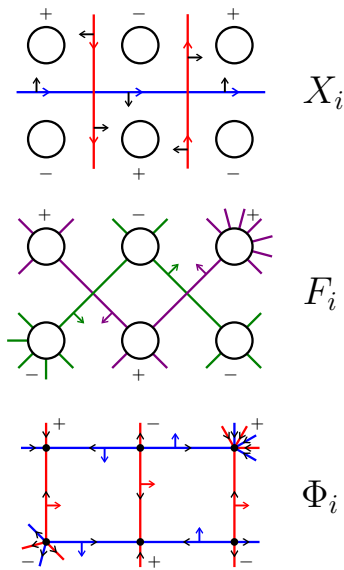}
    \caption{Top: The spine graph $X_i$ on $S_i$. Middle: The graph $F_i$ formed by projections of local stable and unstable leaves of the vertical orbits. Bottom: The flow graph $\Phi_i$ on $\overline{S}_i$. As usual, we color positive/negative edges in red/blue respectively. We indicate orientations and coorientations of edges that are induced by the flow in black, and those that are induced by the choice of orientations on the stable and unstable foliations in the same color as the edge.}
    \label{fig:totallyperiodicgraphs}
\end{figure}

Meanwhile, let $T$ be a boundary component of $P_i$. Observe that either $\phi^t$ points into $P_i$ at every point of $T$, or $\phi^t$ points out of $P_i$ at every point of $T$. We say that $T$ is an \textbf{inward} or \textbf{outward} boundary component of $P_i$ respectively.

An important observation is that in either case, $T$ is a scalloped transverse torus. We recall the definition of these below.

\begin{defn} \label{defn:scalloptorus}
Let $T$ be a transverse torus. Let $T^{s/u}$ be the induced stable/unstable foliations on $T$. We say that $T$ is a \textbf{scalloped transverse torus} if 
\begin{itemize}
    \item Each of $T^s$ and $T^u$ contains a number of closed leaves.
    \item Each closed leaf of $T^s$ and $T^u$ can be given an orientation, which we refer to as the \textbf{spiraling orientation}, so that
    \begin{itemize}
        \item each non-closed $T^{s/u}$-leaf spirals into a closed $T^{s/u}$-leaf in their spiraling orientations on each of its ends, and
        \item adjacent closed $T^{s/u}$-leaves have opposite spiraling orientations.
    \end{itemize}
    \item The closed $T^s$-leaves are not parallel to the closed $T^u$-leaves.
\end{itemize}
See \Cref{fig:scalloptorusfol} for an illustration of such a pair of foliations $(T^s, T^u)$.
\end{defn}

\begin{figure}
    \centering
    \resizebox{!}{3cm}{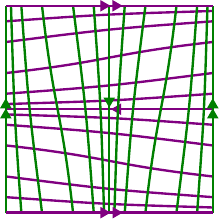}
    \caption{The form of the foliations $(T^s, T^u)$ on a scalloped transverse torus.}
    \label{fig:scalloptorusfol}
\end{figure}

\begin{lemma} \label{lemma:totallyperiodicboundarytorusscallop}
Let $T$ be a boundary component of some Seifert fibered piece. Then $T$ is a scalloped transverse torus.
\end{lemma}
\begin{proof}
This is contained within \cite{BF15}. For an explicit proof, one can reason as follows: $T$ is a scalloped torus in the sense of \cite[Definition 2.13]{BFM23}, since a closed $T^{s/u}$-leaf of an inward/outward boundary component is parallel to a Seifert fiber, and the homology classes of these differ on each pair of glued tori. As explained in \cite[Example 3.11]{Tsa24}, the notion of a scalloped torus in \cite{BFM23} agrees with \Cref{defn:scalloptorus}.
\end{proof}

Suppose $T \cong c \times S^1$ is an inward/outward boundary component, where $c$ is a boundary component of $S_i$. Then $c$ retracts onto a cycle $c'$ of $X_i$ consisting of edges cooriented away/towards $c$, respectively. We define the homology class $[\ell^{s/u}]$ on $T$ to be the homology class of a closed leaf of the induced stable/unstable foliation on $T$, respectively. This is also the homology class of a fiber of $P_i$ with coorientation induced from the orientation of a positive edge on $c'$.

Let $T_+$ be an outward boundary component and $T_-$ be an inward boundary component which are glued in $M$. Then upon identifying $H_1(T_\pm; \mathbb{R})$ using the gluing map, $([\ell^s],[\ell^u])$ is a positive basis.

It turns out that the spine graphs and the gluing maps suffice to recover the flow $\phi^t$. More precisely,

\begin{thm}[{\cite{BF15}}] \label{thm:bftotallyperiodic}
Given $4$-valent graphs $X_i$ embedded in surfaces with boundary $S_i$ with a choice of coorientation, orientation, and sign on each edge, such that:
\begin{itemize}
    \item Each $X_i$ is a deformation retract of $S_i$.
    \item Adjacent edges at a vertex have opposite coorientations and signs.
    \item Opposite edges at a vertex have opposite orientations.
\end{itemize}
We set $P_i = S_i \times S^1$ and refer to a boundary component $c \times S^1$ to be inward/outward if it retracts onto a cycle $c'$ of $X_i$ consisting of edges cooriented away/towards $c$, respectively. In this case, we define the homology class $[\ell^{s/u}]$ on $T$ to be the homology class of a fiber of $P_i$ with coorientation induced from the orientation of a positive edge on $c'$. 

Suppose that we are given a pairing between the inward and outward boundary components of $\bigsqcup_i P_i$, and for each pair $(T_+,T_-)$ we are given a isotopy class of maps $T_- \cong T_+$ so that upon identifying $H_1(T_\pm; \mathbb{R})$ using the gluing map, $([\ell^s],[\ell^u])$ is a positive basis.

Then there is a unique totally periodic Anosov flow $\phi^t$ defined on the graph manifold $M = \bigcup_i P_i$ with oriented stable and unstable foliations so that the spine graphs of $\phi^t$ are $X_i$ and the gluing maps of $M$ are those chosen.
\end{thm}

This provides a neat combinatorial representation of all totally periodic Anosov flows. However, for our purposes, it will be more convenient to record this data in an alternate format. We describe this for the rest of this subsection. 

Let $\overline{S}_i$ be the closed surface obtained by collapsing each boundary component of $S_i$ to a point. We refer to the image of an inward/outward boundary component as an \textbf{inward/outward joint}, respectively.
Let $\overline{F}_i$ be the image of $F_i$ in $\overline{S}_i$. We can consider $\overline{F}_i$ as a graph whose vertices are the vertices of $X_i$ and the joints.
Observe that $\overline{F}_i$ is $4$-valent at the vertices of $X_i$. 
From \Cref{lemma:totallyperiodicboundarytorusscallop}, we see that $\overline{F}_i$ is of even valence at each joint as well.

Each complementary region of $\overline{F}_i$ in $\overline{S}_i$ is a quadrilateral which has a pair of opposite vertices with one vertex lying on an inward joint and the other lying on an outward joint, and which contains a unique edge of $X_i$ separating this pair of vertices.
We define the \textbf{flow graph} $\Phi_i$ on $\overline{S}_i$ by setting the vertices to be the joints and adding one directed edge within each complementary region $C$ of $\overline{F}_i$, going from the inward joint to the outward joint, passing through the edge of $X_i$ in $C$ once. 
We coorient each edge of $\Phi_i$ using the orientation on the edges of $X_i$. We also say that an edge of $\Phi_i$ is \textbf{negative/positive} if it passes through a positive/negative edge of $X_i$ respectively. See \Cref{fig:totallyperiodicgraphs} bottom.

Observe that $\Phi_i$ on $\overline{S}_i$ contains the same amount of data as $X_i$ on $S_i$. In fact, $\Phi_i$ is the dual quandrangulation to $X_i$, with dual orientation/coorientation and with dual edges having opposite sign. As a result, by \Cref{thm:bftotallyperiodic}, we can recover the flow $\phi^t$ from the $\Phi_i$ and the gluing maps. 

The advantage of working with $\Phi_i$ instead of $X_i$ is that the dynamics of the flow can be read off more easily. To illustrate this, we show \Cref{prop:totallyperiodictransitive} below. On the way, we record some other handy properties.

\begin{lemma} \label{lemma:flowgrapheulerchar}
The flow graph $\Phi_i$ divides the closed surface $\overline{S}_i$ into quadrilaterals, each with two opposite positive sides and two opposite negative sides. The number of quadrilaterals equals to the number of vertical orbits in $P_i$, which is in turn equal to $-\chi(S_i)$.
\end{lemma}
\begin{proof}
The first statement is clear from construction. For the second statement, notice that each quadrilateral contains exactly one vertex of $X_i$, hence the number of quadrilaterals equals the number of vertical orbits. Meanwhile, the cellulation of $\overline{S}_i$ by quadrilaterals pulls back to a cellulation of $S_i$ by octagons. The index of an octagon is $-1$, thus by additivity of the index, the number of octagons is $-\chi(S_i)$.
\end{proof}

\begin{lemma} \label{lemma:flowgraphedges=orbitsegment}
There is an edge $e$ of $\Phi_i$ from an inward joint $v_-$ to an outward joint $v_+$ in an isotopy class $[e]$ rel $v_\pm$ if and only if there is an orbit segment $\gamma$ of $\phi^t$ from the inward boundary component of $P_i$ corresponding to $v_-$ to the outward boundary component of $P_i$ corresponding to $v_+$ such that $\pi_i([\gamma]) = [e]$.
\end{lemma}
\begin{proof}
The orbit segments in the complement of the local stable and unstable leaves of the vertical orbits all go from an inward boundary component $T_-$ to an outward boundary component $T_+$. Moreover, the isotopy class of these orbit segments rel $T_\pm$ remain constant in a single component of the complement. By construction, these isotopy classes projects down to the isotopy classes of the edges of $\Phi_i$ rel endpoints.
\end{proof}

\begin{defn} \label{defn:totalflowgraph}
The \textbf{total flow graph} $\Phi$ is the directed graph constructed by starting with the disjoint union of directed graphs $\bigsqcup_i \Phi_i$ and identifying every pair of joints $(j_+,j_-)$ corresponding to a pair $(T_+,T_-)$ of boundary components that are glued in $M$. See \Cref{fig:totalflowgraph}.

\begin{figure}
    \centering
    \fontsize{10pt}{10pt}\selectfont
    %% Creator: Inkscape 1.3.2 (091e20e, 2023-11-25, custom), www.inkscape.org
%% Accompanies image file 'totalflowgraph.pdf' (pdf, eps, ps)
\begingroup%
  \makeatletter%
  \providecommand\color[2][]{%
    \errmessage{(Inkscape) Color is used for the text in Inkscape, but the package 'color.sty' is not loaded}%
    \renewcommand\color[2][]{}%
  }%
  \providecommand\transparent[1]{%
    \errmessage{(Inkscape) Transparency is used (non-zero) for the text in Inkscape, but the package 'transparent.sty' is not loaded}%
    \renewcommand\transparent[1]{}%
  }%
  \providecommand\rotatebox[2]{#2}%
  \newcommand*\fsize{\dimexpr\f@size pt\relax}%
  \newcommand*\lineheight[1]{\fontsize{\fsize}{#1\fsize}\selectfont}%
  \ifx\svgwidth\undefined%
    \setlength{\unitlength}{319.53616405bp}%
    \ifx\svgscale\undefined%
      \relax%
    \else%
      \setlength{\unitlength}{\unitlength * \real{\svgscale}}%
    \fi%
  \else%
    \setlength{\unitlength}{\svgwidth}%
  \fi%
  \global\let\svgwidth\undefined%
  \global\let\svgscale\undefined%
  \makeatother%
  \begin{picture}(1,0.3326601)%
    \lineheight{1}%
    \setlength\tabcolsep{0pt}%
    \put(0,0){\includegraphics[width=\unitlength,page=1]{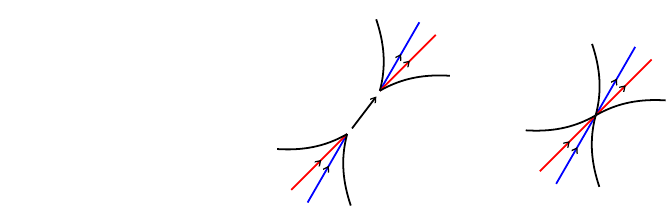}}%
    \put(0.10872601,0.27526947){\color[rgb]{0,0,0}\makebox(0,0)[lt]{\lineheight{1.25}\smash{\begin{tabular}[t]{l}$S_j$\end{tabular}}}}%
    \put(0.18398527,0.052161){\color[rgb]{0,0,0}\makebox(0,0)[lt]{\lineheight{1.25}\smash{\begin{tabular}[t]{l}$S_i$\end{tabular}}}}%
    \put(0.44057978,0.24380933){\color[rgb]{0,0,0}\makebox(0,0)[lt]{\lineheight{1.25}\smash{\begin{tabular}[t]{l}$\Phi_j \subset \overline{S}_j$\end{tabular}}}}%
    \put(0.5327834,0.06510309){\color[rgb]{0,0,0}\makebox(0,0)[lt]{\lineheight{1.25}\smash{\begin{tabular}[t]{l}$\Phi_i \subset \overline{S}_i$\end{tabular}}}}%
    \put(0.5757829,0.17124155){\color[rgb]{0,0,0}\makebox(0,0)[lt]{\lineheight{1.25}\smash{\begin{tabular}[t]{l}$j_-$\end{tabular}}}}%
    \put(0.53155769,0.11960422){\color[rgb]{0,0,0}\makebox(0,0)[lt]{\lineheight{1.25}\smash{\begin{tabular}[t]{l}$j_+$\end{tabular}}}}%
    \put(0.90966502,0.11572818){\color[rgb]{0,0,0}\makebox(0,0)[lt]{\lineheight{1.25}\smash{\begin{tabular}[t]{l}$\Phi$\end{tabular}}}}%
    \put(0,0){\includegraphics[width=\unitlength,page=2]{totalflowgraph.pdf}}%
  \end{picture}%
\endgroup%

    \caption{The total flow graph $\Phi$ is obtained by starting with the disjoint union $\bigsqcup_i \Phi_i$ and identifying every pair of joints $(j_+,j_-)$ corresponding to a pair of boundary components that are glued in $M$.}
    \label{fig:totalflowgraph}
\end{figure}

We can consider $\Phi$ to be embedded in the space obtained by identifying pairs of joints $(j_+,j_-)$ in the union of surfaces $\bigsqcup_i \overline{S}_i$. This also justifies our terminology of `joints'. With this embedding, the coorientations and signs of edges in $\bigsqcup_i \Phi_i$ induces coorientations and signs of edges in $\Phi$. We consider these as part of the data of $\Phi$.

For emphasis, we will sometimes refer to the flow graphs $\Phi_i$ as the \textbf{individual} flow graphs. Note that these can be recovered from the total flow graph $\Phi$ by separating the incoming and outgoing edges at each vertex.
\end{defn}

\begin{prop} \label{prop:totallyperiodictransitive}
The totally periodic Anosov flow $\phi^t$ is transitive if and only if the total flow graph $\Phi$ is strongly connected.
\end{prop}
\begin{proof}
If $\phi^t$ is transitive, then there is a closed orbit $\gamma$ passing through all the glued tori. The orbit segments of $\gamma$ lying within each Seifert fibered piece induces a cycle of $\Phi$ by \Cref{lemma:flowgraphedges=orbitsegment}. That $\gamma$ passes through all the glued tori means that $\gamma$ visits every vertex of $\Phi$. Therefore $\Phi$ is strongly connected.

Conversely, if $\phi^t$ is not transitive, then there is a submanifold $N$ of $M$ that is a sink for $\phi^t$. In \cite{Bru93}, it is shown that $\partial N$ consists of essential tori. Hence up to isotopy along orbits, we can assume that each component of $\partial N$ lies within a Seifert fibered piece. 

Let $V$ denote the collection of vertices of $\Phi$ whose corresponding tori lies within $N$. Note that if $v \in V$, then every vertex $w$ with an edge from $v$ to $w$ lies within $V$ as well. But since $\Phi$ is strongly connected, either $V$ is empty or contains every vertex of $\Phi$. In the former case, $N$ must be contained within a single Seifert fibered piece. But this is impossible since $\phi^t$ restricted to each Seifert fibered piece has no sinks.
In the latter case, we obtain a similar contradiction.
\end{proof}

\subsection{Independence on gluing maps} \label{subsec:indgluingmap}

In this subsection we make the first steps towards showing \Cref{thm:totallyperiodicgenusone} by showing that up to performing horizontal surgery, we can discard the data of the gluing maps.

The key fact is the following result from \cite{Tsa24}.

\begin{thm}[{\cite[Corollary 1.6]{Tsa24}}] \label{thm:hsurscalloptorus}
Let $\phi^t$ be a transitive Anosov flow on a closed oriented $3$-manifold $M$. Suppose $T$ is a scalloped transverse torus. Furthermore, suppose that each closed leaf $\ell^s$ of the induced stable foliation on $T$ intersects each closed leaf $\ell^u$ of the induced unstable foliation on $T$ exactly once. In other words, with respect to any fixed bases $(\ell^s, m^s)$ and $(\ell^u, m^u)$, the identity map on $T$ is represented by some matrix 
$\begin{bmatrix}
a & b \\
-1 & d
\end{bmatrix}$.

Then for every matrix of the form
$\begin{bmatrix}
m & n \\
p & q
\end{bmatrix}$
where $p < 0$, one can construct a pseudo-Anosov flow by cutting along $T$ and regluing using a map with the given matrix representative. Furthermore, this flow is almost equivalent to the initial flow $\phi^t$.
\end{thm}

\begin{lemma} \label{lemma:indgluingmap}
Let $\phi^t_1$ and $\phi^t_2$ be transitive totally periodic Anosov maps with orientable stable and unstable foliations. If $\phi^t_1$ and $\phi^t_2$ have the same total flow graph then $\phi^t_1$ and $\phi^t_2$ are almost equivalent.
\end{lemma}
\begin{proof}
Recall that $\phi^t_1$ and $\phi^t_2$ having the same total flow graph implies that they have the same spine graphs. From \Cref{thm:bftotallyperiodic}, we can construct a totally periodic Anosov flow $\phi^t_3$ with the same spine graph but with gluing maps chosen so that for each pair $(T_+,T_-)$ to be glued, $[\ell^s]$ and $[\ell^u]$ intersect exactly once, i.e. $([\ell^s],[\ell^u])$ is a positive basis for $H_1(T_\pm,\mathbb{Z})$.

With respect to the bases $(\ell^s, \ell^u)$ and $(\ell^u, -\ell^s)$, the identity map on $T_\pm$ is represented by
$\begin{bmatrix}
0 & 1 \\
-1 & 0
\end{bmatrix}$.
Meanwhile, the gluing map that determines $\phi^t_i$ is represented by some matrix
$\begin{bmatrix}
m_i & n_i \\
p_i & q_i
\end{bmatrix}$
where $p_i < 0$. By \Cref{thm:hsurscalloptorus} and the uniqueness in \Cref{thm:bftotallyperiodic}, we have $\phi^t_1 \sim \phi^t_3 \sim \phi^t_2$. Here we use the fact that the total flow graph of $\phi^t_3$ equals to that of $\phi^t_1$ and $\phi^t_2$, so that transitivity of $\phi^t_1$ and $\phi^t_2$ implies transitivity of $\phi^t_3$ using \Cref{prop:totallyperiodictransitive}.
\end{proof}

For the rest of this paper, we will write $\phi^t_\Phi$ for the uniquely determined almost equivalence class of totally periodic Anosov flows that have total flow graph $\Phi$.

\section{Basic moves} \label{sec:basicmoves}

The strategy of the proof of \Cref{thm:totallyperiodicgenusone} is to modify the flow graph $\Phi$ by performing horizontal Goodman surgery on $\phi^t_\Phi$ thus staying in the same almost equivalence class. To this end, we need to devise some admissible moves on $\Phi$. In this section, we introduce three basic types of moves --- a cutting move, a gluing move, and an insertion move.

\subsection{Cutting move} \label{subsec:cuttingmove}

Let $\phi^t$ be a totally periodic Anosov flow. Let $\Phi_i$ be the flow graphs of $\phi^t$. In this subsection, we describe a way of identifying non-scalloped Reebless transverse tori from certain cycles in $\Phi_i$. We then explain how performing horizontal Goodman surgery along such a transverse tori corresponds to a `cutting move' on the level of the total flow graph $\Phi$. 

We first recall the definition of a Reebless non-scalloped transverse torus.

\begin{defn} \label{defn:interiortorus}
Let $T$ be a transverse torus. Let $T^{s/u}$ be the induced stable/unstable foliations on $T$. We say that $T$ is a \textbf{Reebless non-scalloped transverse torus} if
\begin{itemize}
    \item Each of $T^s$ and $T^u$ contains a number of closed leaves.
    \item Each closed leaf of $T^s$ and $T^u$ can be given an orientation, which we refer to as the \textbf{spiraling orientation}, so that
    \begin{itemize}
        \item each non-closed $T^{s/u}$-leaf spirals into a closed $T^{s/u}$-leaf in their spiraling orientations on each of its ends, and
        \item adjacent closed $T^{s/u}$-leaves have opposite spiraling orientations.
    \end{itemize}
    \item The closed $T^s$-leaves are parallel to the closed $T^u$-leaves.
\end{itemize}
See \Cref{fig:interiortorusfol} for an illustration of such a pair of foliations $(T^s, T^u)$.
\end{defn}

\begin{figure}
    \centering
    \resizebox{!}{3cm}{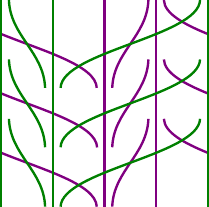}
    \caption{The form of the foliations $(T^s, T^u)$ on a Reebless non-scalloped transverse torus.}
    \label{fig:interiortorusfol}
\end{figure}

Let $c$ be an embedded and cooriented cycle in a flow graph $\Phi_i$ that consists of positive edges. We define a cooriented curve $\alpha$ lying on $S_i$ as follows:
\begin{itemize}
    \item For each edge $e$ in $c$, we take an interval $\alpha_e$ between two opposite sides of the complementary region of $\overline{F}_i$ containing $e$, so that the coorientation on $\alpha_e$ induced by the orientations on $F_i$ agrees with the coorientation on $e$. See \Cref{fig:totallyperiodiccutcurve} first row.
    \item For each vertex $v$ in $c$, say $e_1$ and $e_2$ are the two edges adjacent to $v$ in $c$, we take an interval $\alpha_v$ between the endpoints of $\alpha_{e_1}$ and $\alpha_{e_2}$ near $v$, such that the coorientation on $\alpha_v$ induced by the orientations on $F_i$ agrees with the coorientation on $e_1$ and $e_2$. See \Cref{fig:totallyperiodiccutcurve} second and third rows.
\end{itemize}
We consider each $\alpha_e$ and each $\alpha_v$ to be an interval lying in $S_i$, and take the union of $\alpha_e$ and $\alpha_v$ to get a cooriented curve $\alpha$ on $S_i$.

\begin{figure}
    \centering
    \fontsize{8pt}{8pt}\selectfont
    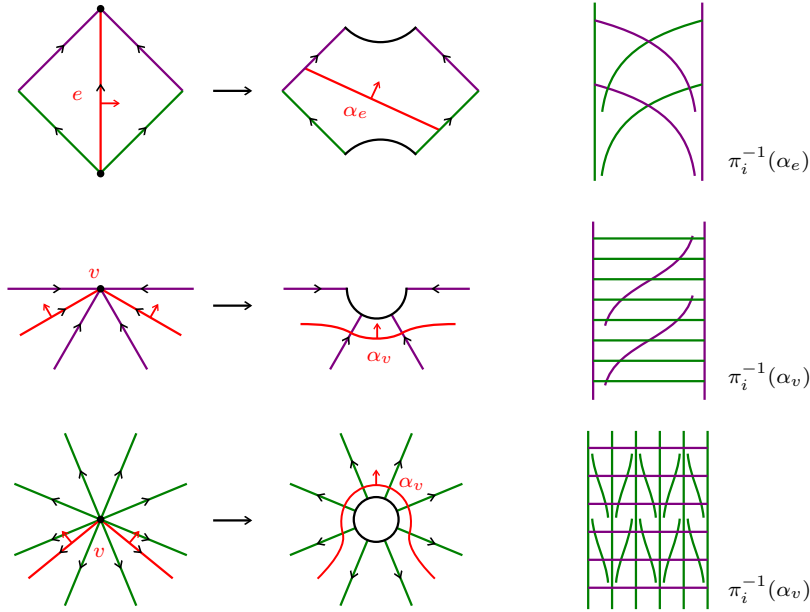
    \caption{Constructing a cooriented curve $\alpha$ on $S_i$ from a cycle of $\Phi_i$ that consists of positive curves and which is cooriented and embedded on $\overline{S}_i$.}
    \label{fig:totallyperiodiccutcurve}
\end{figure}

\begin{lemma} \label{lemma:totallyperiodiccutinteriortorus}
The torus $T=\pi_i^{-1}(\alpha)$ is a Reebless non-scalloped transverse torus that is not boundary parallel in $P_i$.
\end{lemma}
\begin{proof}
For each edge $e$ in $c$, from the fact that the complementary region of $F_i$ contains a Birkhoff annulus, we deduce that $\pi_i^{-1}(\alpha_e)$ consists of one closed $T^s$-leaf $\ell^s$ and one closed $T^u$-leaf $\ell^u$ lying on each of its boundary components. The noncompact leaves of $T^s$ spiral into $\ell^s$ in the direction opposite to the orientation of the vertical orbit lying on the same stable leaf as $\ell^s$. The noncompact leaves of $T^u$ spiral into $\ell^u$ in the direction of the orientation of the vertical orbit lying on the same unstable leaf as $\ell^u$. See \Cref{fig:totallyperiodiccutcurve} first row.

For each vertex $v$ in $c$ that is an inward joint, using again the fact that every complementary region of $F_i$ contains a Birkhoff annulus, we deduce that $\pi_i^{-1}(\alpha_v)$ consists an even number of closed $T^s$-leaves. The noncompact leaves of $T^s$ spiral into these closed leaves in the direction opposite to the orientation of the vertical orbit lying on the same stable leaves. The induced unstable foliation $T^u$ on $\pi_i^{-1}(\alpha_v)$ consists of non-separating intervals. The symmetric description holds for $\pi_i^{-1}(\alpha_v)$ for each vertex $v$ in $c$ that is an outward joint. See \Cref{fig:totallyperiodiccutcurve} second and third rows.

The torus $T=\pi_i^{-1}(\alpha)$ can be obtained by taking the union of the annuli $\pi_i^{-1}(\alpha_e)$ and $\pi_i^{-1}(\alpha_v)$. The properties in \Cref{defn:interiortorus} follows from the recorded properties of $T^{s/u}$ in each annulus.

Finally, if $T$ is boundary parallel in $P_i$, then $\alpha$ is a boundary parallel curve in $S_i$. Without loss of generality suppose that $\alpha$ is parallel to a boundary component lying on the side that $\alpha$ is cooriented into. 

For each outward joint that $c$ passes through, the corresponding boundary component of $S_i$ lies on the side that $\alpha$ is cooriented into, thus $c$ only passes through one outward joint. Since $c$ is embedded, this implies that $c$ can only pass through one inward joint. Hence $\alpha$ consists of two edges with endpoints on the same pair of joints. 

In this case, $\alpha$ can only be boundary parallel if these two edges are isotopic in $\overline{S_i}$, i.e. they bound a disc. Such a disc pulls back to a quadrilateral in $S_i$, which has index $0$. Meanwhile, as pointed out in the proof of \Cref{lemma:flowgrapheulerchar}, this quadrilateral must be tiled by octagons, each of which has index $-1$, giving us a contradiction.
\end{proof}

We will use the following result from \cite{Tsa24}.

\begin{thm}[{\cite[Corollary 1.7]{Tsa24}}] \label{thm:hsurinteriortorus}
Let $\phi^t$ be a transitive Anosov flow on a closed oriented $3$-manifold $M$. Suppose $T$ is a Reebless non-scalloped transverse torus. 
Then one can construct an Anosov flow $\overline{\phi}^t$ by cutting along $T$ and regluing it along some map, so that $T$ becomes a scalloped torus and so that each closed leaf $\ell^s$ of the induced stable foliation on $T$ intersects each closed leaf $\ell^u$ of the induced unstable foliation on $T$ exactly once. 
Furthermore, this flow is almost equivalent to $\phi^t$.
\end{thm}

We apply \Cref{thm:hsurinteriortorus} to $T=\pi_i^{-1}(\alpha)$. This results in an Anosov flow $\overline{\phi}^t$ defined on another graph manifold $\overline{M}$. The Seifert fibered pieces of $\overline{M}$ are the same as $M$ except with $P_i$ cut along $T$. Correspondingly the base surfaces of $\overline{M}$ are the same as $M$ except with $S_i$ cut along $\alpha$. 

We claim that $\overline{\phi}^t$ is totally periodic:
First observe that each component of $P_i \cut T$ contains at least one vertical orbit of $\phi^t$. Indeed, using an index argument as in the proof of \Cref{lemma:flowgrapheulerchar}, we see that the Euler characteristic of the base surface is equal to minus the number of vertical orbits, so this follows from the assumption that $\alpha$ is not boundary parallel. After surgery, the vertical orbits of $\phi^t$ remain homotopic to the Seifert fibers of $P_i \cut T$, hence remain vertical orbits of $\overline{\phi}^t$.

Next, we explain how the total flow graph of $\overline{\phi}^t$ relates to that of $\phi^t$. Let us label the edges of $c$ cyclically as $(e_i)$, $i \in \mathbb{Z}/2n$, where $e_{2i-1}$ and $e_{2i}$ share a vertex at an inward joint and $e_{2i}$ and $e_{2i+1}$ share a vertex at an outward joint. We cut $\overline{S}_i$ along $c$. This gives a surface with two boundary components. We label the one where $c$ is cooriented out of by $c^+$ and the one where $c$ is cooriented into by $c^-$, and we label the copies of the edges $e_i$ by $e^\pm_i$ correspondingly. We now identify $e^-_{2i-1}$ and $e^-_{2i}$, and $e^+_{2i}$ and $e^+_{2i+1}$ in orientation preserving ways for each $i$, and call the resulting graph $\overline{\Phi}_i$. $\overline{\Phi}_i$ has two more vertices than $\Phi_i$. We construct $\overline{\Phi}$ by identifying the two new vertices, then identifying the remaining vertices as for $\Phi_i$. We illustrate one example in \Cref{fig:totallyperiodiccuteg}.

\begin{figure}
    \centering
    \fontsize{10pt}{10pt}\selectfont
    %% Creator: Inkscape 1.3.2 (091e20e, 2023-11-25, custom), www.inkscape.org
%% Accompanies image file 'cuttingeg.pdf' (pdf, eps, ps)
\begingroup%
  \makeatletter%
  \providecommand\color[2][]{%
    \errmessage{(Inkscape) Color is used for the text in Inkscape, but the package 'color.sty' is not loaded}%
    \renewcommand\color[2][]{}%
  }%
  \providecommand\transparent[1]{%
    \errmessage{(Inkscape) Transparency is used (non-zero) for the text in Inkscape, but the package 'transparent.sty' is not loaded}%
    \renewcommand\transparent[1]{}%
  }%
  \providecommand\rotatebox[2]{#2}%
  \newcommand*\fsize{\dimexpr\f@size pt\relax}%
  \newcommand*\lineheight[1]{\fontsize{\fsize}{#1\fsize}\selectfont}%
  \ifx\svgwidth\undefined%
    \setlength{\unitlength}{289.14214127bp}%
    \ifx\svgscale\undefined%
      \relax%
    \else%
      \setlength{\unitlength}{\unitlength * \real{\svgscale}}%
    \fi%
  \else%
    \setlength{\unitlength}{\svgwidth}%
  \fi%
  \global\let\svgwidth\undefined%
  \global\let\svgscale\undefined%
  \makeatother%
  \begin{picture}(1,0.55509297)%
    \lineheight{1}%
    \setlength\tabcolsep{0pt}%
    \put(0,0){\includegraphics[width=\unitlength,page=1]{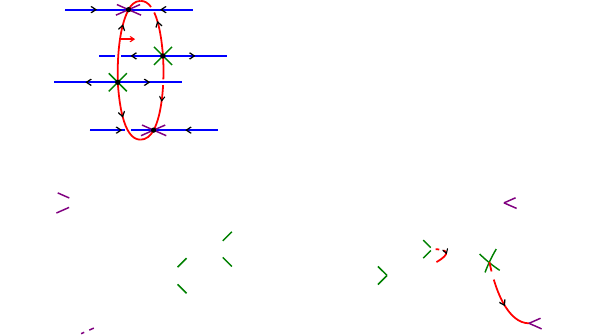}}%
    \put(0.17281253,0.48292118){\color[rgb]{1,0,0}\makebox(0,0)[lt]{\lineheight{1.25}\smash{\begin{tabular}[t]{l}$c$\end{tabular}}}}%
    \put(0,0){\includegraphics[width=\unitlength,page=2]{cuttingeg.pdf}}%
  \end{picture}%
\endgroup%

    \caption{An example of cutting along an embedded and cooriented cycle $c$.}
    \label{fig:totallyperiodiccuteg}
\end{figure}

\begin{lemma} \label{lemma:cuttingmovegluedflowgraph}
The total flow graph of $\overline{\phi}^t$ is given by $\overline{\Phi}$.
\end{lemma}
\begin{proof}
We first check that the components of $\overline{\Phi}_i$ are the flow graphs of $\overline{\phi}^t$ on the components of $P_i \cut T$. 
It is straightforward to check that for every edge $e$ in $\overline{\Phi}_i$, there is an orbit segment of $\overline{\phi}^t$ in an isotopy class that projects down to that of $e$. Indeed, one can simply take restrictions of orbit segments for $\Phi_i$. This implies that the $\overline{\Phi}_i$ is a subset of the correct flow graph. Meanwhile, the number of complementary regions of $\overline{\Phi}_i$ in $S_i \cut c$ equals the number of complementary regions of $\Phi_i$ in $S_i$, which equals $-\chi(S_i)=\chi(S_i \cut c)$ by \Cref{lemma:flowgrapheulerchar}. Hence there are no additional edges to $\overline{\Phi}_i$.

The two additional joints in $\overline{\Phi}_i$ correspond to the tori $c_\pm \times S^1$, hence are to be identified. The rest of the boundary components are not affected by the surgery operation, hence their corresponding joints are to be identified in the same way.
\end{proof}

We refer to the operation of obtaining $\overline{\Phi}$ from $\Phi$ as \textbf{cutting along $c$}. 
An important special case that will come up often later is if $n=1$, that is, when the length of $c$ is $2$. In this case, the operation amounts to cutting $\overline{S}_i$ along $c$ and identifying the two edges in each of the resulting two boundary components.

A symmetric version of this construction holds where $c$ is an embedded and cooriented cycle in $\Phi_i$ that consists of negative edges. From \Cref{thm:horsuralmostequiv} and \Cref{prop:totallyperiodictransitive}, we have the following proposition.

\begin{prop} \label{prop:cuttingmove}
Let $\Phi$ be the total flow graph of a totally periodic Anosov flow with orientable stable and unstable foliations. Let $c$ be an embedded and cooriented cycle of a flow graph $\Phi_i$ that consists of edges of the same sign. Let $\overline{\Phi}$ be obtained by cutting $\Phi$ along $c$. Suppose $\Phi$ is strongly connected. Then $\phi^t_\Phi \sim \phi^t_{\overline{\Phi}}$.
\end{prop}

Note that from \Cref{prop:totallyperiodictransitive}, this in particular means that $\Phi$ being strongly connected implies that $\overline{\Phi}$ is strongly connected. This is also not too difficult to check by hand.

\subsection{Gluing move} \label{subsec:gluingmove}

Next, we add a `gluing move' to our toolkit. This is simply the reverse of the cutting move in the case when $n=1$. A subtle point, however, concerns the strongly connectedness of the \emph{resulting} total flow graph.

Let $\phi^t$ be a totally periodic Anosov flow. Let $\Phi$ be the total flow graph of $\phi^t$. Let $e_1$ and $e_2$ be two distinct positive edges of $\Phi$ where
\begin{itemize}
    \item the forward endpoint $e^+_1$ of $e_1$ equals the backward endpoint $e^-_2$ of $e_2$,
    \item the forward endpoint $e^+_1$ of $e_1$ does not equal the forward endpoint $e^+_2$ of $e_2$, and
    \item the backward endpoint $e^-_1$ of $e_1$ does not equal the backward endpoint $e^-_2$ of $e_2$.
\end{itemize} 
That is, $(e_1,e_2)$ is an embedded edge path in $\Phi$. We also let $\Phi_i$ be the individual flow graph that contains $e_i$ as an edge, for $i=1,2$. Here we might have $\Phi_1=\Phi_2$.

For each $i$, we slit $S_i$ along $e_i$. This creates a boundary component consisting of the two copies of $e_i$ that get created. We label these as $e'_i$ and $e''_i$. We then glue $e'_1$ to $e'_2$ and $e''_1$ to $e''_2$ by orientation preserving maps. This glues up $S_1$ and $S_2$ into a single surface $\overline{S}_1$ with an embedded graph $\overline{\Phi}_1$. The two inward joints $e^-_1$ and $e^-_2$ get identified into a single inward joint $\overline{e}^-_1$ and the two outward joints $e^+_1$ and $e^+_2$ get identified into a single outward joint $\overline{e}^+_1$. We construct $\overline{\Phi}$ by identifying $\overline{e}^-_1$ with the outward joint that $e^-_1$ used to be identified with, and identifying $\overline{e}^+_1$ with the inward joint that $e^+_2$ used to be identified with. Applying \Cref{thm:bftotallyperiodic}, we let $\overline{\phi}^t$ be the totally periodic Anosov flow with total flow graph $\overline{\Phi}$. 

Notice that there are in fact two possibilities for $\overline{\Phi}_1$, hence for $\overline{\Phi}$ and $\overline{\phi}^t$. These differ by the choice of labelling $e'_i$ and $e''_i$ for each $i$. See \Cref{fig:gluingeg} for an example where we show both possibilities for $\overline{\Phi}_1$. We refer to the operation of obtaining either $\overline{\Phi}$ from $\Phi$ as \textbf{gluing $e_1$ to $e_2$}.

\begin{figure}
    \centering
    \fontsize{8pt}{8pt}\selectfont
    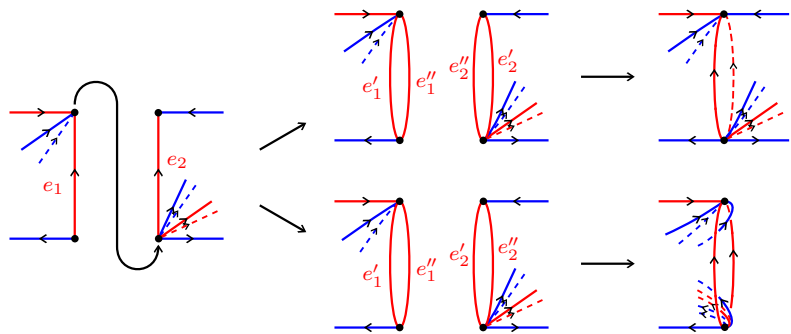
    \caption{Modifying a total flow graph by gluing two distinct positive edges $e_1$ and $e_2$.}
    \label{fig:gluingeg}
\end{figure}

In either case, the images of $e'_1$ and $e''_1$ in $\overline{\Phi}_1$ form an embedded cycle $c$ that consists of positive edges. We coorient $c$ to point towards $S_2$ and cut $\overline{\Phi}$ along $c$. It is straightforward to verify that the result recovers $\Phi$. See \Cref{fig:slidingmove} top.

\begin{figure}
    \centering
    \fontsize{8pt}{8pt}\selectfont
    %% Creator: Inkscape 1.3.2 (091e20e, 2023-11-25, custom), www.inkscape.org
%% Accompanies image file 'slidingmove.pdf' (pdf, eps, ps)
\begingroup%
  \makeatletter%
  \providecommand\color[2][]{%
    \errmessage{(Inkscape) Color is used for the text in Inkscape, but the package 'color.sty' is not loaded}%
    \renewcommand\color[2][]{}%
  }%
  \providecommand\transparent[1]{%
    \errmessage{(Inkscape) Transparency is used (non-zero) for the text in Inkscape, but the package 'transparent.sty' is not loaded}%
    \renewcommand\transparent[1]{}%
  }%
  \providecommand\rotatebox[2]{#2}%
  \newcommand*\fsize{\dimexpr\f@size pt\relax}%
  \newcommand*\lineheight[1]{\fontsize{\fsize}{#1\fsize}\selectfont}%
  \ifx\svgwidth\undefined%
    \setlength{\unitlength}{211.73636609bp}%
    \ifx\svgscale\undefined%
      \relax%
    \else%
      \setlength{\unitlength}{\unitlength * \real{\svgscale}}%
    \fi%
  \else%
    \setlength{\unitlength}{\svgwidth}%
  \fi%
  \global\let\svgwidth\undefined%
  \global\let\svgscale\undefined%
  \makeatother%
  \begin{picture}(1,0.77624965)%
    \lineheight{1}%
    \setlength\tabcolsep{0pt}%
    \put(0,0){\includegraphics[width=\unitlength,page=1]{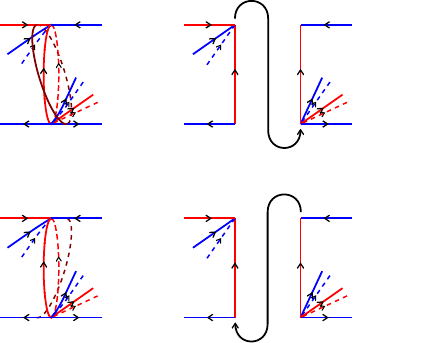}}%
    \put(0.05148448,0.59703201){\color[rgb]{0.50196078,0,0}\makebox(0,0)[lt]{\lineheight{1.25}\smash{\begin{tabular}[t]{l}$c$\end{tabular}}}}%
    \put(0.16752529,0.19412804){\color[rgb]{0.50196078,0,0}\makebox(0,0)[lt]{\lineheight{1.25}\smash{\begin{tabular}[t]{l}$c$\end{tabular}}}}%
    \put(0,0){\includegraphics[width=\unitlength,page=2]{slidingmove.pdf}}%
    \put(0.24465536,0.48650295){\color[rgb]{0,0,0}\makebox(0,0)[lt]{\lineheight{1.25}\smash{\begin{tabular}[t]{l}$\overline{\Phi}$\end{tabular}}}}%
    \put(0.80987582,0.48650295){\color[rgb]{0,0,0}\makebox(0,0)[lt]{\lineheight{1.25}\smash{\begin{tabular}[t]{l}$\Phi$\end{tabular}}}}%
    \put(0.80987766,0.0477791){\color[rgb]{0,0,0}\makebox(0,0)[lt]{\lineheight{1.25}\smash{\begin{tabular}[t]{l}$\widecheck{\Phi}$\end{tabular}}}}%
    \put(0.24465699,0.0477791){\color[rgb]{0,0,0}\makebox(0,0)[lt]{\lineheight{1.25}\smash{\begin{tabular}[t]{l}$\overline{\Phi}$\end{tabular}}}}%
    \put(0,0){\includegraphics[width=\unitlength,page=3]{slidingmove.pdf}}%
  \end{picture}%
\endgroup%

    \caption{Cutting along the cycle $c$ formed by the images of $e'_1$ and $e''_1$. Top: If we coorient $c$ towards $S_2$, then we recover $\Phi$. Bottom: If we coorient $c$ towards $S_1$, then we slide $S_1$ and $S_2$ across each other.}
    \label{fig:slidingmove}
\end{figure}

\Cref{prop:cuttingmove} thus applies in this context to give the following proposition.

\begin{prop} \label{prop:gluingmove}
Let $\Phi$ be the total flow graph of a totally periodic Anosov flow with orientable stable and unstable foliations. Let $e_1$ and $e_2$ be two distinct edges of $\Phi$ of the same sign so that $(e_1,e_2)$ is an embedded edge path of $\Phi$. Let $\overline{\Phi}$ be obtained by gluing $e_1$ to $e_2$. Suppose $\overline{\Phi}$ is strongly connected. Then $\phi^t_\Phi \sim \phi^t_{\overline{\Phi}}$.
\end{prop}

A subtle point here, however, is that the assumption of $\overline{\Phi}$ being strongly connected is \emph{not} implied by $\Phi$ being strongly connected. This means that starting from a strongly connected total flow graph, one cannot modify it by applying gluing moves as one wishes (and remain in the same almost equivalence class of flows); instead one has to first verify that the resulting graph is strongly connected.

We record one situation where we can know in advance that $\overline{\Phi}$ is strongly connected.

\begin{prop} \label{prop:cyclicgluingmove}
Let $\Phi$ be the total flow graph of a totally periodic Anosov flow with orientable stable and unstable foliations. Let $e_1$ and $e_2$ be two distinct edges of $\Phi$ of the same sign so that the forward endpoint $e_1$ equals the backward endpoint $e_2$. Suppose $\Phi$ is strongly connected and suppose $(e_1,e_2)$ is a segment of an embedded directed cycle $c$ of $\Phi$. Let $\overline{\Phi}$ be obtained by gluing $e_1$ to $e_2$. Then $\phi^t_\Phi \sim \phi^t_{\overline{\Phi}}$.
\end{prop}
\begin{proof}
It suffices to show that the existence of $c$ implies that $\overline{\Phi}$ is strongly connected. Let $\overline{v}_1$ and $\overline{v}_2$ be two vertices of $\overline{\Phi}$. We have to show that there is a directed path from $\overline{v}_1$ to $\overline{v}_2$. For each $i$, we associate to $\overline{v}_i$ a vertex $v_i$ of $\Phi$:
\begin{itemize}
    \item If $\overline{v}_i \neq \overline{e}^{\pm}_1$, then $v_i$ is just the vertex in $\Phi$ that corresponds to $\overline{v}_i$.
    \item If $\overline{v}_i = \overline{e}^{\pm}_1$, then $v_i = e^-_2 = e^+_1$.
\end{itemize}

Since $\overline{\Phi}$ is strongly connected, there exists a directed path $\gamma$ from $v_1$ to $v_2$. Notice that if $\gamma$ does not pass through $e^-_2 = e^+_1$, then the image of $\gamma$ in $\overline{\Phi}$ is a directed path from $\overline{v}_1$ to $\overline{v}_2$. If $\gamma$ does pass through $e^-_2 = e^+_1$, but does so via entering by $e_1$ or exiting $e_2$, the image of $\gamma$ minus $e_1$ or $e_2$, respectively, is a directed path from $\overline{v}_1$ to $\overline{v}_2$ as well.

The nontrivial case is when $\gamma$ does pass through $e^-_2 = e^+_1$, but does so via entering by an edge other than $e_1$ and exiting via an edge other than $e_2$. But in this case, we can modify $\gamma$ by concatenating it with the cycle $c$ at $e^-_2 = e^+_1$. It is now true that $\gamma$ pass through $e^-_2 = e^+_1$ via entering by $e_1$ or exiting $e_2$ (at both such instances), so the argument above applies.
\end{proof}

Recall that to recover $\Phi$ from $\overline{\Phi}$, we cut along the cycle $c$ formed by the images of $e'_1$ and $e''_1$ and cooriented towards $S_2$. Suppose that we coorient $c$ towards $S_1$ instead, and let $\widecheck{\Phi}$ be obtained by cutting $\overline{\Phi}$ along $c$ with this alternate coorientation. See \Cref{fig:slidingmove} bottom. Then the collection of individual flow graphs in $\widecheck{\Phi}$ would agree with that of $\Phi$, but they would be connected up in a different way. Namely, 
\begin{itemize}
    \item $e^-_1$ would be identified with $e^+_2$,
    \item $e^+_1$ would be identified with the inward joint originally identified with $e^+_2$, and
    \item $e^-_2$ would be identified with the outward joint originally identified with $e^-_1$.
\end{itemize}
Intuitively, the base surfaces $S_1$ and $S_2$ are `slid across one another' over $e_1$ and $e_2$.

By repeating this procedure, one can glue two non-adjacent edges by first sliding one edge towards another until they are adjacent. We explain the precise setting of this in the following proposition.

\begin{prop} \label{prop:transportmove}
Let $\Phi$ be the total flow graph of a totally periodic Anosov flow with orientable stable and unstable foliations. Let $e_1$ and $e_2$ be two distinct edges of $\Phi$, belonging to flow graphs $\Phi_1$ and $\Phi_2$ respectively. Suppose $\Phi$ is strongly connected and suppose $e_1$ and $e_2$ are part of an embedded directed cycle $c$ of $\Phi$ that consists of edges of the same sign. Then by repeated cutting and gluing, we can obtain $\overline{\Phi}$ so that the collection of flow graphs in $\overline{\Phi}$ equals that of $\Phi$, but with $e^+_1=e^-_2$. Moreover, $\phi^t_\Phi \sim \phi^t_{\overline{\Phi}}$ and $(e_1,e_2)$ is a segment of an embedded directed cycle of $\overline{\Phi}$.
\end{prop}
\begin{proof}
We denote the edges on $c$ between $e_1$ and $e_2$ as $f_1,...,f_k$. We first glue $e_1$ to $f_1$. The existence of $c$ implies that we stay in the same almost equivalence class, by \Cref{prop:cyclicgluingmove}. We then cut along the cycle formed by the images of $e'_1$ and $e''_1$ and cooriented towards $S_1$. By \Cref{prop:cuttingmove}, we again stay in the same almost equivalence class. 

This swaps the positions of $e_1$ and $f_1$ on $c$. In particular, $e_1$ and $e_2$ are still part of an embedded directed cycle that consists of positive edges, but now there is one less edge between them on the cycle. We can thus repeat this procedure until $e_1$ is adjacent to $e_2$. See \Cref{fig:transportmove} for a pictorial summary of the argument.
\end{proof}

\begin{figure}
    \centering
    \fontsize{8pt}{8pt}\selectfont
    %% Creator: Inkscape 1.3.2 (091e20e, 2023-11-25, custom), www.inkscape.org
%% Accompanies image file 'transportmove.pdf' (pdf, eps, ps)
\begingroup%
  \makeatletter%
  \providecommand\color[2][]{%
    \errmessage{(Inkscape) Color is used for the text in Inkscape, but the package 'color.sty' is not loaded}%
    \renewcommand\color[2][]{}%
  }%
  \providecommand\transparent[1]{%
    \errmessage{(Inkscape) Transparency is used (non-zero) for the text in Inkscape, but the package 'transparent.sty' is not loaded}%
    \renewcommand\transparent[1]{}%
  }%
  \providecommand\rotatebox[2]{#2}%
  \newcommand*\fsize{\dimexpr\f@size pt\relax}%
  \newcommand*\lineheight[1]{\fontsize{\fsize}{#1\fsize}\selectfont}%
  \ifx\svgwidth\undefined%
    \setlength{\unitlength}{371.49416465bp}%
    \ifx\svgscale\undefined%
      \relax%
    \else%
      \setlength{\unitlength}{\unitlength * \real{\svgscale}}%
    \fi%
  \else%
    \setlength{\unitlength}{\svgwidth}%
  \fi%
  \global\let\svgwidth\undefined%
  \global\let\svgscale\undefined%
  \makeatother%
  \begin{picture}(1,0.38151954)%
    \lineheight{1}%
    \setlength\tabcolsep{0pt}%
    \put(0,0){\includegraphics[width=\unitlength,page=1]{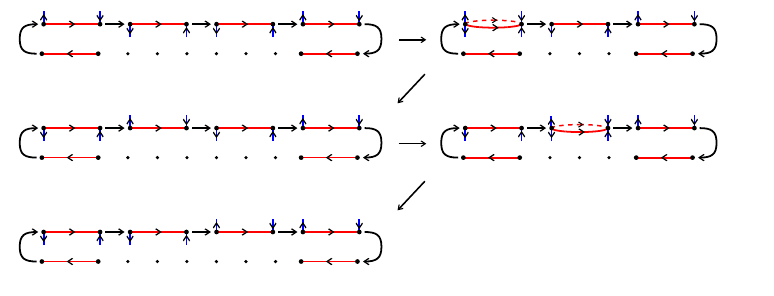}}%
    \put(0.08262962,0.36343148){\color[rgb]{1,0,0}\makebox(0,0)[lt]{\lineheight{1.25}\smash{\begin{tabular}[t]{l}$e_1$\end{tabular}}}}%
    \put(0.19428413,0.22917286){\color[rgb]{1,0,0}\makebox(0,0)[lt]{\lineheight{1.25}\smash{\begin{tabular}[t]{l}$e_1$\end{tabular}}}}%
    \put(0.41757905,0.36343148){\color[rgb]{1,0,0}\makebox(0,0)[lt]{\lineheight{1.25}\smash{\begin{tabular}[t]{l}$e_2$\end{tabular}}}}%
    \put(0.8505589,0.36342831){\color[rgb]{1,0,0}\makebox(0,0)[lt]{\lineheight{1.25}\smash{\begin{tabular}[t]{l}$e_2$\end{tabular}}}}%
    \put(0.30593289,0.09491423){\color[rgb]{1,0,0}\makebox(0,0)[lt]{\lineheight{1.25}\smash{\begin{tabular}[t]{l}$e_1$\end{tabular}}}}%
    \put(0.85055972,0.22916965){\color[rgb]{1,0,0}\makebox(0,0)[lt]{\lineheight{1.25}\smash{\begin{tabular}[t]{l}$e_2$\end{tabular}}}}%
    \put(0.41757992,0.22917283){\color[rgb]{1,0,0}\makebox(0,0)[lt]{\lineheight{1.25}\smash{\begin{tabular}[t]{l}$e_2$\end{tabular}}}}%
    \put(0.41758079,0.09491423){\color[rgb]{1,0,0}\makebox(0,0)[lt]{\lineheight{1.25}\smash{\begin{tabular}[t]{l}$e_2$\end{tabular}}}}%
  \end{picture}%
\endgroup%

    \caption{By repeated cuting and gluing, we can slide $e_1$ towards $e_2$.}
    \label{fig:transportmove}
\end{figure}
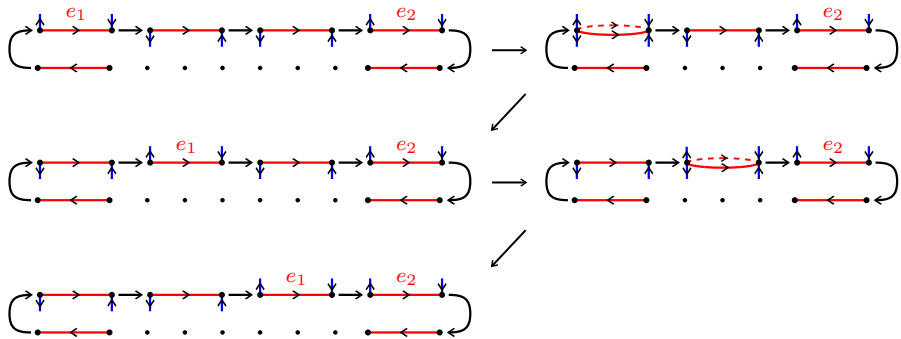

\subsection{Insertion move} \label{subsec:insertionmove}

In this subsection, we introduce the final type of basic move in our toolkit. The idea is that we can insert additional Seifert fibered pieces by performing horizontal Goodman surgery along a closed braid lying close to a boundary component of a Seifert fibered piece.

Let $\phi^t$ be a totally periodic Anosov flow with orientable stable and unstable foliations on a graph manifold $M$.
Let $T$ be a boundary component of a Seifert fibered piece.
By \Cref{lemma:totallyperiodicboundarytorusscallop}, $T$ is a scalloped transverse torus. By \cite[Example 3.13]{Tsa24}, there is a positive horizontal surgery curve $b$ lying on $T$. As explained in \cite[Example 3.15]{Tsa24}, we can construct another positive horizontal curve $\beta$ by inserting the $2$-strand closed braid with a single negative crossing along $b$. We illustrate a picture of $\beta$ in a neighborhood $N(T)$ of $T$ in \Cref{fig:insertionbraid}.

\begin{figure}
    \centering
    \fontsize{10pt}{10pt}\selectfont
    %% Creator: Inkscape 1.3.2 (091e20e, 2023-11-25, custom), www.inkscape.org
%% Accompanies image file 'insertionbraid.pdf' (pdf, eps, ps)
\begingroup%
  \makeatletter%
  \providecommand\color[2][]{%
    \errmessage{(Inkscape) Color is used for the text in Inkscape, but the package 'color.sty' is not loaded}%
    \renewcommand\color[2][]{}%
  }%
  \providecommand\transparent[1]{%
    \errmessage{(Inkscape) Transparency is used (non-zero) for the text in Inkscape, but the package 'transparent.sty' is not loaded}%
    \renewcommand\transparent[1]{}%
  }%
  \providecommand\rotatebox[2]{#2}%
  \newcommand*\fsize{\dimexpr\f@size pt\relax}%
  \newcommand*\lineheight[1]{\fontsize{\fsize}{#1\fsize}\selectfont}%
  \ifx\svgwidth\undefined%
    \setlength{\unitlength}{156.85322318bp}%
    \ifx\svgscale\undefined%
      \relax%
    \else%
      \setlength{\unitlength}{\unitlength * \real{\svgscale}}%
    \fi%
  \else%
    \setlength{\unitlength}{\svgwidth}%
  \fi%
  \global\let\svgwidth\undefined%
  \global\let\svgscale\undefined%
  \makeatother%
  \begin{picture}(1,0.49416041)%
    \lineheight{1}%
    \setlength\tabcolsep{0pt}%
    \put(0,0){\includegraphics[width=\unitlength,page=1]{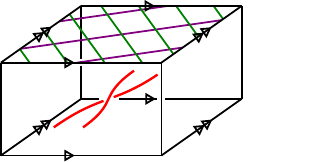}}%
    \put(0.33696973,0.09664374){\color[rgb]{1,0,0}\makebox(0,0)[lt]{\lineheight{1.25}\smash{\begin{tabular}[t]{l}$\beta$\end{tabular}}}}%
    \put(0.77848007,0.27332436){\color[rgb]{0,0,0}\makebox(0,0)[lt]{\lineheight{1.25}\smash{\begin{tabular}[t]{l}$N(T)$\end{tabular}}}}%
  \end{picture}%
\endgroup%

    \caption{The horizontal surgery curve that we use for the insertion move.}
    \label{fig:insertionbraid}
\end{figure}

The claim is that the surgered flow $\phi^t_1(\beta)$ is a totally periodic flow whose collection of flow graphs differ from that of $\phi^t$ by adding the graph labelled (1QP) in \Cref{fig:simplepieces} top. 
A version of this claim was established in unpublished work of Sergio Fenley, Jessica Purcell, and Mario Shannon. We thank them for showing us their work.
We will show the claim by first proving the topological \Cref{lemma:insertiontopology} below.

\begin{lemma} \label{lemma:insertiontopology}
The $3$-manifold $N(T)_1(\beta)$ obtained by performing Dehn surgery on $N(T)$ along $\beta$ with coefficient $1$ is homeomorphic to $(\mathbb{R}P^2 \backslash \text{two discs}) \times S^1$. Moreover, under this homeomorphism, the orbit segment from $\beta$ to itself within $N(T)$ is mapped to a fiber $\{*\} \times S^1$. 
\end{lemma}
\begin{proof}
Let $N_1=N(T)$. Let $\nu_1$ be a tubular neighborhood of $\beta$. Consider the two surfaces $A_1$ and $B_1$ in $N_1 \backslash \nu_1$ illustrated in \Cref{fig:insertiontopology} top left. The surface $A_1$ is a Mobius band bounded by $\beta$ with core $b$. Alternatively, one can consider projecting $\beta$ to $T$ to obtain a curve that self-intersects once. The surface $A_1$ can be constructed by taking the complementary region of this curve that is homeomorphic to a disc and adding a band across the crossing of $\beta$. Under this perspective, the surface $B_1$ can be obtained by taking the other complementary region and adding a band across the crossing of $\beta$. Notice that, up to isotopy, $A_1$ and $B_1$ intersect in the orbit segment from $\beta$ to itself within $N(T)$.

Meanwhile, let $N_2 = (\mathbb{R}P^2 \backslash \text{two discs}) \times S^1$. Let $a$ be a nonseparating curve on $\mathbb{R}P^2 \backslash \text{two discs}$. Let $\alpha$ be the curve $a \times \{*\}$ in $N_2$. Let $\nu_2$ be a tubular neighborhood of $\alpha$. Consider the two surfaces $A_2$ and $B_2$ in $N_2 \backslash \nu_2$ illustrated in \Cref{fig:insertiontopology} top right. The surface $A_2$ is $(a \times S^1) \backslash \nu_2$. The surface $B_2$ can be obtained by taking a nonseparating curve $a'$ that intersects $a$ once and forming $(a' \times S^1) \backslash \nu_2$. Notice that $A_2$ and $B_2$ intersect in an arc that lies on a fiber $\{*\} \times S^1$. 

\begin{figure}
    \centering
    \fontsize{6pt}{6pt}\selectfont
    %% Creator: Inkscape 1.3 (0e150ed6c4, 2023-07-21), www.inkscape.org
%% Accompanies image file 'insertiontopology.pdf' (pdf, eps, ps)
\begingroup%
  \makeatletter%
  \providecommand\color[2][]{%
    \errmessage{(Inkscape) Color is used for the text in Inkscape, but the package 'color.sty' is not loaded}%
    \renewcommand\color[2][]{}%
  }%
  \providecommand\transparent[1]{%
    \errmessage{(Inkscape) Transparency is used (non-zero) for the text in Inkscape, but the package 'transparent.sty' is not loaded}%
    \renewcommand\transparent[1]{}%
  }%
  \providecommand\rotatebox[2]{#2}%
  \newcommand*\fsize{\dimexpr\f@size pt\relax}%
  \newcommand*\lineheight[1]{\fontsize{\fsize}{#1\fsize}\selectfont}%
  \ifx\svgwidth\undefined%
    \setlength{\unitlength}{231.9994156bp}%
    \ifx\svgscale\undefined%
      \relax%
    \else%
      \setlength{\unitlength}{\unitlength * \real{\svgscale}}%
    \fi%
  \else%
    \setlength{\unitlength}{\svgwidth}%
  \fi%
  \global\let\svgwidth\undefined%
  \global\let\svgscale\undefined%
  \makeatother%
  \begin{picture}(1,0.95791326)%
    \lineheight{1}%
    \setlength\tabcolsep{0pt}%
    \put(0,0){\includegraphics[width=\unitlength,page=1]{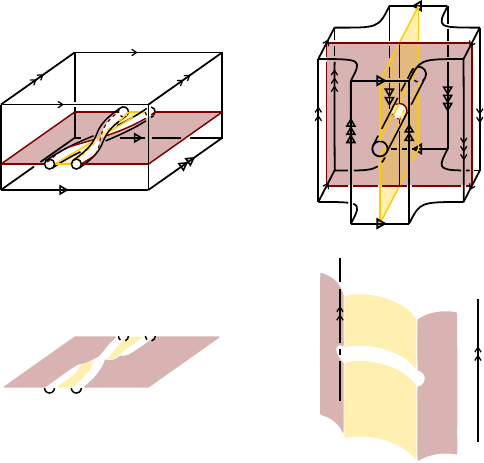}}%
    \put(0.18356141,0.75241372){\color[rgb]{0.50196078,0,0}\makebox(0,0)[lt]{\lineheight{1.25}\smash{\begin{tabular}[t]{l}$B_1$\end{tabular}}}}%
    \put(0.87675985,0.87989464){\color[rgb]{0.50196078,0,0}\makebox(0,0)[lt]{\lineheight{1.25}\smash{\begin{tabular}[t]{l}$B_2$\end{tabular}}}}%
    \put(0.51645276,0.71466076){\color[rgb]{0,0,0}\makebox(0,0)[lt]{\lineheight{1.25}\smash{\begin{tabular}[t]{l}\huge $\cong$\end{tabular}}}}%
    \put(0.51644922,0.21046513){\color[rgb]{0,0,0}\makebox(0,0)[lt]{\lineheight{1.25}\smash{\begin{tabular}[t]{l}\huge $\cong$\end{tabular}}}}%
    \put(0.27359106,0.75240906){\color[rgb]{1,0.8,0}\makebox(0,0)[lt]{\lineheight{1.25}\smash{\begin{tabular}[t]{l}$A_1$\end{tabular}}}}%
    \put(0.80448217,0.90868423){\color[rgb]{1,0.8,0}\makebox(0,0)[lt]{\lineheight{1.25}\smash{\begin{tabular}[t]{l}$A_2$\end{tabular}}}}%
    \put(0,0){\includegraphics[width=\unitlength,page=2]{insertiontopology.pdf}}%
  \end{picture}%
\endgroup%

    \caption{Seeing that $N_1 \backslash \nu_1 \cong N_2 \backslash \nu_2$ by considering the surfaces $A_1,B_1,A_2,B_2$.}
    \label{fig:insertiontopology}
\end{figure}

We claim that $N_1 \backslash \nu_1$ is homeomorphic to $N_2 \backslash \nu_2$ by a map that sends $A_1$ to $B_1$ and $A_2$ to $B_2$. One way to see this is to analyze the topology of $(N_i \backslash \nu_i) \cut (A_i \cup B_i)$. One checks that both $(N_1 \backslash \nu_1) \cut (A_1 \cup B_1)$ and $(N_2 \backslash \nu_2) \cut (A_2 \cup B_2)$ consist of two complementary regions. Each component is homeomorphic to $T^2 \times I$ with one boundary component on $\partial N_i$ and the other component divided into an annulus face that lies on $B_i$, a disc face that lies on $A_i$, and a disc face that lies on $\partial \nu_i$. See \Cref{fig:insertiontopology} bottom left and bottom right. To recover $N_i \backslash \nu_i$ one has to glue two of such components together so that the faces along $A_i$ get identified and the faces along $B_i$ get identified. Such a gluing is unique up to Dehn twists along the $B_i$ faces, and each Dehn twist can be extended into a homeomorphism of $(N_i \backslash \nu_i) \cut (A_i \cup B_i)$, showing the claim.

Recall that the longitude we take on $\partial \nu_1$ is that induced by the flow. In particular, the boundary slope of $B_1$ is $1$. Hence the manifold $(N_1)_1(\beta)$ is homeomorphic to the $3$-manifold obtained by filling in $N_2 \backslash \nu_2$ along the boundary slope of $B_2$, which is just $N_2$. The orbit segment from $\beta$ to itself within $N(T)$ is mapped to a fiber $\{*\} \times S^1$ under our homeomorphism.
\end{proof}

The first statement of \Cref{lemma:insertiontopology} implies that the manifold $M_1(\beta)$ is a graph manifold whose collection of Seifert fibered pieces is that of $M$ with $(\mathbb{R}P^2 \backslash \text{two discs}) \times S^1$ added. The surgery operation does not affect the vertical orbits of $\phi^t$ in each Seifert fibered piece of $M$, so it remains true for those Seifert fibered pieces that their fibers are homotopic to a closed orbit of $\phi^t_1(\beta)$. 

In performing the surgery, we have to cut along a surgery annulus $A \cong S^1 \times I$ and reglue by a map of the form $\sigma(h,k)=(h+\rho(k),k)$ where $\rho$ is a non-increasing function with $\rho=0$ near $0$ and $\rho=-1$ near $1$. 

We claim that there is an orbit segment from a point $x \in A$ to $\sigma^{-1}(x) \in A$. To see this, consider the segment $J_k \times \{k\}$ of $S^1 \times \{k\}$ consisting of points $x$ for which there is an orbit segment from $x$ to $A$ within $N(T)$. The values of $k$ such that there is an orbit segment from $J_k$ to $\sigma^{-1}(J_k)$ within $N(T)$ contains some interval $[k_1,k_2]$. Here the orbit segment for $k=k_1$ ends at the rightmost point of $\sigma^{-1}(J_{k_1})$, while that for $k=k_2$ ends at the leftmost point of $\sigma^{-1}(J_{k_2})$, hence applying the intermediate value theorem, there is some $k$ so that the orbit segment goes from a point $x \in J_k$ to $\sigma^{-1}(x) \in A$. See \Cref{fig:insertioncrossing} for an illustration of the argument.

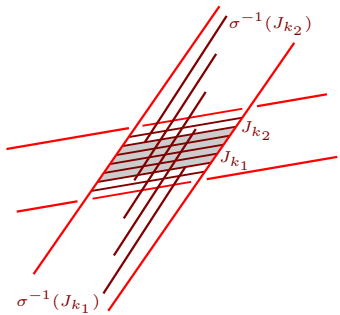
\begin{figure}
    \centering
    \fontsize{6pt}{6pt}\selectfont
    %% Creator: Inkscape 1.3.2 (091e20e, 2023-11-25, custom), www.inkscape.org
%% Accompanies image file 'insertioncrossing.pdf' (pdf, eps, ps)
\begingroup%
  \makeatletter%
  \providecommand\color[2][]{%
    \errmessage{(Inkscape) Color is used for the text in Inkscape, but the package 'color.sty' is not loaded}%
    \renewcommand\color[2][]{}%
  }%
  \providecommand\transparent[1]{%
    \errmessage{(Inkscape) Transparency is used (non-zero) for the text in Inkscape, but the package 'transparent.sty' is not loaded}%
    \renewcommand\transparent[1]{}%
  }%
  \providecommand\rotatebox[2]{#2}%
  \newcommand*\fsize{\dimexpr\f@size pt\relax}%
  \newcommand*\lineheight[1]{\fontsize{\fsize}{#1\fsize}\selectfont}%
  \ifx\svgwidth\undefined%
    \setlength{\unitlength}{135.68922989bp}%
    \ifx\svgscale\undefined%
      \relax%
    \else%
      \setlength{\unitlength}{\unitlength * \real{\svgscale}}%
    \fi%
  \else%
    \setlength{\unitlength}{\svgwidth}%
  \fi%
  \global\let\svgwidth\undefined%
  \global\let\svgscale\undefined%
  \makeatother%
  \begin{picture}(1,0.85002351)%
    \lineheight{1}%
    \setlength\tabcolsep{0pt}%
    \put(0,0){\includegraphics[width=\unitlength,page=1]{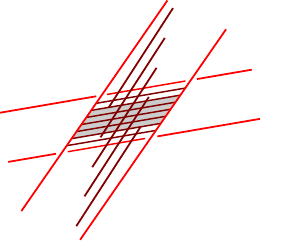}}%
    \put(0.02833103,0.01625032){\color[rgb]{0.50196078,0,0}\makebox(0,0)[lt]{\lineheight{1.25}\smash{\begin{tabular}[t]{l}$\sigma^{-1}(J_{k_1})$\end{tabular}}}}%
    \put(0.61803834,0.78853777){\color[rgb]{0.50196078,0,0}\makebox(0,0)[lt]{\lineheight{1.25}\smash{\begin{tabular}[t]{l}$\sigma^{-1}(J_{k_2})$\end{tabular}}}}%
    \put(0.59111193,0.41452102){\color[rgb]{0.50196078,0,0}\makebox(0,0)[lt]{\lineheight{1.25}\smash{\begin{tabular}[t]{l}$J_{k_1}$\end{tabular}}}}%
    \put(0.65056518,0.49845829){\color[rgb]{0.50196078,0,0}\makebox(0,0)[lt]{\lineheight{1.25}\smash{\begin{tabular}[t]{l}$J_{k_2}$\end{tabular}}}}%
  \end{picture}%
\endgroup%

    \caption{Arguing that there is an orbit segment from a point $x \in A$ to $\sigma^{-1}(x) \in A$.}
    \label{fig:insertioncrossing}
\end{figure}

This orbit segment closes up to a closed orbit of $\phi^t_1(\beta)$. By the second statement of \Cref{lemma:insertiontopology}, this is homotopic to the fiber of $(\mathbb{R}P^2 \backslash \text{two discs}) \times S^1$. Hence $\phi^t_1(\beta)$ is totally periodic.

It remains to compute the flow graph $\Phi_i$ corresponding to the base surface $S_i = \mathbb{R}P^2 \backslash \text{two discs}$. Note that $\chi(S_i)=-1$ so by \Cref{lemma:flowgrapheulerchar}, $\Phi_i$ can only have one quadrilateral in its complement. This leaves only one possibility, which is that illustrated as (1QP) --- $1$ Quad Projective plane in \Cref{fig:simplepieces}.

\begin{figure}
    \centering
    %% Creator: Inkscape 1.3 (0e150ed6c4, 2023-07-21), www.inkscape.org
%% Accompanies image file 'simplepieces.pdf' (pdf, eps, ps)
\begingroup%
  \makeatletter%
  \providecommand\color[2][]{%
    \errmessage{(Inkscape) Color is used for the text in Inkscape, but the package 'color.sty' is not loaded}%
    \renewcommand\color[2][]{}%
  }%
  \providecommand\transparent[1]{%
    \errmessage{(Inkscape) Transparency is used (non-zero) for the text in Inkscape, but the package 'transparent.sty' is not loaded}%
    \renewcommand\transparent[1]{}%
  }%
  \providecommand\rotatebox[2]{#2}%
  \newcommand*\fsize{\dimexpr\f@size pt\relax}%
  \newcommand*\lineheight[1]{\fontsize{\fsize}{#1\fsize}\selectfont}%
  \ifx\svgwidth\undefined%
    \setlength{\unitlength}{137.92812678bp}%
    \ifx\svgscale\undefined%
      \relax%
    \else%
      \setlength{\unitlength}{\unitlength * \real{\svgscale}}%
    \fi%
  \else%
    \setlength{\unitlength}{\svgwidth}%
  \fi%
  \global\let\svgwidth\undefined%
  \global\let\svgscale\undefined%
  \makeatother%
  \begin{picture}(1,0.79873492)%
    \lineheight{1}%
    \setlength\tabcolsep{0pt}%
    \put(0,0){\includegraphics[width=\unitlength,page=1]{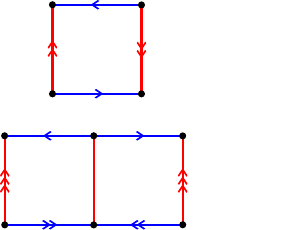}}%
    \put(0.6019287,0.61062393){\color[rgb]{0,0,0}\makebox(0,0)[lt]{\lineheight{1.25}\smash{\begin{tabular}[t]{l}(1QP)\end{tabular}}}}%
    \put(0.74540946,0.15460284){\color[rgb]{0,0,0}\makebox(0,0)[lt]{\lineheight{1.25}\smash{\begin{tabular}[t]{l}(2QS)\end{tabular}}}}%
  \end{picture}%
\endgroup%

    \caption{Two simple examples of flow graphs. Here the arrows on the sides of the quadrilaterals indicate identifications. We label the top one as (1QP) for 1 Quad Projective plane, and the bottom one as (2QS) for 2 Quad Sphere.}
    \label{fig:simplepieces}
\end{figure}

The torus $T$ corresponds to a vertex $v$ of the total flow graph $\Phi$. The effect of performing this horizontal Goodman surgery along $\beta$ is that we separate $v$ into an inward and outward joint $v_-$ and $v_+$ respectively, then identify $v_-$ with the outward joint and $v_+$ with the inward joint of a (1QP). We refer to this operation as \textbf{inserting a (1QP) at $v$}.

\Cref{thm:horsuralmostequiv} implies the following proposition.

\begin{prop} \label{prop:insertionmove}
Let $\Phi$ be the total flow graph of a totally periodic Anosov flow with orientable stable and unstable foliations. Let $v$ be a vertex of $\Phi$. Let $\overline{\Phi}$ be obtained by inserting a (1QP) at $v$. Suppose $\Phi$ is strongly connected. Then $\phi^t_\Phi \sim \phi^t_{\overline{\Phi}}$.
\end{prop}

Like \Cref{prop:cuttingmove}, \Cref{prop:insertionmove} contains the fact that if $\Phi$ is strongly connected then $\overline{\Phi}$ is strongly connected. In fact, it is easy to check directly that in this context, $\Phi$ is strongly connected if and only if $\overline{\Phi}$ is strongly connected.

\section{Proof of \Cref{thm:totallyperiodicgenusone}} \label{sec:totallyperiodicgenusoneproof}

\subsection{Bridge move} \label{subsec:bridgemove}

Aside from (1QP), another flow graph which will play a big role in the proof of \Cref{thm:totallyperiodicgenusone} is the flow graph which we name (2QS) --- 2 Quad Sphere in \Cref{fig:simplepieces} bottom.

In this subsection, we explain how to perform the bridge move. This move inserts a (2QS) between any two specified vertices of the total flow graph, akin to building a bridge. This move will be used to arrange for the conditions in \Cref{prop:cyclicgluingmove} and \Cref{prop:transportmove} concerning existence of certain directed cycles.

Let $\phi^t$ be a totally periodic Anosov flow with orientable stable and unstable foliations. Let $\Phi$ be the total flow graph of $\phi^t$. Suppose $\Phi$ is strongly connected. Let $v$ and $w$ be two distinct vertices of $\Phi$. We write $v^-$ for the inward joint that gets identified to $v$, and $v^+$ for the outward joint that gets identified to $v$, and similarly for $w^\pm$. 

Meanwhile, let $s_1$ and $s_2$ be the two positive edges of a (2QS). The goal is to transform $\Phi$ via the moves of \Cref{sec:basicmoves} so that the (2QS) is added to the collection of flow graphs and so that
\begin{itemize}
    \item $s^-_1$ is identified with $v^+$,
    \item $s^+_1$ is identified with $v^-$,
    \item $s^-_2$ is identified with $w^+$,
    \item $s^+_2$ is identified with $w^-$, and
    \item the rest of the joints are identified as in $\Phi$.
\end{itemize}
Let us denote this target total flow graph as $\overline{\Phi}$. Observe that $\overline{\Phi}$ is automatically strongly connected.

Since $\Phi$ is strongly connected, there exists an embedded directed path from $v$ to $w$. Suppose the vertices along the path are $v=v_1,v_2,...,v_n=w$. Let $e_i$ be the edge between $v_i$ and $v_{i+1}$. Suppose $e_i$ lies in the flow graph $\Phi_i$. 

We first arrange it so that we can assume the edges $e_i$ are all positive: For every $i$ where $e_i$ is negative, we perform the insertion move at $v_{i+1}$ to add a (1QP), then glue the negative edge of the (1QP) to $e_i$. We illustrate these operations in \Cref{fig:swapcolor}, where we use a $\copyright$ symbol to denote a crosscap (i.e. one cuts out each $\copyright$ and glues in a Mobius band). The effect on $\Phi$ can be described as splitting $e_i$ into two copies and adding a positive edge with the same endpoints as $e_i$. In particular, $\Phi$ remains strongly connected after this procedure. Hence we can apply \Cref{prop:gluingmove} to conclude that we remain in the same almost equivalence class.

\begin{figure}
    \centering
    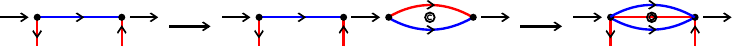
    \caption{Performing an insertion move and a gluing move to arrange it so that we can assume the edges $e_i$ are all positive. Here we use a $\otimes$ symbol to denote a crosscap.}
    \label{fig:swapcolor}
\end{figure}

With this arranged, we perform the insertion move at $v$ twice to get two (1QP)s. See \Cref{fig:bridgemove} first arrow. We then observe the following cutting moves: We start with the flow graph shown in \Cref{fig:simpleequalities} left. If we cut along the cycle formed by the two negative edges, we obtain two (1QP) with a pair of joints identified. See \Cref{fig:simpleequalities} top right. If we cut along the cycle formed by the two positive edges, we obtain (2QS) with the endpoints of a negative edge identified. See \Cref{fig:simpleequalities} bottom right.

\begin{figure}
    \centering
    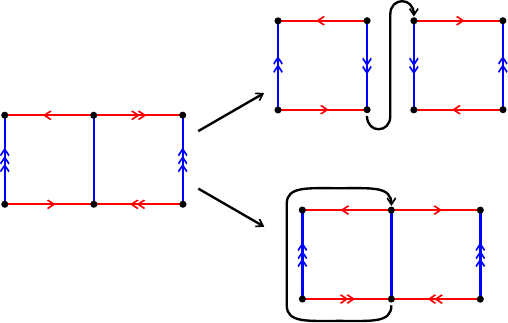
    \caption{The flow graph on the left can be cut along two choices of cycles to produce two (1QP)s or one (2QS).}
    \label{fig:simpleequalities}
\end{figure}

Hence applying \Cref{prop:cuttingmove} twice, we conclude that we can replace $\Phi$ with the graph obtained by separating $v$ and inserting a (2QS), such that
\begin{itemize}
    \item $s^-_1$ is identified with $v^+$,
    \item $s^+_2$ is identified with $v^-$, and
    \item $s^+_1$ is identified with $s^-_2$,
\end{itemize}
and remain in the same almost equivalence class. See \Cref{fig:bridgemove} second arrow.
Here we use the fact that the graph obtained by separating $v$ and inserting the flow graph in \Cref{fig:simpleequalities} left is strongly connected.

The rest of the operation involves sliding the edge $s_2$ towards $w$: We glue $s_2$ to $e_1$. See \Cref{fig:bridgemove} third arrow. The total flow graph at this stage can be obtained from the initial one by replacing $v$ by a positive edge, splitting $e_1$ into two, and adding two negative edges. In particular the total flow graph is strongly connected after gluing $s_2$ to $e_1$, hence we can apply \Cref{prop:gluingmove} to see that we stay in the same almost equivalence class. We then cut along the cycle formed by the gluing move but cooriented in the opposite way. See \Cref{fig:bridgemove} fourth arrow. This slides $s_2$ one step closer to $w$, and we repeat this procedure until $s^-_2$ is identified with $w^+$ and $s^+_2$ is identified with $w^-$. See \Cref{fig:bridgemove} last frame.

The individual flow graphs at this point are those of the initial total flow graph, except with a (2QS) added and certain components glued with (1QP)s. We undo the gluing with (1QP)s by cutting along the cycles produced by the gluings. The total flow graph will then be the target $\overline{\Phi}$ with certain (1QP)s inserted. By \Cref{prop:insertionmove} and the observation that $\overline{\Phi}$ is strongly connected, we can remove those extra (1QP)s and remain in the same almost equivalence class.

\begin{figure}
    \centering
    \fontsize{6pt}{6pt}\selectfont
    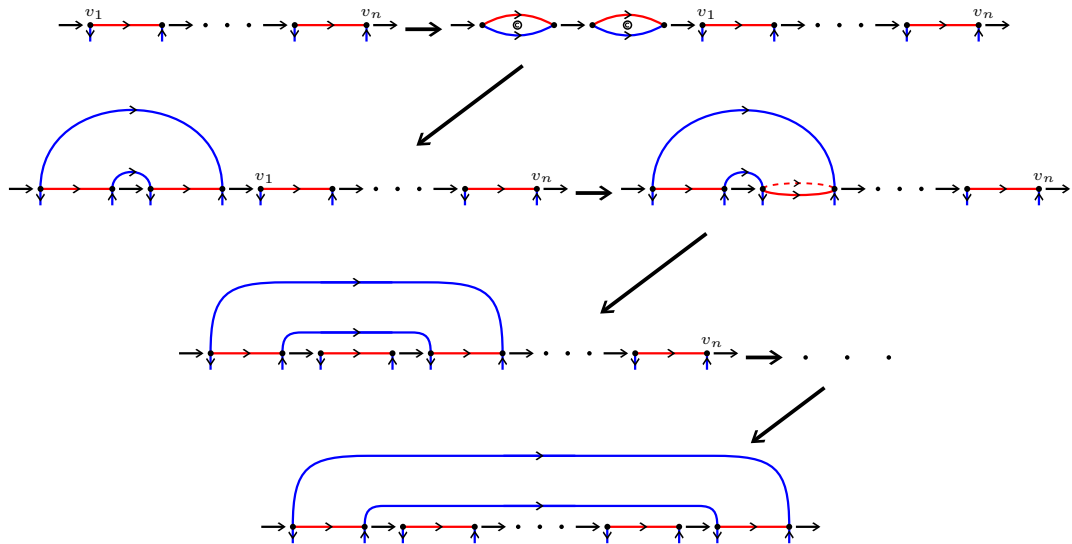
    \caption{Performing the bridge move at vertices $v=v_1$ and $w=v_n$.}
    \label{fig:bridgemove}
\end{figure}

We refer to the operation of obtaining $\overline{\Phi}$ from $\Phi$ as performing a \textbf{bridge move on vertices $v$ and $w$}. The symmetric version of this construction where we insert the (2QS) along its negative edges also holds. For our applications, the context will make it clear which version of the construction we wish to apply.

We record this discussion as the following proposition.

\begin{prop} \label{prop:bridgemove}
Let $\phi^t$ be a totally periodic Anosov flow with orientable stable and unstable foliations. Let $\Phi$ be the total flow graph of $\phi^t$. Let $v$ and $w$ be two vertices of $\Phi$. Suppose $\Phi$ is strongly connected. Let $\overline{\Phi}$ be obtained from $\Phi$ by performing a bridge move on $v$ and $w$. Then $\phi^t_{\Phi} \sim \phi^t_{\overline{\Phi}}$.
\end{prop}

Using the bridge move, we can extend the gluing move to glue any pair of edges whose forward and backward endpoints do not coincide, at the cost of adding (2QS)s. 

To set this up more precisely, suppose $e_1$ and $e_2$ are positive edges of $\Phi$ so that $e^+_1 \neq e^+_2$ and $e^-_1 \neq e^-_2$. We let $\Phi_i$ be the individual flow graph that contains $e_i$. Here we might have $\Phi_1=\Phi_2$. 
We apply the bridge move to the vertices $(e^+_1, e^-_2)$ and to the vertices $(e^-_1,e^+_2)$ so that $e_1$ and $e_2$ are part of an embedded directed cycle formed by $e_1,e_2$, and one positive side from each inserted (2QS). We can now apply \Cref{prop:transportmove} to slide $e_1$ towards $e_2$ so that these edges are adjacent, then apply \Cref{prop:gluingmove} to glue $e_1$ to $e_2$. As remarked in \Cref{subsec:gluingmove}, there are two choices for this gluing.

We illustrate one example in \Cref{fig:distantgluingmove}.

\begin{figure}
    \centering
    \fontsize{6pt}{6pt}\selectfont
    \resizebox{!}{8cm}{%% Creator: Inkscape 1.3 (0e150ed6c4, 2023-07-21), www.inkscape.org
%% Accompanies image file 'distantgluingmove.pdf' (pdf, eps, ps)
\begingroup%
  \makeatletter%
  \providecommand\color[2][]{%
    \errmessage{(Inkscape) Color is used for the text in Inkscape, but the package 'color.sty' is not loaded}%
    \renewcommand\color[2][]{}%
  }%
  \providecommand\transparent[1]{%
    \errmessage{(Inkscape) Transparency is used (non-zero) for the text in Inkscape, but the package 'transparent.sty' is not loaded}%
    \renewcommand\transparent[1]{}%
  }%
  \providecommand\rotatebox[2]{#2}%
  \newcommand*\fsize{\dimexpr\f@size pt\relax}%
  \newcommand*\lineheight[1]{\fontsize{\fsize}{#1\fsize}\selectfont}%
  \ifx\svgwidth\undefined%
    \setlength{\unitlength}{232.96162763bp}%
    \ifx\svgscale\undefined%
      \relax%
    \else%
      \setlength{\unitlength}{\unitlength * \real{\svgscale}}%
    \fi%
  \else%
    \setlength{\unitlength}{\svgwidth}%
  \fi%
  \global\let\svgwidth\undefined%
  \global\let\svgscale\undefined%
  \makeatother%
  \begin{picture}(1,1.07358573)%
    \lineheight{1}%
    \setlength\tabcolsep{0pt}%
    \put(0,0){\includegraphics[width=\unitlength,page=1]{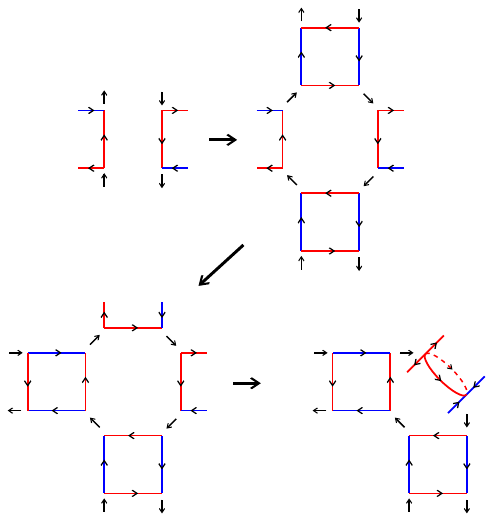}}%
    \put(0.16294736,0.78017014){\color[rgb]{0,0,0}\makebox(0,0)[lt]{\lineheight{1.25}\smash{\begin{tabular}[t]{l}$e_1$\end{tabular}}}}%
    \put(0.3545386,0.78016842){\color[rgb]{0,0,0}\makebox(0,0)[lt]{\lineheight{1.25}\smash{\begin{tabular}[t]{l}$e_2$\end{tabular}}}}%
    \put(0,0){\includegraphics[width=\unitlength,page=2]{distantgluingmove.pdf}}%
  \end{picture}%
\endgroup%
}
    \caption{Modifying the total flow graph by gluing $e_1$ to $e_2$.}
    \label{fig:distantgluingmove}
\end{figure}

We refer to this operation as \textbf{gluing $e_1$ to $e_2$}. The effect on the individual flow graphs is as follows: 
\begin{itemize}
    \item Two (2QS)s are added.
    \item All the flow graphs in $\Phi$ except $\Phi_1$ and $\Phi_2$ are unchanged.
    \item $\Phi_1$ is slit along $e_1$, $\Phi_2$ is slit along $e_2$, and the two boundary cycles are glued.
\end{itemize}  

\Cref{prop:bridgemove} and \Cref{prop:gluingmove} give the following proposition.

\begin{prop} \label{prop:distantgluingmove}
Let $\Phi$ be the total flow graph of a transitive totally periodic Anosov flow with orientable stable and unstable foliations. Let $e_1$ and $e_2$ be two distinct edges of $\Phi$ of the same sign so that $e^+_1 \neq e^+_2$ and $e^-_1 \neq e^-_2$. Let $\overline{\Phi}$ be obtained by gluing $e_1$ to $e_2$. Then $\phi^t_\Phi \sim \phi^t_{\overline{\Phi}}$.
\end{prop}

\subsection{Reducing to only (2QS)s}

We begin the proof of \Cref{thm:totallyperiodicgenusone}. Let $\phi^t$ be a transitive totally periodic Anosov flow with orientable stable and unstable foliations. Let $\Phi$ be the total flow graph of $\phi^t$. The first part of the proof involves performing operations on $\Phi$ until its individual flow graphs consist of only (2QS)s. We do this by inducting on 
$$(N := \sum (\text{\# quads in $\overline{S}_i$}), G:= \sum (\text{genus}(\overline{S}_i)-1))$$
where the sums are taken over all flow graphs $\Phi_i$ that are not (2QS)s, with the lexicographic order.

\begin{defn} \label{defn:quadtypes}
Let $Q$ be a quadrilateral in the complement of an individual flow graph $\Phi_i$ in $S_i$. We say that 
\begin{itemize}
    \item $Q$ is \textbf{free} if the two inward joints that lie on $\partial Q$ are distinct and the two outward joints that lie on $\partial Q$ are distinct.
    \item $Q$ is \textbf{outwardly rigid} if the two inward joints that lie on $\partial Q$ are distinct and the two outward joints that lie on $\partial Q$ are the same.
    \item $Q$ is \textbf{inwardly rigid} if the two inward joints that lie on $\partial Q$ are the same and the two outward joints that lie on $\partial Q$ are distinct.
    \item $Q$ is \textbf{totally rigid} if the two inward joints that lie on $\partial Q$ are the same and the two outward joints that lie on $\partial Q$ are the same.
\end{itemize}
\end{defn}

\begin{prop} \label{prop:flowgraphtype}
Let $\Phi_i$ be an individual flow graph in $\Phi$. Then one of the following statements is true:
\begin{enumerate}
    \item $\Phi_i$ is a (1QP).
    \item $\Phi_i$ is a (2QS).
    \item There is an embedded and coorientable cycle of $\Phi_i$ that consists of edges of the same sign.
    \item There is a free quadrilateral in the complement of $\Phi_i$ in $\overline{S}_i$.
    \item There is an outwardly rigid quadrilateral $Q_+$ and an inwardly rigid quadrilateral $Q_-$ in the complement of $\Phi_i$ in $\overline{S}_i$ so that $Q_+$ and $Q_-$ share a side.
    \item All quadrilaterals are outwardly or totally rigid, there is at least one outwardly rigid quadrilateral, and there is only one outward joint but $\geq 2$ inward joints in $\overline{S}_i$. 
    \item All quadrilaterals are inwardly or totally rigid, there is at least one inwardly rigid quadrilateral, and there is only one inward joint but $\geq 2$ outward joints in $\overline{S}_i$.
\end{enumerate}
In case (6) we say that $\Phi_i$ is \textbf{outwardly rigid}. In case (7) we say that $\Phi_i$ is \textbf{inwardly rigid}.
\end{prop}
\begin{proof}
Let $Q$ be a quadrilateral in the complement of $\Phi_i$. We can construct a map of an annulus into $\overline{S}_i$ that is injective in the interior by starting with $Q$ and successively including quadrilaterals that are adjacent across positive edges. We call the image of such a map a \textbf{negative strip} on $\overline{S}_i$. If there is more than one negative strip on $\overline{S}_i$, then a boundary component of a negative strip satisfies (3). Symmetrically, one can define a notion of \textbf{positive strips}, and see that if there is more than one positive strip, then (3) is true. Hence we can assume that there is exactly one positive and one negative strip on $\overline{S}_i$.

We can also assume that none of the quadrilaterals in $\overline{S}_i$ are free, for otherwise (4) is true. 

Consider the unique positive strip on $\overline{S}_i$. If there are both outwardly rigid and inwardly rigid quadrilaterals, then there must be a sequence of adjacent quadrilaterals $Q_1,...,Q_n$ so that $Q_1$ is outwardly rigid, $Q_2,...,Q_{n-1}$ are totally rigid, and $Q_n$ is inwardly rigid. 
If $n=2$ then (5) is true.
If $n \geq 4$, then there are $n-1 \geq 3$ negative sides with the same forward and backward endpoints, namely the negative sides of the totally rigid quadrilaterals $Q_2,...,Q_n$. In that case, there must exist a pair of such edges so that the cycle formed by them is coorientable, i.e. (3) is true.

It remains to analyze the case when $n=3$. Notice that the two positive sides of the totally rigid quadrilateral $Q_2$ are distinct, since we have assumed that there is a unique negative strip. One positive side of $Q_1$ and one positive side of $Q_3$ has the same endpoints as these positive sides. We can assume that this positive side of $Q_1$ coincides with a positive side of $Q_2$ and the positive side of $Q_3$ coincides with the other positive side of $Q_2$, for otherwise there are $\geq 3$ positive sides with the same endpoints, which implies (3) is true as above. There are two cases here as indicated in \Cref{fig:posrigidnegcontradiction}. But either cases cannot occur, for otherwise one endpoint of these positive sides only meet $2$ quadrilaterals each, contradicting the fact that $Q_1,Q_2,Q_3$ are distinct.

\begin{figure}
    \centering
    \fontsize{10pt}{10pt}\selectfont
    %% Creator: Inkscape 1.3.2 (091e20e, 2023-11-25, custom), www.inkscape.org
%% Accompanies image file 'posrigidnegcontradiction.pdf' (pdf, eps, ps)
\begingroup%
  \makeatletter%
  \providecommand\color[2][]{%
    \errmessage{(Inkscape) Color is used for the text in Inkscape, but the package 'color.sty' is not loaded}%
    \renewcommand\color[2][]{}%
  }%
  \providecommand\transparent[1]{%
    \errmessage{(Inkscape) Transparency is used (non-zero) for the text in Inkscape, but the package 'transparent.sty' is not loaded}%
    \renewcommand\transparent[1]{}%
  }%
  \providecommand\rotatebox[2]{#2}%
  \newcommand*\fsize{\dimexpr\f@size pt\relax}%
  \newcommand*\lineheight[1]{\fontsize{\fsize}{#1\fsize}\selectfont}%
  \ifx\svgwidth\undefined%
    \setlength{\unitlength}{237.64878881bp}%
    \ifx\svgscale\undefined%
      \relax%
    \else%
      \setlength{\unitlength}{\unitlength * \real{\svgscale}}%
    \fi%
  \else%
    \setlength{\unitlength}{\svgwidth}%
  \fi%
  \global\let\svgwidth\undefined%
  \global\let\svgscale\undefined%
  \makeatother%
  \begin{picture}(1,0.43560359)%
    \lineheight{1}%
    \setlength\tabcolsep{0pt}%
    \put(0,0){\includegraphics[width=\unitlength,page=1]{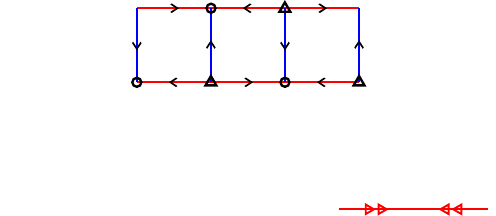}}%
    \put(0.32504707,0.32960098){\color[rgb]{0,0,0}\makebox(0,0)[lt]{\lineheight{1.25}\smash{\begin{tabular}[t]{l}$Q_1$\end{tabular}}}}%
    \put(0.47467144,0.32960098){\color[rgb]{0,0,0}\makebox(0,0)[lt]{\lineheight{1.25}\smash{\begin{tabular}[t]{l}$Q_2$\end{tabular}}}}%
    \put(0.62429582,0.32960098){\color[rgb]{0,0,0}\makebox(0,0)[lt]{\lineheight{1.25}\smash{\begin{tabular}[t]{l}$Q_3$\end{tabular}}}}%
    \put(0,0){\includegraphics[width=\unitlength,page=2]{posrigidnegcontradiction.pdf}}%
  \end{picture}%
\endgroup%

    \caption{Assuming that there is a unique positive strip and a unique negative strip, if we have adjacent quadrilaterals $Q_1,Q_2,Q_3$ where $Q_1$ is outwardly rigid, $Q_2$ is totally rigid, $Q_3$ is inwardly rigid, then we will have $\geq 3$ edges with the same endpoints, giving (3), for otherwise we will be in the bottom two cases, both cases leading to contradictions.}
    \label{fig:posrigidnegcontradiction}
\end{figure}

Thus we can assume that all quadrilaterals are outwardly rigid or totally rigid, or all quadrilaterals are inwardly rigid or totally rigid. Without loss of generality suppose we are in the former case.

In this case there is only one outward joint, so it suffices to show that there are $\geq 2$ inward joints. Suppose otherwise that there is only one inward joint. Consider the number of quadrilaterals.
If $\Phi_i$ only has one quadrilateral, then it is a (1QP). If $\Phi_i$ only has two quadrilaterals, then from our assumptions, either $\Phi_i$ is a (2QS) or $\Phi_i$ is the flow graph shown in \Cref{fig:2qt}. In the latter case, the cycle formed by the positive edges satisfies (3).
If $\Phi_i$ has $\geq 3$ quadrilaterals, then there are $\geq 3$ positive edges with the same forward and backward endpoints.
In that case, (3) is true as above.
\end{proof}

\begin{figure}
    \centering
    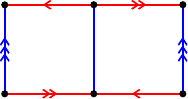
    \caption{If $\Phi_i$ is of this form, then the cycle formed by the positive edges satisfies (3).}
    \label{fig:2qt}
\end{figure}

Suppose $N > 0$. Then one of (1),(3)-(8) in \Cref{prop:flowgraphtype} is true for some individual flow graph. If (1) is true for some flow graph, we perform the reverse of the insertion move to remove this individual flow graph from $\Phi$. As remarked below \Cref{prop:insertionmove}, the total flow graph remains strongly connected after this removal, hence \Cref{prop:insertionmove} applies to show that we remain in the same almost equivalence class. $N$ decreases by $1$ for every (1QP) removed. 

If (3) is true, then we cut along that cycle. \Cref{prop:cuttingmove} implies that we stay in the same almost equivalence class. $N$ stays constant or decreases after this operation, while $G$ decreases.

If (4) is true, we apply \Cref{prop:distantgluingmove} to glue the two positive sides of the free quadrilateral $Q$. We perform the gluing so that the cycle formed by the two negative side of $Q$ form a separating embedded curve $c$. We then cut along $c$. A computation shows that this creates a (1QP). See \Cref{fig:freequadarg}. On the level of the closed surface $\overline{S}_i$, we are attaching a non-orientable $1$-handle, which is equivalent to taking a connect sum with a Klein bottle $\cong \mathbb{R}P^2 \# \mathbb{R}P^2$ and detaching one of the $\mathbb{R}P^2$. We can then remove this (1QP) as in case (1), decreasing $N$ by $1$. 

\begin{figure}
    \centering
    %% Creator: Inkscape 1.3 (0e150ed6c4, 2023-07-21), www.inkscape.org
%% Accompanies image file 'freequadarg.pdf' (pdf, eps, ps)
\begingroup%
  \makeatletter%
  \providecommand\color[2][]{%
    \errmessage{(Inkscape) Color is used for the text in Inkscape, but the package 'color.sty' is not loaded}%
    \renewcommand\color[2][]{}%
  }%
  \providecommand\transparent[1]{%
    \errmessage{(Inkscape) Transparency is used (non-zero) for the text in Inkscape, but the package 'transparent.sty' is not loaded}%
    \renewcommand\transparent[1]{}%
  }%
  \providecommand\rotatebox[2]{#2}%
  \newcommand*\fsize{\dimexpr\f@size pt\relax}%
  \newcommand*\lineheight[1]{\fontsize{\fsize}{#1\fsize}\selectfont}%
  \ifx\svgwidth\undefined%
    \setlength{\unitlength}{302.76981624bp}%
    \ifx\svgscale\undefined%
      \relax%
    \else%
      \setlength{\unitlength}{\unitlength * \real{\svgscale}}%
    \fi%
  \else%
    \setlength{\unitlength}{\svgwidth}%
  \fi%
  \global\let\svgwidth\undefined%
  \global\let\svgscale\undefined%
  \makeatother%
  \begin{picture}(1,0.2494018)%
    \lineheight{1}%
    \setlength\tabcolsep{0pt}%
    \put(0,0){\includegraphics[width=\unitlength,page=1]{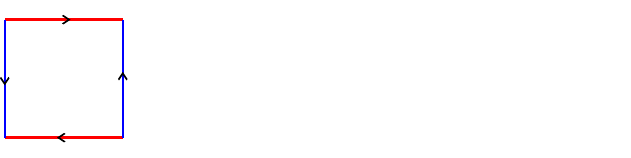}}%
    \put(0.08491694,0.11470345){\color[rgb]{0,0,0}\makebox(0,0)[lt]{\lineheight{1.25}\smash{\begin{tabular}[t]{l}$Q$\end{tabular}}}}%
    \put(0,0){\includegraphics[width=\unitlength,page=2]{freequadarg.pdf}}%
  \end{picture}%
\endgroup%

    \caption{If we have a free quadrilateral $Q$, we apply this sequence of operations to produce a (1QP).}
    \label{fig:freequadarg}
\end{figure}

If (5) is true, without loss of generality suppose that $Q_+$ and $Q_-$ share a positive side. Notice that the positive side of $Q_+$ not shared with $Q_-$ and the positive side of $Q_-$ not shared with $Q_+$ have distinct endpoints, for otherwise $Q_+$ or $Q_-$ is totally rigid. We apply \Cref{prop:distantgluingmove} to glue these two sides. We perform the gluing so that the negative sides of the two quadrilaterals form a pair of embedded cycles $c_1$ and $c_2$, so that each of $c_i$ is non-separating, but $c_1 \cup c_2$ is separating. We cut along $c_1$ then $c_2$. A computation shows that this creates a (2QS). See \Cref{fig:posnegquadarg}. On the level of the closed surface $\overline{S}_i$, we are attaching an orientable $1$-handle, then attaching two copies of the dual $2$-handle to separate out a sphere. This decreases $N$ by $2$. 

\begin{figure}
    \centering
    \fontsize{10pt}{10pt}\selectfont
    %% Creator: Inkscape 1.3 (0e150ed6c4, 2023-07-21), www.inkscape.org
%% Accompanies image file 'posnegquadarg.pdf' (pdf, eps, ps)
\begingroup%
  \makeatletter%
  \providecommand\color[2][]{%
    \errmessage{(Inkscape) Color is used for the text in Inkscape, but the package 'color.sty' is not loaded}%
    \renewcommand\color[2][]{}%
  }%
  \providecommand\transparent[1]{%
    \errmessage{(Inkscape) Transparency is used (non-zero) for the text in Inkscape, but the package 'transparent.sty' is not loaded}%
    \renewcommand\transparent[1]{}%
  }%
  \providecommand\rotatebox[2]{#2}%
  \newcommand*\fsize{\dimexpr\f@size pt\relax}%
  \newcommand*\lineheight[1]{\fontsize{\fsize}{#1\fsize}\selectfont}%
  \ifx\svgwidth\undefined%
    \setlength{\unitlength}{313.75028355bp}%
    \ifx\svgscale\undefined%
      \relax%
    \else%
      \setlength{\unitlength}{\unitlength * \real{\svgscale}}%
    \fi%
  \else%
    \setlength{\unitlength}{\svgwidth}%
  \fi%
  \global\let\svgwidth\undefined%
  \global\let\svgscale\undefined%
  \makeatother%
  \begin{picture}(1,0.30220857)%
    \lineheight{1}%
    \setlength\tabcolsep{0pt}%
    \put(0,0){\includegraphics[width=\unitlength,page=1]{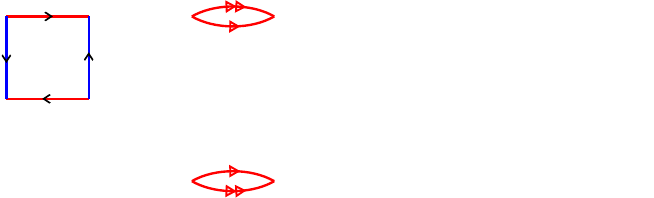}}%
    \put(0.05372866,0.20731622){\color[rgb]{0,0,0}\makebox(0,0)[lt]{\lineheight{1.25}\smash{\begin{tabular}[t]{l}$Q_+$\end{tabular}}}}%
    \put(0,0){\includegraphics[width=\unitlength,page=2]{posnegquadarg.pdf}}%
    \put(0.05372866,0.08130787){\color[rgb]{0,0,0}\makebox(0,0)[lt]{\lineheight{1.25}\smash{\begin{tabular}[t]{l}$Q_-$\end{tabular}}}}%
    \put(0,0){\includegraphics[width=\unitlength,page=3]{posnegquadarg.pdf}}%
  \end{picture}%
\endgroup%

    \caption{If (5) is true, then we apply this sequence of operations to produce a (2QS).}
    \label{fig:posnegquadarg}
\end{figure}

With the analysis above, we can assume that each flow graph is either a (2QS) or outwardly rigid or inwardly rigid. But recall that the total number of inward joints and outward joints among all base surfaces are equal. This implies that under our assumptions, there is an outwardly rigid flow graph if and only if there is an inwardly rigid flow graph. In this case, we apply \Cref{prop:distantgluingmove} to glue a positive side of an outwardly rigid quadrilateral in an outwardly rigid flow graph to a positive side of an inwardly rigid quadrilateral in an inwardly rigid flow graph. This returns us to case (5), where we showed how to decrease $N$ by $2$.

By induction, we eventually arrange it so that every individual flow graph is a (2QS).

\subsection{Simplifying (2QS) graphs and concluding}

For the next part of the proof, we stay within the category of total flow graphs where every individual flow graph is a (2QS). For convenience, let us refer to such total flow graphs as \textbf{(2QS) graphs}.

Recall that a \textbf{fat graph} is a graph with the data of a cyclic ordering of half edges at each vertex. Let $\Phi$ be a (2QS) graph. We define a fat graph $G$ from $\Phi$ as follows: The vertices of $G$ are the directed cycles of positive edges of $\Phi$. We place an edge between two vertices for every (2QS) with its positive edges belonging to the two corresponding cycles. The ordering of half-edges is determined by the cyclic ordering of edges in each cycle. See \Cref{fig:2qsfatgraph} for an example. 
Conversely, given a fat graph $G$, we can define a (2QS) graph by replacing each edge by a (2QS) and identifying the incoming and outgoing joints of the positive edges according to the cyclic ordering at each vertex. 

\begin{figure}
    \centering
    \resizebox{!}{4.5cm}{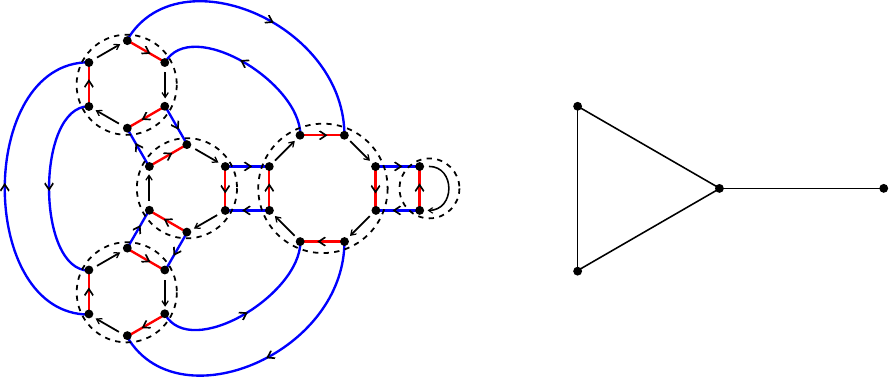}
    \caption{There is a correspondence between (2QS) graphs $\Phi$ and fat graphs $G$.}
    \label{fig:2qsfatgraph}
\end{figure}
This defines a correspondence between (2QS) graphs $\Phi$ and fat graphs $G$. Notice that under this correspondence, $\Phi$ is strongly connected if and only if $G$ is connected. 

Suppose $\Phi'$ is obtained by performing a bridge move on a (2QS) graph $\Phi$. Then the corresponding fat graph $G'$ is obtained by adding an edge to $G$. Conversely, if we are given a (2QS) graph $\Phi$, then as long as $G$ does not consist of only one edge, we can take a non-separating edge $e$ of $G$ and define a fat graph $G'$ by removing $e$. Let $\Phi'$ be the corresponding (2QS) graph. Then $\phi^t_{\Phi} \sim \phi^t_{\Phi'}$ by \Cref{prop:bridgemove}. 

Eventually, we reduce to the (2QS) graph whose corresponding fat graph only has one edge. That is, we get to the total flow graph that is a single (2QS) with its joints identified. Using \Cref{fig:simpleequalities}, we can transform that into the total flow graph that consists of two (1QP)s with joints identified. We apply the reverse of the insertion move to remove one of the (1QP)s, so that we are left with the total flow graph that is a single (1QP) with its joints identified.

We claim that the corresponding totally periodic Anosov flow is almost equivalent to a suspension Anosov flow. Once we show this, then \Cref{thm:totallyperiodicgenusone} follows from \Cref{prop:genusonesection=sus}.
We note that this claim is also observed in unpublished work of Sergio Fenley, Jessica Purcell, and Mario Shannon that we alluded to before. Once again, we thank them for showing us their work.

The argument for this is very similar to the insertion move. Let $\phi^t_A$ be some suspension Anosov flow with orientable stable and unstable foliations, defined on mapping torus $T_A$. Let $T$ be a fiber of $T_A$. By \cite[Example 3.7]{Tsa24}, there is a positive horizontal surgery curve $b$ lying on $T$. As explained in \cite[Example 3.15]{Tsa24}, we can construct another positive horizontal curve $\beta$ by inserting the $2$-strand closed braid with a single negative crossing along $b$. We perform horizontal Goodman surgery along $\beta$ with coefficient $1$. \Cref{lemma:insertiontopology} implies that the surgered $3$-manifold $(T_A)_1(\beta)$ is a graph manifold consisting of one Seifert fibered piece that is $(\mathbb{R}P^2 \backslash \text{two discs}) \times S^1$. Moreover the fiber of this Seifert fibered piece is homotopic to a closed orbit. Hence the surgered flow is totally periodic.

As in \Cref{subsec:insertionmove}, the glued flow graph of this totally periodic Anosov flow can only be a single (1QP). Thus \Cref{lemma:indgluingmap} concludes the proof of \Cref{thm:totallyperiodicgenusone}.

\bibliographystyle{alpha}

\bibliography{bib.bib}

\end{document}